\documentclass[a4paper,draft]{amsart}

\usepackage{latexsym}\usepackage{ifthen}\usepackage[leqno]{amsmath}\usepackage{enumerate}\usepackage{calc}
\usepackage{hyphenat}\usepackage{amstext,amsbsy,amsopn,amsthm,amsgen,amsfonts,amscd,amsxtra,upref}
\usepackage{mathrsfs}\usepackage{euscript}\usepackage{amssymb}

\swapnumbers
\theoremstyle{plain}
\newtheorem{Thm}{Theorem}[section]
\newtheorem{Lem}[Thm]{Lemma}
\newtheorem{Cor}[Thm]{Corollary}
\newtheorem{Pro}[Thm]{Proposition}
\newtheorem{Prp}[Thm]{Properties}
\newtheorem{Sub}[Thm]{Sublemma}

\theoremstyle{definition}
\newtheorem{Def}[Thm]{Definition}
\newtheorem{Exm}[Thm]{Example}
\newtheorem{Exs}[Thm]{Examples}
\newtheorem{Prb}[Thm]{Problem}

\theoremstyle{remark}
\newtheorem{Rem}[Thm]{Remark}
\newtheorem{Rms}[Thm]{Remarks}
\newtheorem*{Com}{Commentary}

\newcommand{\myEmail}{piotr.niemiec@uj.edu.pl}
\newcommand{\myAddress}[1]{\noindent{}\ITE{\equal{#1}{}}{}{Piotr Niemiec\\{}}
   In\-sty\-tut Ma\-te\-ma\-ty\-ki\\{}Wy\-dzia\l{} Ma\-te\-ma\-ty\-ki i In\-for\-ma\-ty\-ki\\{}
   U\-ni\-wer\-sy\-tet Ja\-giel\-lo\'{n}\-ski\\{}ul. \L{}o\-ja\-sie\-wi\-cza 6\\{}
   30-348 Kra\-k\'{o}w\\{}Poland}
\newcommand{\myData}[1][Piotr Niemiec]{\author[P. Niemiec]{Piotr Niemiec}\address{\myAddress{#1}}
   \email{\myEmail}}

\newcommand{\NNN}{\mathbb{N}}
\newcommand{\QQQ}{\mathbb{Q}}\newcommand{\RRR}{\mathbb{R}}

\newcommand{\ZZZ}{\mathbb{Z}}
\newcommand{\CCc}{\CMcal{C}}
\newcommand{\EEe}{\CMcal{E}}\newcommand{\FFf}{\CMcal{F}}
\newcommand{\GGg}{\CMcal{G}}\newcommand{\HHh}{\CMcal{H}}

\newcommand{\OOo}{\CMcal{O}}

\newcommand{\WWw}{\CMcal{W}}

\newcommand{\SsS}{\EuScript{S}}

\newcommand{\mM}{\mathfrak{m}}

\newcommand{\bbB}{\mathscr{B}}

\newcommand{\ggG}{\mathscr{G}}

\newcommand{\uuU}{\mathscr{U}}

\newcommand{\SECT}[1]{\section{#1}\renewcommand{\theequation}{\arabic{section}-\arabic{equation}}
   \setcounter{equation}{0}}

\newcounter{help}
\newcommand{\ITE}[3]{\ifthenelse{#1}{#2}{#3}}\newcommand{\ITEE}[3]{\ITE{\equal{#1}{#2}}{#3}{}}

\newcommand{\card}{\operatorname{card}}

\newcommand{\id}{\operatorname{id}}

\newcommand{\Iso}{\operatorname{Iso}}

\newcommand{\leqsl}{\leqslant}\newcommand{\geqsl}{\geqslant}
\newcommand{\epsi}{\varepsilon}\newcommand{\varempty}{\varnothing}\newcommand{\dd}{\colon}

\newcommand{\TFCAE}{The following conditions are equivalent:}

\newcommand{\COR}[1]{Corollary~\textup{\ref{cor:#1}}}

\newcommand{\EXM}[1]{Example~\textup{\ref{exm:#1}}}

\newcommand{\LEM}[1]{Lemma~\textup{\ref{lem:#1}}}

\newcommand{\PRO}[1]{Proposition~\textup{\ref{pro:#1}}}

\newcommand{\THM}[1]{Theorem~\textup{\ref{thm:#1}}}

\newenvironment{cor}[1]{\begin{Cor}\label{cor:#1}}{\end{Cor}}
\newenvironment{dfn}[1]{\begin{Def}\label{def:#1}}{\end{Def}}
\newenvironment{exm}[1]{\begin{Exm}\label{exm:#1}}{\end{Exm}}

\newenvironment{lem}[1]{\begin{Lem}\label{lem:#1}}{\end{Lem}}
\newenvironment{prb}[1]{\begin{Prb}\label{prb:#1}}{\end{Prb}}
\newenvironment{pro}[1]{\begin{Pro}\label{pro:#1}}{\end{Pro}}

\newenvironment{rem}[1]{\begin{Rem}\label{rem:#1}}{\end{Rem}}

\newenvironment{thm}[1]{\begin{Thm}\label{thm:#1}}{\end{Thm}}

\newcommand{\ArEe}{\operatorname{AE}}\newcommand{\cIRC}[1]{\textup{\textcircled{$#1$}}} 
\newcommand{\IGH}{\textup{\textsf{IGH}}} 

\newcommand{\bibITEM}[2]{\ITE{\equal{#2}{}}{\bibitem{#1} }{\bibitem[#2]{#1} }}
\newcommand{\BIB}[8]{
   \bibITEM{#1}{#8} #2, \textit{#3}, #4{} \textbf{#5} (#6), #7.}
\newcommand{\myBIB}[7][P. Niemiec]{\ITE{\equal{#7}{*}\or\equal{#7}{**}}{}{#1, \textit{#2}, }
   #3{}\ITE{\equal{#4}{}}{}{ \textbf{#4}} (#5), #6\ITE{\equal{#7}{*}}{}{.}}
\newcommand{\BIb}[6]{
   \bibITEM{#1}{#6} #2, \textit{#3}, #4, #5.}
\newcommand{\BiB}[9]{
   \bibITEM{#1}{#9} #2, \textit{#3}, #4{} \textit{#5}, #6, #7, #8.}

\newcommand{\myBAPP}[4][P. Niemiec]{
   \ITE{\equal{#4}{*}}{}{#1, \textit{#2}, }#3}
\newcommand{\oNlINE}[2]{\ITEE{#1}{.}{\\#2}\ITEE{#1}{}{{} #2}}

\newcommand{\jRN}[2][]{
   \ITEE{#2}{AbhHamburg}{\ITE{\equal{#1}{+}}
      {Abh. Math. Sem. Hamburg}{Abh. Math. Sem. Hamburg}}
   \ITEE{#2}{ActaM}{\ITE{\equal{#1}{+}}
      {Acta Mathematica}{Acta Math.}}
   \ITEE{#2}{ActaMSinES}{\ITE{\equal{#1}{+}}
      {Acta Mathematica Sinica (English Series)}{Acta Math. Sin. (Engl. Ser.)}}
   \ITEE{#2}{AdvM}{\ITE{\equal{#1}{+}}
      {Advances in Mathematics}{Adv. Math.}}
   \ITEE{#2}{ACS}{\ITE{\equal{#1}{+}}
      {Applied Categorical Structures}{Appl. Categ. Structures}}
   \ITEE{#2}{ActaSM}{\ITE{\equal{#1}{+}}
      {Acta Scientiarum Mathematicarum (Szeged)}{Acta Sci. Math. (Szeged)}}
   \ITEE{#2}{AmJM}{\ITE{\equal{#1}{+}}
      {American Journal of Mathematics}{Amer. J. Math.}}
   \ITEE{#2}{AmMMon}{\ITE{\equal{#1}{+}}
      {The American Mathematical Monthly}{Amer. Math. Monthly}}
   \ITEE{#2}{AnnSciEcNormSupT}{\ITE{\equal{#1}{+}}
      {Annales Scientifiques de l'\'{E}cole Normale Sup\'{e}rieure (3)}
      {Ann. Sci. \'{E}c. Norm. Sup\'{e}r. (3)}}
   \ITEE{#2}{AnnM}{\ITE{\equal{#1}{+}}
      {Annals of Mathematics}{Ann. Math.}}
   \ITEE{#2}{AnnProb}{\ITE{\equal{#1}{+}}
      {The Annals of Probability}{Ann. Probab.}}
   \ITEE{#2}{AnnPALog}{\ITE{\equal{#1}{+}}
      {Annals of Pure and Applied Logic}{Ann. Pure Appl. Logic}}
   \ITEE{#2}{APM}{\ITE{\equal{#1}{+}}
      {Annales Polonici Mathematici}{Ann. Polon. Math.}}
   \ITEE{#2}{ArchM}{\ITE{\equal{#1}{+}}
      {Archiv der Mathematik}{Arch. Math.}}
   \ITEE{#2}{AttiAccLincRendNat}{\ITE{\equal{#1}{+}}
      {Atti della Accademia Nazionale dei Lincei. Rendiconti. Classe di Scienze Fisiche, 
      Matematiche e Naturali}{Atti Accad. Naz. Lincei Rend. Cl. Sci. Fis. Mat. Nat.}}
   \ITEE{#2}{BAMS}{\ITE{\equal{#1}{+}}
      {Bulletin of the American Mathematical Society}{Bull. Amer. Math. Soc.}}
   \ITEE{#2}{BAustrMS}{\ITE{\equal{#1}{+}}
      {Bulletin of the Australian Mathematical Society}{Bull. Austral. Math. Soc.}}
   \ITEE{#2}{BLondMS}{\ITE{\equal{#1}{+}}
      {Bulletin of the London Mathematical Sociecy}{Bull. Lond. Math. Soc.}}
   \ITEE{#2}{BAPolSSSM}{\ITE{\equal{#1}{+}}
      {Bulletin de l'Acad\'{e}mie Polonaise des Sciences. S\'{e}rie des Sciences 
      Math\'{e}matiques}{Bull. Acad. Pol. Sci. S\'{e}r. Sci. Math.}}
   \ITEE{#2}{BullSM}{\ITE{\equal{#1}{+}}
      {Bulletin des Sciences Math\'{e}matiques}{Bull. Sci. Math.}}
   \ITEE{#2}{BPAS}{\ITE{\equal{#1}{+}}
      {Bulletin of the Polish Academy of Sciences: Mathematics}{Bull. Pol. Acad. Sci. Math.}}
   \ITEE{#2}{CanadJM}{\ITE{\equal{#1}{+}}
      {Canadian Journal Mathematics}{Canad. J. Math.}}
   \ITEE{#2}{CollectM}{\ITE{\equal{#1}{+}}
      {Collectanea Mathematica}{Collect. Math.}}
   \ITEE{#2}{CMUC}{\ITE{\equal{#1}{+}}
      {Commentationes Mathematicae Universitatis Carolinae}{Comment. Math. Univ. Carolin.}}
   \ITEE{#2}{CRParis}{\ITE{\equal{#1}{+}}
      {C. R. Paris}{C. R. Paris}}
   \ITEE{#2}{CRASParis}{\ITE{\equal{#1}{+}}
      {Comptes Rendus de l'Acad\'{e}mie des Sciences. Paris}{C. R. Acad. Sci. Paris}}
   \ITEE{#2}{CEurJM}{\ITE{\equal{#1}{+}}
      {Central European Journal of Mathematics}{Cent. Eur. J. Math.}}
   \ITEE{#2}{CMHelv}{\ITE{\equal{#1}{+}}
      {Commentarii Mathematici Helvetici}{Comment. Math. Helv.}}
   \ITEE{#2}{CollM}{\ITE{\equal{#1}{+}}
      {Colloquium Mathematicum}{Coll. Math.}}
   \ITEE{#2}{CollMSJBoly}{\ITE{\equal{#1}{+}}
      {Colloquia Math. Soc. Janos Bolyai}{Colloq. Math. Soc. Janos Bolyai}}
   \ITEE{#2}{ComposM}{\ITE{\equal{#1}{+}}
      {Compositio Mathematica}{Compos. Math.}}
   \ITEE{#2}{CzMJ}{\ITE{\equal{#1}{+}}
      {Czechoslovak Mathematical Journal}{Czech. Math. J.}}
   \ITEE{#2}{DissM}{\ITE{\equal{#1}{+}}
      {Dissertationes Mathematicae (Roz\-pra\-wy Ma\-te\-ma\-tycz\-ne)}
      {Dissertationes Math. (Roz\-pra\-wy Mat.)}}
   \ITEE{#2}{DANSSSR}{\ITE{\equal{#1}{+}}
      {Doklady Akademii Nauk SSSR}{Dokl. Akad. Nauk SSSR}}
   \ITEE{#2}{DMJ}{\ITE{\equal{#1}{+}}
      {Duke Mathematical Journal}{Duke Math. J.}}
   \ITEE{#2}{ELA}{\ITE{\equal{#1}{+}}
      {The Electronic Journal of Linear Algebra}{Electron. J. Linear Algebra}}
   \ITEE{#2}{ExtrM}{\ITE{\equal{#1}{+}}
      {Extracta Mathematicae}{Extracta Math.}}
   \ITEE{#2}{FM}{\ITE{\equal{#1}{+}}
      {Fundamenta Mathematicae}{Fund. Math.}}
   \ITEE{#2}{FAA}{\ITE{\equal{#1}{+}}
      {Functional Analysis and its Applications}{Funct. Anal. Appl.}}
   \ITEE{#2}{FunkAnalPril}{\ITE{\equal{#1}{+}}
      {Funktsional'ny\u{\i} Analiz i Ego Prilozheniya}{Funkts. Anal. Prilozh.}}
   \ITEE{#2}{GTopA}{\ITE{\equal{#1}{+}}
      {General Topology and its Applications}{General Topol. Appl.}}
   \ITEE{#2}{HJM}{\ITE{\equal{#1}{+}}
      {Houston Journal of Mathematics}{Houston J. Math.}}
   \ITEE{#2}{IllinoisJM}{\ITE{\equal{#1}{+}}
      {Illinois Journal of Mathematics}{Illinois J. Math.}}
   \ITEE{#2}{IndagMP}{\ITE{\equal{#1}{+}}
      {Indagationes Mathematicae (Proceedings)}{Indagationes Math. Proc.}}
   \ITEE{#2}{IndianaUMJ}{\ITE{\equal{#1}{+}}
      {Indiana University Mathematical Journal}{Indiana Univ. Math. J.}}
   \ITEE{#2}{InHauEtSPM}{\ITE{\equal{#1}{+}}
      {Inst. Hautes \'{E}tudes Sci. Publ. Math.}{Inst. Hautes \'{E}tudes Sci. Publ. Math.}}
   \ITEE{#2}{IEOT}{\ITE{\equal{#1}{+}}
      {Integral Equations and Operator Theory}{Integral Equations Operator Theory}}
   \ITEE{#2}{InterJM}{\ITE{\equal{#1}{+}}
      {International Journal of Mathematics}{Internat. J. Math.}}
   \ITEE{#2}{IsraelJM}{\ITE{\equal{#1}{+}}
      {Israel Journal of Mathematics}{Israel J. Math.}}
   \ITEE{#2}{JAT}{\ITE{\equal{#1}{+}}
      {Journal of Approximation Theory}{J. Approx. Theory}}
   \ITEE{#2}{JAusMSA}{\ITE{\equal{#1}{+}}
      {Journal of the Australian Mathematical Society. Series A}{J. Aust. Math. Soc. Ser. A}}
   \ITEE{#2}{JCA}{\ITE{\equal{#1}{+}}
      {Journal of Convex Analysis}{J. Convex Anal.}}
   \ITEE{#2}{JChinUST}{\ITE{\equal{#1}{+}}
      {J. China Univ. Sci. Tech.}{J. China Univ. Sci. Tech.}}
   \ITEE{#2}{JFA}{\ITE{\equal{#1}{+}}
      {Journal of Functional Analysis}{J. Funct. Anal.}}
   \ITEE{#2}{JKoreanMS}{\ITE{\equal{#1}{+}}
      {Journal of the Korean Mathematical Society}{J. Korean Math. Soc.}}
   \ITEE{#2}{JLieTh}{\ITE{\equal{#1}{+}}
      {Journal of Lie Theory}{J. Lie Theory}}
   \ITEE{#2}{JMAnApp}{\ITE{\equal{#1}{+}}
      {J. Math. Anal. Appl.}{J. Math. Anal. Appl.}}
   \ITEE{#2}{JLondMS}{\ITE{\equal{#1}{+}}
      {Journal of the London Mathematical Society}{J. London Math. Soc.}}
   \ITEE{#2}{JMPuApNS}{\ITE{\equal{#1}{+}}
      {J. Math. Pures Appl., N. S.}{J. Math. Pures Appl., N. S.}}
   \ITEE{#2}{JOT}{\ITE{\equal{#1}{+}}
      {Journal of Operator Theory}{J. Operator Theory}}
   \ITEE{#2}{JReinAngM}{\ITE{\equal{#1}{+}}
      {Journal f\"{u}r die reine und angewandte Mathematik}{J. Reine Angew. Math.}}
   \ITEE{#2}{KodaiMSemRep}{\ITE{\equal{#1}{+}}
      {Kodai Math. Sem. Rep.}{Kodai Math. Sem. Rep.}}
   \ITEE{#2}{LAA}{\ITE{\equal{#1}{+}}
      {Linear Algebra and its Applications}{Linear Algebra Appl.}}
   \ITEE{#2}{LMLA}{\ITE{\equal{#1}{+}}
      {Linear and Multilinear Algebra}{Linear Multilinear Algebra}}
   \ITEE{#2}{LNM}{\ITE{\equal{#1}{+}}
      {Lecture Notes in Mathematics}{Lecture Notes in Math.}}
   \ITEE{#2}{MathJap}{\ITE{\equal{#1}{+}}
      {Math. Japon.}{Math. Japon.}}
   \ITEE{#2}{MLQ}{\ITE{\equal{#1}{+}}
      {Mathematical Logic Quarterly}{Math. Log. Q.}}
   \ITEE{#2}{MNotes}{\ITE{\equal{#1}{+}}
      {Math. Notes}{Math. Notes}}
   \ITEE{#2}{MProcCambPhS}{\ITE{\equal{#1}{+}}
      {Mathematical Proceedings of the Cambridge Philosophical Society}
      {Math. Proc. Cambridge Phil. Soc.}}
   \ITEE{#2}{MMag}{\ITE{\equal{#1}{+}}
      {Mathematics Magazine}{Math. Mag.}}
   \ITEE{#2}{MSb}{\ITE{\equal{#1}{+}}
      {Matematicheski\u{\i} Sbornik}{Mat. Sb.}}
   \ITEE{#2}{MStud}{\ITE{\equal{#1}{+}}
      {Matematychni Studi\"{\i}}{Mat. Stud.}}
   \ITEE{#2}{MScand}{\ITE{\equal{#1}{+}}
      {Mathematica Scandinavica}{Math. Scand.}}
   \ITEE{#2}{MAnn}{\ITE{\equal{#1}{+}}
      {Mathematische Annalen}{Math. Ann.}}
   \ITEE{#2}{MAMS}{\ITE{\equal{#1}{+}}
      {Memoirs of the American Mathematical Society}{Mem. Amer. Math. Soc.}}
   \ITEE{#2}{MichMJ}{\ITE{\equal{#1}{+}}
      {Michigan Mathematical Journal}{Mich. Math. J.}}
   \ITEE{#2}{MonatM}{\ITE{\equal{#1}{+}}
      {Monatshefte f\"{u}r Mathematik}{Mh. Math.}}
   \ITEE{#2}{MZ}{\ITE{\equal{#1}{+}}
      {Math. Z.}{Math. Z.}}
   \ITEE{#2}{MZamet}{\ITE{\equal{#1}{+}}
      {Mat. Zametki}{Mat. Zametki}}
   \ITEE{#2}{NonlinA}{\ITE{\equal{#1}{+}}
      {Nonlinear Analysis: Theory, Methods \& Applications}{Nonlinear Anal.}}
   \ITEE{#2}{NAMS}{\ITE{\equal{#1}{+}}
      {Notices of the American Mathematical Society}{Notices Amer. Math. Soc.}}
   \ITEE{#2}{OpusM}{\ITE{\equal{#1}{+}}
      {Opuscula Mathematica}{Opuscula Math.}}
   \ITEE{#2}{PacJM}{\ITE{\equal{#1}{+}}
      {Pacific Journal of Mathematics}{Pacific J. Math.}}
   \ITEE{#2}{PeriodMHung}{\ITE{\equal{#1}{+}}
      {Periodica Mathematica Hungarica}{Period. Math. Hungarica}}
   \ITEE{#2}{PAMS}{\ITE{\equal{#1}{+}}
      {Proceedings of the American Mathematical Society}{Proc. Amer. Math. Soc.}}
   \ITEE{#2}{ProcCambPhS}{\ITE{\equal{#1}{+}}
      {Proceedings of the Cambridge Philosophical Society}{Proc. Cambridge Phil. Soc.}}
   \ITEE{#2}{ProcImpAcadTokyo}{\ITE{\equal{#1}{+}}
      {Proc. Imp. Acad. Tokyo}{Proc. Imp. Acad. Tokyo}}
   \ITEE{#2}{ProcKonink}{\ITE{\equal{#1}{+}}
      {Proceedings of the Koninklijke Nederlandse Akademie van Wetenschappen}
      {Nederl. Akad. Wetensch. Proc. Ser. A}}
   \ITEE{#2}{PLondMS}{\ITE{\equal{#1}{+}}
      {Proceedings of the London Mathematical Society}{Proc. London Math. Soc.}}
   \ITEE{#2}{PNAS}{\ITE{\equal{#1}{+}}
      {Proceedings of the National Academy of Sciences of the United States of America}
      {Proc. Natl. Acad. Sci. USA}}
   \ITEE{#2}{PublRIMSKyoto}{\ITE{\equal{#1}{+}}
      {Publ. Res. Inst. Math. Sci. Kyoto Univ.}{Publ. Res. Inst. Math. Sci.}}
   \ITEE{#2}{PublSUAA}{\ITE{\equal{#1}{+}}
      {Publ. Sci. Univ. Alger. S\'{e}r. A}{Publ. Sci. Univ. Alger. S\'{e}r. A}}
   \ITEE{#2}{PWN}{\ITE{\equal{#1}{+}}
      {PWN -- Polish Scientific Publishers, Warszawa}
      {PWN -- Polish Scientific Publishers, Warszawa}}
   \ITEE{#2}{RCMP}{\ITE{\equal{#1}{+}}
      {Rendiconti del Circolo Matematico di Palermo}{Rend. Circ. Mat. Palermo}}
   \ITEE{#2}{RussMS}{\ITE{\equal{#1}{+}}
      {Russian Mathematical Surveys}{Russian Math. Surveys}}
   \ITEE{#2}{SbM}{\ITE{\equal{#1}{+}}
      {Sbornik: Mathematics}{Sb. Math.}}
   \ITEE{#2}{SciRepTokyoA}{\ITE{\equal{#1}{+}}
      {Science Reports of Tokyo Kyoiku Daigaku, Section A}
      {Sci. Rep. Tokyo Kyoiku Daigaku Sect. A}}
   \ITEE{#2}{SeminProbStras}{\ITE{\equal{#1}{+}}
      {S\'{e}minaire de probabilit\'{e}s de Strasbourg}{S\'{e}min. Prob. Strasbourg}}
   \ITEE{#2}{SIAMJMAA}{\ITE{\equal{#1}{+}}
      {SIAM Journal on Matrix Analysis and Applications}{SIAM J. Matrix Anal. Appl.}}
   \ITEE{#2}{SibirMZ}{\ITE{\equal{#1}{+}}
      {Sibirski\v{\i} Mat. \v{Z}hurnal}{Sibirsk. Mat. \v{Z}.}}
   \ITEE{#2}{SM}{\ITE{\equal{#1}{+}}
      {Studia Mathematica}{Studia Math.}}
   \ITEE{#2}{TAMS}{\ITE{\equal{#1}{+}}
      {Transactions of the American Mathematical Society}{Trans. Amer. Math. Soc.}}
   \ITEE{#2}{TohokuMJ}{\ITE{\equal{#1}{+}}
      {T\^{o}hoku Mathematical Journal}{T\^{o}hoku Math. J.}}
   \ITEE{#2}{TomskUnivRev}{\ITE{\equal{#1}{+}}
      {Tomsk Universitet Review}{Tomsk. Univ. Rev.}}
   \ITEE{#2}{TopA}{\ITE{\equal{#1}{+}}
      {Topology and its Applications}{Topology Appl.}}
   \ITEE{#2}{TMNA}{\ITE{\equal{#1}{+}}
      {Topological Methods in Nonlinear Analysis}{Topol. Methods Nonlinear Anal.}}
   \ITEE{#2}{TsukubaJM}{\ITE{\equal{#1}{+}}
      {Tsukuba Journal of Mathematics}{Tsukuba J. Math.}}
   \ITEE{#2}{UspekhiMN}{\ITE{\equal{#1}{+}}
      {Uspekhi Matem. Nauk}{Uspekhi Mat. Nauk}}
   \ITEE{#2}{AnnSciEcNormSup}{\ITE{\equal{#1}{+}}
      {Annales Scientifiques de l'\'{E}cole Normale Sup\'{e}rieure}
      {Ann. Sci. \'{E}c. Norm. Sup\'{e}r.}}
   }

\newcommand{\paplist}[3][]{
   \ITEE{#3}{HAbels,AManoussos,GNoskov2011}{
      \BIB{#2}{H. Abels, A. Manoussos, G. Noskov}
         {Proper actions and proper invariant metrics}
         {\jRN{JLondMS} (2)}{83}{2011}{619--636}{#1}}
   \ITEE{#3}{NIAkhiezer,IMGlazman1993}{
      \BIb{#2}{N.I. Akhiezer and I.M. Glazman}
         {Theory of Linear Operators in Hilbert Space}
         {Dover Publications, Inc., New York}{1993}{#1}}
   \ITEE{#3}{ASAmitsur,JLevitzki1950}{
      \BIB{#2}{A.S. Amitsur and J. Levitzki}
         {Minimal identities for algebras}
         {\jRN{PAMS}}{1}{1950}{449--463}{#1}}
   \ITEE{#3}{RDAnderson1966}{
      \BIB{#2}{R.D. Anderson}
         {Hilbert space is homeomorphic to the countable infinite product of lines}
         {\jRN{BAMS}}{72}{1966}{515--519}{#1}}
   \ITEE{#3}{RDAnderson1967}{
      \BIB{#2}{R.D. Anderson}
         {On topological infinite deficiency}
         {\jRN{MichMJ}}{14}{1967}{365--383}{#1}}
   \ITEE{#3}{RDAnderson,JMcCharen1970}{
      \BIB{#2}{R.D. Anderson and J. McCharen}
         {On extending homeomorphisms to Fr\'{e}chet manifolds}
         {\jRN{PAMS}}{25}{1970}{283--289}{#1}}
   \ITEE{#3}{RDAnderson,DWCurtis,JVanMill1982}{
      \BIB{#2}{R.D. Anderson, D.W. Curtis, J. van Mill}
         {A fake topological Hilbert space}
         {\jRN{TAMS}}{272}{1982}{311--321}{#1}}
   \ITEE{#3}{RJArchbold1995}{
      \BIB{#2}{R.J. Archbold}
         {On residually finite\hyp{}dimensional $\CCc^*$-algebras}
         {\jRN{PAMS}}{123}{1995}{2935--2937}{#1}}
   \ITEE{#3}{RArens,JEells1956}{
      \BIB{#2}{R. Arens and J. Eells}
         {On embedding uniform and topological spaces}
         {\jRN{PacJM}}{6}{1956}{397--403}{#1}}
   \ITEE{#3}{AVArhangelskii2002}{
      \BIB{#2}{A.V. Arhangel'skii}
         {The Hewitt\hyp{}Nachbin completion in topological algebra. Some effects of homogeneity}
         {\jRN{ACS}}{10}{2002}{267--278}{#1}}
   \ITEE{#3}{AVArhangelskii,MGTkachenko2008}{
      \BIb{#2}{A.V. Arhangel'skii and M.G. Tkachenko}
         {Topological Groups and Related Structures}
         {Atlantis Press, Paris; World Scientific, Hackensack, NJ}{2008}{#1}}
   \ITEE{#3}{NAronszajn,PPanitchpakdi1956}{
      \BIB{#2}{N. Aronszajn and P. Panitchpakdi}
         {Extension of uniformly continuous transformations and hyperconvex metric spaces}
         {\jRN{PacJM}}{6}{1956}{405--439}{#1}}
   \ITEE{#3}{KJBabenko1948}{
      \BIB{#2}{K.J. Babenko}
         {On conjugate functions}
         {\jRN{DANSSSR}}{62}{1948}{157--160}{#1}}
   \ITEE{#3}{TBanakh1995}{
      \BIB{#2}{T.O. Banakh}
         {Topology of spaces of probability measures, I}
         {\jRN{MStud}}{5}{1995}{65--87 (Russian)}{#1}}
   \ITEE{#3}{TBanakh1995a}{
      \BIB{#2}{T.O. Banakh}
         {Topology of spaces of probability measures, II}
         {\jRN{MStud}}{5}{1995}{88--106 (Russian)}{#1}}
   \ITEE{#3}{TBanakh1998}{
      \BIB{#2}{T. Banakh}
         {Characterization of spaces admitting a homotopy dense embedding into a Hilbert manifold}
         {\jRN{TopA}}{86}{1998}{123--131}{#1}}
   \ITEE{#3}{TBanakh,CzBessaga2000}{
      \BIB{#2}{T. Banakh and Cz. Bessaga}
         {On linear operators extending [pseudo]metrics}
         {\jRN{BPAS}}{48}{2000}{35--49}{#1}}
   \ITEE{#3}{TBanakh,TNRadul1997}{
      \BIB{#2}{T.O. Banakh and T.N. Radul}
         {Topology of spaces of probability measures}
         {\jRN{SbM}}{188}{1997}{973--995}{#1}}
   \ITEE{#3}{TBanakh,TRadul,MZarichnyi1996}{
      \BIb{#2}{T. Banakh, T. Radul, M. Zarichnyi}
         {Absorbing sets in infinite\hyp{}dimensional manifolds}
         {VNTL Publishers, Lviv}{1996}{#1}}
   \ITEE{#3}{TBanakh,IZarichnyy2008}{
      \BIB{#2}{T. Banakh and I. Zarichnyy}
         {Topological groups and convex sets homeomorphic to non\hyp{}separable Hilbert spaces}
         {\jRN{CEurJM}}{6}{2008}{77--86}{#1}}
   \ITEE{#3}{HBecker,ASKechris1996}{
      \BIb{#2}{H. Becker and A.S. Kechris}{The Descriptive Set Theory of Polish Group Actions 
         \textup{(London Math. Soc. Lecture Note Series, vol. 232)}}
         {University Press, Cambridge}{1996}{#1}}
   \ITEE{#3}{GBeer1993}{
      \BIb{#2}{G. Beer}
         {Topologies on Closed and Closed Convex Sets \textup{(Mathematics and Its Applications)}}
         {Kluwer Academic Publishers, Dordrecht}{1993}{#1}}
   \ITEE{#3}{NEBenamara,NNikolski1999}{
      \BIB{#2}{N.E. Benamara and N. Nikolski}
         {Resolvent tests for similarity to a normal operator}
         {\jRN{PLondMS}}{78}{1999}{585--626}{#1}}
   \ITEE{#3}{YBenyamini,JLindenstrauss2000}{
      \BIb{#2}{Y. Benyamini and J. Lindenstrauss}
         {Geometric nonlinear functional analysis I}
         {AMS Colloquium Publications 48}{2000}{#1}}
   \ITEE{#3}{SKBerberian1974}{
      \BIb{#2}{S.K. Berberian}
         {Lectures in Functional Analysis and Operator Theory}
         {Graduate Texts in Mathematics 15, Springer\hyp{}Verlag, New York}{1974}{#1}}
   \ITEE{#3}{SNBernstein1954}{
      \BIb{#2}{S.N. Bernstein}
         {Collected Works II}
         {Akad. Nauk SSSR, Moscow}{1954 (Russian)}{#1}}
   \ITEE{#3}{CzBessaga,APelczynski1972}{
      \BIB{#2}{Cz. Bessaga and A. Pe\l{}czy\'{n}ski}
         {On spaces of measurable functions}
         {\jRN{SM}}{44}{1972}{597--615}{#1}}
   \ITEE{#3}{CzBessaga,APelczynski1975}{
      \BIb{#2}{Cz. Bessaga and A. Pe\l{}czy\'{n}ski}
         {Selected topics in infinite\hyp{}dimensional topology}
         {\jRN{PWN}}{1975}{#1}}
   \ITEE{#3}{MBestvina,JMogilski1986}{
      \BIB{#2}{M. Bestvina and J. Mogilski}
         {Characterizing certain incomplete infinite\hyp{}dimensional absolute retracts}
         {\jRN{MichMJ}}{33}{1986}{291--313}{#1}}
   \ITEE{#3}{MBestvina,PBowers,JMogilsky,JWalsh1986}{
      \BIB{#2}{M. Bestvina, P. Bowers, J. Mogilsky, J. Walsh}
         {Characterization of Hilbert space manifolds revisited}
         {\jRN{TopA}}{24}{1986}{53--69}{#1}}
   \ITEE{#3}{RBhatia1997}{
      \BIb{#2}{R. Bhatia}
         {Matrix Analysis}
         {Springer, New York}{1997}{#1}}
   \ITEE{#3}{GBirkhoff1936}{
      \BIB{#2}{G. Birkhoff}
         {A note on topological groups}
         {\jRN{ComposM}}{3}{1936}{427--430}{#1}}
   \ITEE{#3}{MSBirman,MZSolomjak1987}{
      \BIb{#2}{M.S. Birman and M.Z. Solomjak}
         {Spectral Theory of Self\hyp{}Adjoint Operators in Hilbert Space}
         {D. Reidel Publishing Co., Dordrecht}{1987}{#1}}
   \ITEE{#3}{EBishop1961}{
      \BIB{#2}{E. Bishop}
         {A generalization of the Stone\hyp{}Weierstrass theorem}
         {\jRN{PacJM}}{11}{1961}{777--783}{#1}}
   \ITEE{#3}{BBlackadar2006}{\BIb{#2}{B. Blackadar}{Operator Algebras. 
         Theory of $\CCc^*$\hyp{}algebras and von Neumann algebras \textup{(Encyclopaedia 
         of Mathematical Sciences, vol. 122: Operator Algebras and Non\hyp{}Commutative Geometry 
         III)}}{Springer\hyp{}Verlag, Berlin\hyp{}Heidelberg}{2006}{#1}}
   \ITEE{#3}{JBlass,WHolsztynski1972}{
      \BIB{#2}{J. Blass and W. Holszty\'{n}ski}
         {Cubical polyhedra and homotopy III}
         {\jRN{AttiAccLincRendNat}}{53}{1972}{275--279}{#1}}
   \ITEE{#3}{FFBonsall,NJDuncan1973}{
      \BIb{#2}{F.F. Bonsall and N.J. Duncan}
         {Complete Normed Algebras}
         {Springer Verlag, Berlin}{1973}{#1}}
   \ITEE{#3}{NBourbaki2002}{
      \BIb{#2}{N. Bourbaki}
         {Lie Groups and Lie Algebras, Chapters 4--6}
         {Springer, New York}{2002}{#1}}
   \ITEE{#3}{ABouziad1996}{
      \BIB{#2}{A. Bouziad}
         {Every \v{C}ech-analytic Baire semitopological group is a topological group}
         {\jRN{PAMS}}{124}{1996}{953--959}{#1}}
   \ITEE{#3}{PLBowers1989}{
      \BIB{#2}{P.L. Bowers}
         {Limitation topologies on function spaces}
         {\jRN{TAMS}}{314}{1989}{421--431}{#1}}
   \ITEE{#3}{JBraconnier1948}{
      \BIB{#2}{J. Braconnier}
         {Sur les groupes topologiques localement compacts}
         {\jRN{JMPuApNS}}{27}{1948}{1--85}{#1}}
   \ITEE{#3}{NBrand1982}{
      \BIB{#2}{N. Brand}
         {Another note on the continuity of the inverse}
         {\jRN{ArchM}}{39}{1982}{241--245}{#1}}
   \ITEE{#3}{ABrown1953}{
      \BIB{#2}{A. Brown}
         {On a class of operators}
         {\jRN{PAMS}}{4}{1953}{723--728}{#1}}
   \ITEE{#3}{ABrown,CKFong,DWHadwin1978}{
      \BIB{#2}{A. Brown, C.-K. Fong, D.W. Hadwin}
         {Parts of operators on Hilbert space}
         {\jRN{IllinoisJM}}{22}{1978}{306--314}{#1}}
   \ITEE{#3}{RBrown,SMorris1977}{
      \BIB{#2}{R. Brown and S. Morris}
         {Embeddings in contractible or compact objects}
         {\jRN{CollM}}{38}{1977}{213--222}{#1}}
   \ITEE{#3}{AMBruckner,JBBruckner,BSThomson1997}{
      \BIb{#2}{A.M. Bruckner, J.B. Bruckner, B.S. Thomson}
         {Real Analysis}
         {Prentice\hyp{}Hall, New Jersey}{1997}{#1}}
   \ITEE{#3}{PJCameron,AMVershik2006}{
      \BIB{#2}{P.J. Cameron and A.M. Vershik}
         {Some isometry groups of Urysohn space}
         {\jRN{AnnPALog}}{143}{2006}{70--78}{#1}}
   \ITEE{#3}{CCastaing1966}{
      \BIB{#2}{C. Castaing}{Quelques 
         probl\`{e}mes de mesurabilit\'{e} li\'{e}es \`{a} la th\'{e}orie de la commande}
         {\jRN{CRParis}}{262}{1966}{409--411}{#1}}
   \ITEE{#3}{JAVanCasteren1980}{
      \BIB{#2}{J.A. van Casteren}
         {A problem of Sz.\hyp{}Nagy}
         {\jRN{ActaSM}}{42}{1980}{189--194}{#1}}
   \ITEE{#3}{JAVanCasteren1983}{
      \BIB{#2}{J.A. van Casteren}
         {Operators similar to unitary or selfadjoint ones}
         {\jRN{PacJM}}{104}{1983}{241--255}{#1}}
   \ITEE{#3}{XCatepillan,MPtak,WSzymanski1994}{
      \BIB{#2}{X. Catepill\'{a}n, M. Ptak, W. Szyma\'{n}ski}{Multiple 
         canonical decompositions of families of operators and a model of quasinormal families}
         {\jRN{PAMS}}{121}{1994}{1165--1172}{#1}}
   \ITEE{#3}{RCauty1994}{
      \BIB{#2}{R. Cauty}
         {Un espace m\'{e}trique lin\'{e}aire qui n'est pas un r\'{e}tracte absolu}
         {\jRN{FM}}{146}{1994}{85--99, (French)}{#1}}
   \ITEE{#3}{TAChapman1971}{
      \BIB{#2}{T.A. Chapman}
         {Deficiency in infinite\hyp{}dimensional manifolds}
         {\jRN{GTopA}}{1}{1971}{263--272}{#1}}
   \ITEE{#3}{TAChapman1976}{
      \BIb{#2}{T.A. Chapman}
         {Lectures on Hilbert cube manifolds}
         {C.B.M.S. Regional Conference Series in Math. No 28, Amer. Math. Soc.}{1976}{#1}}
   \ITEE{#3}{WMChing1974}{
      \BIB{#2}{W.-M. Ching}
         {Topologies on the quasi-spectrum of a $\CCc^*$\hyp{}algebra}
         {\jRN{PAMS}}{46}{1974}{273--276}{#1}}
   \ITEE{#3}{RBChuaqui1977}{
      \BIB{#2}{R.B. Chuaqui}
         {Measures invariant under a group of transformations}
         {\jRN{PacJM}}{68}{1977}{313--329}{#1}}
   \ITEE{#3}{JBConway1985}{
      \BIb{#2}{J.B. Conway}
         {A Course in Functional Analysis}
         {Springer\hyp{}Verlag, New York}{1985}{#1}}
   \ITEE{#3}{JBConway2000}{
      \BIb{#2}{J.B. Conway}
         {A Course in Operator Theory}
         {(Graduate Studies in Mathematics, vol. 21) Amer. Math. Soc., Providence}{2000}{#1}}
   \ITEE{#3}{GCorach,AMaestripieri,MMbekhta2009}{
      \BIB{#2}{G. Corach, A. Maestripieri, M. Mbekhta}
         {Metric and homogeneous structure of closed range operators}
         {\jRN{JOT}}{61}{2009}{171--190}{#1}}
   \ITEE{#3}{MJCowen,RGDouglas1978}{
      \BIB{#2}{M.J. Cowen and R.G. Douglas}
         {Complex geometry and operator theory}
         {\jRN{ActaM}}{141}{1978}{187--261}{#1}}
   \ITEE{#3}{DWCurtis1985}{
      \BIB{#2}{D.W. Curtis}
         {Boundary sets in the Hilbert cube}
         {\jRN{TopA}}{20}{1985}{201--221}{#1}}
   \ITEE{#3}{DVanDantzig,BLVanDerWaerden1928}{
      \BIB{#2}{D. van Dantzig and B.L. van der Waerden}
         {\"{U}ber metrisch homogene R\"{a}ume}
         {\jRN{AbhHamburg}}{6}{1928}{367--376}{#1}}
   \ITEE{#3}{MMDay1958}{
      \BIb{#2}{M.M. Day}
         {Normed Linear Spaces}
         {Springer Verlag, Berlin}{1958}{#1}}
   \ITEE{#3}{CDellacherie1967}{
      \BIB{#2}{C. Dellacherie}
         {Un compl\'{e}ment au th\'{e}or\`{e}me de Weierstrass\hyp{}Stone}
         {\jRN{SeminProbStras}}{1}{1967}{52--53}{#1}}
   \ITEE{#3}{JJDijkstra1987}{
      \BIB{#2}{J.J. Dijkstra}
         {Strong negligibility of $\sigma$\hyp{}compacta does not characterize Hilbert space}
         {\jRN{PacJM}}{127}{1987}{19--30}{#1}}
   \ITEE{#3}{JJDijkstra1990}{
      \BIB{#2}{J.J. Dijkstra}
         {Characterizing Hilbert space topology in terms of strong negligibility}
         {\jRN{ComposM}}{75}{1990}{299--306}{#1}}
   \ITEE{#3}{JDixmier,CFoias1972}{
      \BiB{#2}{J. Dixmier and C. Foias}
         {Sur le spectre ponctuel d'un op\'{e}rateur \textup{(French)}}{in:}
         {Hilbert space operators and operator algebras (Proc. Internat. Conf., Tihany, 1970)}
         {\jRN{CollMSJBoly}, No. 5, North-Holland, Amsterdam}{1972}{127--133}{#1}}
   \ITEE{#3}{TDobrowolski,WMarciszewski2002}{
      \BIB{#2}{T. Dobrowolski and W. Marciszewski}
         {Failure of the Factor Theorem for Borel pre\hyp{}Hilbert spaces}
         {\jRN{FM}}{175}{2002}{53--68}{#1}}
   \ITEE{#3}{TDobrowolski,JMogilski1990}{
      \BiB{#2}{T. Dobrowolski and J. Mogilski}{Problems on Topological Classification 
         of Incomplete Metric Spaces}{Chapter 25 in:}{Open Problems in Topology}{J. van Mill 
         and G.M. Reed (eds.), North\hyp{}Holland Amsterdam}{1990}{411--429}{#1}}
   \ITEE{#3}{TDobrowolski,HTorunczyk1981}{
      \BIB{#2}{T. Dobrowolski and H. Toru\'{n}czyk}
         {Separable complete ANR's admitting a group structure are Hilbert manifolds}
         {\jRN{TopA}}{12}{1981}{229--235}{#1}}
   \ITEE{#3}{RGDouglas1966}{
      \BIB{#2}{R.G. Douglas}
         {On majorization, factorization and range inclusion of operators in Hilbert space}
         {\jRN{PAMS}}{17}{1966}{413--416}{#1}}
   \ITEE{#3}{CHDowker1947}{
      \BIB{#2}{C.H. Dowker}
         {Mapping theorems for non\hyp{}compact spaces}
         {\jRN{AmJM}}{69}{1947}{200--242}{#1}}
   \ITEE{#3}{CHDowker1952}{
      \BIB{#2}{C.H. Dowker}
         {Topology of metric complexes}
         {\jRN{AmJM}}{74}{1952}{555--577}{#1}}
   \ITEE{#3}{JDugundji1951}{
      \BIB{#2}{J. Dugundji}
         {An extension of Tietze's theorem}
         {\jRN{PacJM}}{1}{1951}{353--367}{#1}}
   \ITEE{#3}{JDugundji1958}{
      \BIB{#2}{J. Dugundji}
         {Absolute neighborhood retracts and local connectedness for arbitrary metric spaces}
         {\jRN{ComposM}}{13}{1958}{229--246}{#1}}
   \ITEE{#3}{JDugundji1965}{
      \BIB{#2}{J. Dugundji}
         {Locally equiconnected spaces and absolute neighborhood retracts}
         {\jRN{FM}}{57}{1965}{187--193}{#1}}
   \ITEE{#3}{NDunford,JTSchwartz1958}{
      \BIb{#2}{N. Dunford and J.T. Schwartz}
         {Linear Operators, part I}
         {Interscience Publishers, New York}{1958}{#1}}
   \ITEE{#3}{NDunford,JTSchwartz1963}{
      \BIb{#2}{N. Dunford and J.T. Schwartz}
         {Linear Operators, part II}
         {Interscience Publishers, New York}{1963}{#1}}
   \ITEE{#3}{NDunford,JTSchwartz1971}{
      \BIb{#2}{N. Dunford and J.T. Schwartz}
         {Linear Operators, part III}
         {Wiley\hyp{}Interscience, New York}{1971}{#1}}
   \ITEE{#3}{MLEaton,MDPerlman1977}{
      \BIB{#2}{M.L. Eaton and M.D. Perlman}
         {Reflection groups, generalized Schur functions and the geometry of majorization}
         {\jRN{AnnProb}}{5}{1977}{829--860}{#1}}
   \ITEE{#3}{JEells,NHKuiper1969}{
      \BIB{#2}{J. Eells and N.H. Kuiper}
         {Homotopy negligible subsets in infinite\hyp{}dimensional manifolds}
         {\jRN{ComposM}}{21}{1969}{151--161}{#1}}
   \ITEE{#3}{EGEffros1965}{
      \BIB{#2}{E.G. Effros}
         {The Borel space of von Neumann algebras on a separable Hilbert space}
         {\jRN{PacJM}}{15}{1965}{1153--1164}{#1}}
   \ITEE{#3}{EGEffros1966}{
      \BIB{#2}{E.G. Effros}
         {Global structure in von Neumann algebras}
         {\jRN{TAMS}}{121}{1966}{434--454}{#1}}
   \ITEE{#3}{REllis1957}{
      \BIB{#2}{R. Ellis}
         {A note on the continuity of the inverse}
         {\jRN{PAMS}}{8}{1957}{372--373}{#1}}
   \ITEE{#3}{REngelking1977}{
      \BIb{#2}{R. Engelking}
         {General Topology}
         {\jRN{PWN}}{1977}{#1}}
   \ITEE{#3}{REngelking1978}{
      \BIb{#2}{R. Engelking}
         {Dimension Theory}
         {\jRN{PWN}}{1978}{#1}}
   \ITEE{#3}{REngelking1989}{
      \BIb{#2}{R. Engelking}{General Topology. 
         Revised and completed edition \textup{(Sigma series in pure mathematics, vol. 6)}}
         {Heldermann Verlag, Berlin}{1989}{#1}}
   \ITEE{#3}{PErdos,RDMauldin1976}{
      \BIB{#2}{P. Erd\"{o}s and R.D. Mauldin}
         {The nonexistence of certain invariant measures}
         {\jRN{PAMS}}{59}{1976}{321--322}{#1}}
   \ITEE{#3}{JErnest1976}{
      \BIB{#2}{J. Ernest}
         {Charting the operator terrain}
         {\jRN{MAMS}}{171}{1976}{207 pp}{#1}}
   \ITEE{#3}{REspinola,MAKhamsi2001}{
      \BiB{#2}{R. Espinola and M.A. Khamsi}{Introduction 
         to hyperconvex spaces}{Chapter XIII in:}{Handbook of Metric Fixed Point Theory}
         {W.A. Kirk and B. Sims (editors), Kluwer Academic Publishers}{2001}{391--435}{#1}}
   \ITEE{#3}{JMFell1960a}{
      \BIB{#2}{J.M. Fell}
         {$\CCc^*$-algebras with smooth dual}
         {\jRN{IllinoisJM}}{4}{1960}{221--230}{#1}}
   \ITEE{#3}{JMFell1960b}{
      \BIB{#2}{J.M. Fell}
         {The dual spaces of $\CCc^*$-algebras}
         {\jRN{TAMS}}{94}{1960}{365--403}{#1}}
   \ITEE{#3}{LAFialkow1975}{
      \BIB{#2}{L.A. Fialkow}
         {The similarity orbit of a normal operator}
         {\jRN{TAMS}}{210}{1975}{129--137}{#1}}
   \ITEE{#3}{PAFillmore,JPWilliams1971}{
      \BIB{#2}{P.A. Fillmore and J.P. Williams}
         {On operator ranges}
         {\jRN{AdvM}}{7}{1971}{254--281}{#1}}
   \ITEE{#3}{RHFox1943}{
      \BIB{#2}{R.H. Fox}
         {On fiber spaces, II}
         {\jRN{BAMS}}{49}{1943}{733--735}{#1}}
   \ITEE{#3}{RFraisse1954}{
      \BIB{#2}{R. Fra\"{\i}ss\'{e}}
         {Sur quelques classifications des syst\`{e}mes de relations}
         {\jRN{PublSUAA}}{1}{1954}{35--182}{#1}}
   \ITEE{#3}{NAFriedman1970}{
      \BIb{#2}{N.A. Friedman}
         {Introduction to ergodic theory}
         {Van Nostrand Reinhold Company}{1970}{#1}}
   \ITEE{#3}{MFujii,MKajiwara,YKato,FKubo1976}{
      \BIB{#2}{M. Fujii, M. Kajiwara, Y. Kato, F. Kubo}
         {Decompositions of operators in Hilbert spaces}
         {\jRN{MathJap}}{21}{1976}{117--120}{#1}}
   \ITEE{#3}{SGao,ASKechris2003}{
      \BIB{#2}{S. Gao and A.S. Kechris}
         {On the classification of Polish metric spaces up to isometry}
         {\jRN{MAMS}}{161}{2003}{viii+78}{#1}}
   \ITEE{#3}{MIGarrido,FMontalvo1991}{
      \BIB{#2}{M.I. Garrido and F. Montalvo}
         {On some generalizations of the Kakutani\hyp{}Stone and Stone\hyp{}Weierstrass theorems}
         {\jRN{ExtrM}}{6}{1991}{156--159}{#1}}
   \ITEE{#3}{LGe,JShen2002}{
      \BIB{#2}{L. Ge and J. Shen}
         {Generator problem for certain property T factors}
         {\jRN{PNAS}}{99}{2002}{565--567}{#1}}
   \ITEE{#3}{IMGelfand,MANaimark1943}{
      \BIB{#2}{I.M. Gelfand and M.A. Naimark}
         {On the embedding of normed rings into the ring of operators in Hilbert space}
         {\jRN{MSb}}{12}{1943}{197--213}{#1}}
   \ITEE{#3}{RGellar,LPage1974}{
      \BIB{#2}{R. Gellar and L. Page}
         {Limits of unitarily equivalent normal operators}
         {\jRN{DMJ}}{41}{1974}{319--322}{#1}}
   \ITEE{#3}{FGesztesy,MMalamud,MMitrea,SNaboko2009}{
      \BIB{#2}{F. Gesztesy, M. Malamud, M. Mitrea, S. Naboko}{Generalized 
         polar decompositions for closed operators in Hilbert spaces and some applications}
         {\jRN{IEOT}}{64}{2009}{83--113}{#1}}
   \ITEE{#3}{LGillman,MJerison1960}{
      \BIb{#2}{L. Gillman and M. Jerison}
         {Rings of continuous functions}
         {New York}{1960}{#1}}
   \ITEE{#3}{JGlimm1960}{
      \BIB{#2}{J. Glimm}
         {A Stone\hyp{}Weierstrass theorem for $\CCc^*$\hyp{}algebras}
         {\jRN{AnnM}}{72}{1960}{216--244}{#1}}
   \ITEE{#3}{JGlimm1961}{
      \BIB{#2}{J. Glimm}
         {Type I $\CCc^*$-algebras}
         {\jRN{AnnM}}{73}{1961}{572--612}{#1}}
   \ITEE{#3}{GGodefroy,NJKalton2003}{
      \BIB{#2}{G. Godefroy and N.J. Kalton}
         {Lipschitz\hyp{}free Banach spaces}
         {\jRN{SM}}{159}{2003}{121--141}{#1}}
   \ITEE{#3}{ICGohberg,MGKrein1967}{
      \BIB{#2}{I.C. Gohberg and M.G. Krein}
         {On a description of contraction operators similar to unitary ones}
         {\jRN{FunkAnalPril}}{1}{1967}{38--60}{#1}}
   \ITEE{#3}{KRGoodearl,PMenal1990}{
      \BIB{#2}{K.R. Goodearl and P. Menal}
         {Free and residually finite\hyp{}dimensional $\CCc^*$-algebras}
         {\jRN{JFA}}{90}{1990}{391--410}{#1}}
   \ITEE{#3}{ELGriffinJr1953}{
      \BIB{#2}{E.L. Griffin Jr.}
         {Some contributions to the theory of rings of operators}
         {\jRN{TAMS}}{75}{1953}{471--504}{#1}}
   \ITEE{#3}{ELGriffinJr1955}{
      \BIB{#2}{E.L. Griffin Jr.}
         {Some contributions to the theory of rings of operators II}
         {\jRN{TAMS}}{79}{1955}{389--400}{#1}}
   \ITEE{#3}{MGromov1981}{
      \BIB{#2}{M. Gromov}
         {Groups of polynomial growth and expanding maps}
         {\jRN{InHauEtSPM}}{53}{1981}{53--73}{#1}}
   \ITEE{#3}{MGromov1999}{
      \BIb{#2}{M. Gromov}
         {Metric Structures for Riemannian and Non\hyp{}Riemannian Spaces}
         {Progress in Math. \textbf{152}, Birkh\"{a}user}{1999}{#1}}
   \ITEE{#3}{JDeGroot1956}{
      \BIB{#2}{J. de Groot}
         {Non\hyp{}archimedean metrics in topology}
         {\jRN{PAMS}}{7}{1956}{948--953}{#1}}
   \ITEE{#3}{LCGrove,CTBenson1985}{
      \BIb{#2}{L.C. Grove and C.T. Benson}
         {Finite Reflection Group}
         {2nd ed., Springer\hyp{}Verlag}{1985}{#1}}
   \ITEE{#3}{JBGuerrero,ARodriguez-Palacios2002}{
      \BIB{#2}{J.B. Guerrero and A. Rodr\'{\i}guez\hyp{}Palacios}
         {Transitivity of the Norm on Banach Spaces}
         {\jRN{ExtrM}}{17}{2002}{1--58}{#1}}
   \ITEE{#3}{VIGurarii1966}{
      \BIB{#2}{V.I. Gurari\v{\i}}{Spaces of universal placement, isotropic spaces and a problem 
         of Mazur on rotations of Banach spaces \textup{(Russian)}}
         {\jRN{SibirMZ}}{7}{1966}{1002--1013}{#1}}
   \ITEE{#3}{DWHadwin1974}{
      \BIB{#2}{D.W. Hadwin}
         {Closures of unitary equivalence classes}
         {\jRN{NAMS}}{21}{1974}{\#74T-B55}{#1}}
   \ITEE{#3}{DWHadwin1976}{
      \BIB{#2}{D.W. Hadwin}
         {An operator\hyp{}valued spectrum}
         {\jRN{NAMS}}{23}{1976}{A-163}{#1}}
   \ITEE{#3}{DWHadwin1977}{
      \BIB{#2}{D.W. Hadwin}
         {An operator\hyp{}valued spectrum}
         {\jRN{IndianaUMJ}}{26}{1977}{329--340}{#1}}
   \ITEE{#3}{DWHadwin1981}{
      \BIB{#2}{D.W. Hadwin}
         {Nonseparable approximate equivalence}
         {\jRN{TAMS}}{266}{1981}{203--231}{#1}}
   \ITEE{#3}{HHahn1932}{
      \BIb{#2}{H. Hahn}
         {Reelle Funktionen I}
         {Leipzig}{1932}{#1}}
   \ITEE{#3}{PRHalmos1950}{
      \BIb{#2}{P.R. Halmos}
         {Measure theory}
         {Van Nostrand, New York}{1950}{#1}}
   \ITEE{#3}{PRHalmos1951}{
      \BIb{#2}{P.R. Halmos}
         {Introduction to Hilbert Space and the Theory of Spectral Multiplicity}
         {Chelsea Publishing Company, New York}{1951}{#1}}
   \ITEE{#3}{PRHalmos1956}{
      \BIb{#2}{P.R. Halmos}
         {Lectures on Ergodic Theory}
         {Publ. Math. Soc. Japan, Tokyo}{1956}{#1}}
   \ITEE{#3}{PRHalmos1974}{
      \BIb{#2}{P.R. Halmos}
         {Measure theory}
         {Springer\hyp{}Verlag, New York}{1974}{#1}}
   \ITEE{#3}{PRHalmos1982}{
      \BIb{#2}{P.R. Halmos}
         {A Hilbert Space Problem Book}
         {Springer\hyp{}Verlag New York Inc.}{1982}{#1}}
  \ITEE{#3}{PRHalmos,JEMcLaughlin1963}{
      \BIB{#2}{P.R. Halmos and J.E. McLaughlin}
         {Partial isometries}
         {\jRN{PacJM}}{13}{1963}{585--596}{#1}}
   \ITEE{#3}{RWHansell1972}{
      \BIB{#2}{R.W. Hansell}
         {On the nonseparable theory of Borel and Souslin sets}
         {\jRN{BAMS}}{78}{1972}{236--241}{#1}}
   \ITEE{#3}{SHartman,JMycielski1957}{
      \BIB{#2}{S. Hartman and J. Mycielski}
         {On the imbedding of topological groups into connected topological groups}
         {\jRN{CollM}}{5}{1957}{167--169}{#1}}
   \ITEE{#3}{FHausdorff1930}{
      \BIB{#2}{F. Hausdorff}
         {Erweiterung einer Hom\"{o}omorphie}
         {\jRN{FM}}{16}{1930}{353--360}{#1}}
   \ITEE{#3}{FHausdorff1934}{
      \BIB{#2}{F. Hausdorff}
         {\"{U}ber innere Abbildungen}
         {\jRN{FM}}{23}{1934}{279--291}{#1}}
   \ITEE{#3}{FHausdorff1938}{
      \BIB{#2}{F. Hausdorff}
         {Erweiterung einer stetigen Abbildung}
         {\jRN{FM}}{30}{1938}{40--47}{#1}}
   \ITEE{#3}{DWHenderson1971}{
      \BIB{#2}{D.W. Henderson}
         {Corrections and extensions of two papers about infinite\hyp{}dimensional manifolds}
         {\jRN{GTopA}}{1}{1971}{321--327}{#1}}
   \ITEE{#3}{DWHenderson1975}{
      \BIB{#2}{D.W. Henderson}
         {$Z$\hyp{}sets in ANR's}
         {\jRN{TAMS}}{213}{1975}{205--216}{#1}}
   \ITEE{#3}{DWHenderson,RMSchori1970}{
      \BIB{#2}{D.W. Henderson and R.M. Schori}
         {Topological classification of infinite\hyp{}dimensional manifolds by homotopy type}
         {\jRN{BAMS}}{76}{1970}{121--124}{#1}}
   \ITEE{#3}{DWHenderson,JEWest1970}{
      \BIB{#2}{D.W. Henderson and J.E. West}
         {Triangulated infinite\hyp{}dimensional manifolds}
         {\jRN{BAMS}}{76}{1970}{655--660}{#1}}
   \ITEE{#3}{DAHerrero1976}{
      \BIB{#2}{D.A. Herrero}
         {Closure of similarity orbits of Hilbert space operators, II: normal operators}
         {\jRN{JLondMS}}{13}{1976}{299--316}{#1}}
   \ITEE{#3}{EHewitt,KARoss1979}{
      \BIb{#2}{E. Hewitt and K.A. Ross}{Abstract Harmonic 
         Analysis I \textup{(A Series of Comprehensive Studies in Mathematics, Vol. 115)}}
         {Springer\hyp{}Verlag, New York}{1979}{#1}}
   \ITEE{#3}{EHewitt,KARoss1997}{
      \BIb{#2}{E. Hewitt and K.A. Ross}{Abstract Harmonic 
         Analysis II \textup{(A Series of Comprehensive Studies in Mathematics, Vol. 152}}
         {Springer\hyp{}Verlag, Berlin}{1997}{#1}}
   \ITEE{#3}{BHoffmann1979}{
      \BIB{#2}{B. Hoffmann}
         {A compact contractible topological group is trivial}
         {\jRN{ArchM}}{32}{1979}{585--587}{#1}}
   \ITEE{#3}{DHofmann2002}{
      \BIB{#2}{D. Hofmann}
         {On a generalization of the Stone\hyp{}Weierstrass theorem}
         {\jRN{ACS}}{10}{2002}{569--592}{#1}}
   \ITEE{#3}{GHognas,AMukherjea1995}{
      \BIb{#2}{G. H\"ogn\"as and A. Mukherjea}{Probability 
         Measures on Semigroups. Convolution Products, Random Walks, and Random Matrices}
         {Plenum Press, New York}{1995}{#1}}
   \ITEE{#3}{MRHolmes1992}{
      \BIB{#2}{M.R. Holmes}{The universal 
         separable metric space of Urysohn and isometric embeddings thereof in Banach spaces}
         {\jRN{FM}}{140}{1992}{199--223}{#1}}
   \ITEE{#3}{MRHolmes2008}{
      \BIB{#2}{M.R. Holmes}
         {The Urysohn space embeds in Banach spaces in just one way}
         {\jRN{TopA}}{155}{2008}{1479--1482}{#1}}
   \ITEE{#3}{RRHolmes,TYTam1999}{
      \BIB{#2}{R.R. Holmes and T.Y. Tam}
         {Distance to the convex hull of an orbit under the action of a compact group}
         {\jRN{JAusMSA}}{66}{1999}{331--357}{#1}}
   \ITEE{#3}{RHorn,RMathias1990}{
      \BIB{#2}{R. Horn and R. Mathias}
         {Cauchy\hyp{}Schwartz inequalities associated with positive semidefinite matrices}
         {\jRN{LAA}}{142}{1990}{63--82}{#1}}
   \ITEE{#3}{GEHuhunaisvili1955}{
      \BIB{#2}{G.E. Huhunai\v{s}vili}
         {On a property of Urysohn's universal metric space}
         {\jRN{DANSSSR}}{101}{1955}{607--610 (Russian)}{#1}}
   \ITEE{#3}{JEHumphreys1990}{
      \BIb{#2}{J.E. Humphreys}
         {Reflection Groups and Coxeter Groups}
         {Cambridge University Press}{1990}{#1}}
   \ITEE{#3}{JRIsbell1964}{
      \BIB{#2}{J.R. Isbell}
         {Six theorems about injective metric spaces}
         {\jRN{CMHelv}}{39}{1964}{65--76}{#1}}
   \ITEE{#3}{SIzumino,YKato1985}{
      \BIB{#2}{S. Izumino and Y. Kato}
         {The closure of invertible operators on Hilbert space}
         {\jRN{ActaSM}}{49}{1985}{321--327}{#1}}
   \ITEE{#3}{CJiang2004}{
      \BIB{#2}{C. Jiang}
         {Similarity classification of Cowen\hyp{}Douglas operators}
         {\jRN{CanadJM}}{56}{2004}{742--775}{#1}}
   \ITEE{#3}{WBJohnson,JLindenstrauss2001}{
      \BiB{#2}{W.B. Johnson and J. Lindenstrauss}{Basic Concepts in the Geometry of Banach Spaces}
         {Chapter 1 in:}{Handbook of the Geometry of Banach Spaces, Vol. 1}{W.B. Johnson 
         and J. Lindenstrauss (editors), Elsevier Science B.V., Amsterdam}{2001}{1--84}{#1}}
   \ITEE{#3}{IBJung,JStochel2008}{
      \BIB{#2}{I.B. Jung and J. Stochel}
         {Subnormal operators whose adjoints have rich point spectrum}
         {\jRN{JFA}}{255}{2008}{1797--1816}{#1}}
   \ITEE{#3}{RVKadison,JRRingrose1983}{
      \BIb{#2}{R.V. Kadison and J.R. Ringrose}
         {Fundamentals of the Theory of Operator Algebras. Volume I: Elementary Theory}
         {Academic Press, Inc., New York\hyp{}London}{1983}{#1}}
   \ITEE{#3}{RVKadison,JRRingrose1986}{
      \BIb{#2}{R.V. Kadison and J.R. Ringrose}
         {Fundamentals of the Theory of Operator Algebras. Volume II: Advanced Theory}
         {Academic Press, Inc., Orlando\hyp{}London}{1986}{#1}}
   \ITEE{#3}{SKakutani1936}{
      \BIB{#2}{S. Kakutani}
         {\"{U}ber die Metrisation der topologischen Gruppen}
         {\jRN{ProcImpAcadTokyo}}{12}{1936}{82--84}{#1}}
   \ITEE{#3}{SKakutani1938}{
      \BIB{#2}{S. Kakutani}
         {Two fixed\hyp{}point theorems concerning bicompact convex sets}
         {\jRN{ProcImpAcadTokyo}}{14}{1938}{242--245}{#1}}
   \ITEE{#3}{SKakutani1941}{
      \BIB{#2}{S. Kakutani}
         {Concrete representation of abstract L\hyp{}spaces}
         {\jRN{AnnM}}{42}{1941}{523--537}{#1}}
   \ITEE{#3}{SKakutani1941a}{
      \BIB{#2}{S. Kakutani}
         {Concrete representation of abstract M\hyp{}spaces}
         {\jRN{AnnM}}{42}{1941}{994--1024}{#1}}
   \ITEE{#3}{NKalton2007}{
      \BIB{#2}{N. Kalton}
         {Extending Lipschitz maps into $\CCc(K)$\hyp{}spaces}
         {\jRN{IsraelJM}}{162}{2007}{275--315}{#1}}
   \ITEE{#3}{RKane2001}{
      \BIb{#2}{R. Kane}
         {Reflection Groups and Invariant Theory}
         {Canadian Mathematical Society, Springer}{2001}{#1}}
   \ITEE{#3}{VKannan,SRRaju1980}{
      \BIB{#2}{V. Kannan and S.R. Raju}
         {The nonexistence of invariant universal measures on semigroups}
         {\jRN{PAMS}}{78}{1980}{482--484}{#1}}
   \ITEE{#3}{IKaplansky1951}{
      \BIB{#2}{I. Kaplansky}
         {A theorem on rings of operators}
         {\jRN{PacJM}}{1}{1951}{227--232}{#1}}
   \ITEE{#3}{MKatetov1988}{
      \BiB{#2}{M. Kat\v{e}tov}{On universal metric spaces}{in: Frolik (ed.),}{General Topology 
         and its Relations to Modern Analysis and Algebra VI. Proceedings of the Sixth Prague 
         Topological Symposium 1986}{Heldermann Verlag Berlin}{1988}{323--330}{#1}}
   \ITEE{#3}{YKatznelson1960}{
      \BIB{#2}{Y. Katznelson}{Sur les alg\'{e}bres 
         dont les \'{e}l\'{e}ments non n\'{e}gatifs admettent des racines carr\'{e}es}
         {\jRN{AnnSciEcNormSupT}}{77}{1960}{167--174}{#1}}
   \ITEE{#3}{RKaufman1981}{
      \BIB{#2}{R. Kaufman}
         {Lipschitz spaces and Suslin sets}
         {\jRN{JFA}}{42}{1981}{271--273}{#1}}
   \ITEE{#3}{RKaufman1984}{
      \BIB{#2}{R. Kaufman}
         {Representation of Suslin sets by operators}
         {\jRN{IEOT}}{7}{1984}{808--814}{#1}}
   \ITEE{#3}{OHKeller1931}{
      \BIB{#2}{O.H. Keller}
         {Die Homoiomorphie der kompakten konvexen Mengen in Hilbertschen Raum}
         {\jRN{MAnn}}{105}{1931}{748--758}{#1}}
   \ITEE{#3}{MAKhamsi,WAKirk,CMartinez2000}{
      \BIB{#2}{M.A. Khamsi, W.A. Kirk, C. Martinez}
         {Fixed point and selection theorems in hyperconvex spaces}
         {\jRN{PAMS}}{128}{2000}{3275--3283}{#1}}
   \ITEE{#3}{ABKhararazishvili1998}{
      \BIb{#2}{A.B. Khararazishvili}
         {Transformation groups and invariant measures. Set\hyp{}theoretic aspects}
         {World Scientific Publishing Co., Inc., River Edge, NJ}{1998}{#1}}
   \ITEE{#3}{YKijima1987}{
      \BIB{#2}{Y. Kijima}
         {Fixed points of nonexpansive self\hyp{}maps of a compact metric space}
         {\jRN{JMAnApp}}{123}{1987}{114--116}{#1}}
  \ITEE{#3}{JSKim,ChRKim,SGLee1980}{
      \BIB{#2}{J.S. Kim, Ch.R. Kim, S.G. Lee}
         {Reducing operator valued spectra of a Hilbert space operator}
         {\jRN{JKoreanMS}}{17}{1980}{123--129}{#1}}
   \ITEE{#3}{JKindler1995}{
      \BIB{#2}{J. Kindler}
         {Minimax theorems with applications to convex metric spaces}
         {\jRN{CollM}}{68}{1995}{179--186}{#1}}
   \ITEE{#3}{WAKirk1998}{
      \BIB{#2}{W.A. Kirk}
         {Hyperconvexity of $\RRR$\hyp{}trees}
         {\jRN{FM}}{156}{1998}{67--72}{#1}}
   \ITEE{#3}{VLKleeJr1952}{
      \BIB{#2}{V.L. Klee Jr.}
         {Invariant metrics in groups (solution of a problem of Banach)}
         {\jRN{PAMS}}{3}{1952}{484--487}{#1}}
   \ITEE{#3}{JLKoszul1965}{
      \BIb{#2}{J.L. Koszul}
         {Lectures on groups of transformations}
         {Tata Institute of Fundamental Research, Bombay}{1965}{#1}}
   \ITEE{#3}{HJKowalsky1957}{
      \BIB{#2}{H.J. Kowalsky}
         {Einbettung metrischer R\"{a}ume}
         {\jRN{ArchM}}{8}{1957}{336--339}{#1}}
   \ITEE{#3}{WKubis,MRubin2010}{
      \BIB{#2}{W. Kubi\'{s} and M. Rubin}
         {Extension and reconstruction theorems for the Urysohn universal metric space}
         {\jRN{CzMJ}}{60}{2010}{1--29}{#1}}
   \ITEE{#3}{KKuratowski1966}{
      \BIb{#2}{K. Kuratowski}
         {Topology. \textup{Vol. I}}
         {\jRN{PWN}}{1966}{#1}}
   \ITEE{#3}{KKuratowski,BKnaster1927}{
      \BIB{#2}{K. Kuratowski and B. Knaster}
         {A connected and connected im kleinen point set which contains no perfect subset}
         {\jRN{BAMS}}{33}{1927}{106--109}{#1}}
   \ITEE{#3}{KKuratowski,AMostowski1976}{
      \BIb{#2}{K. Kuratowski and A. Mostowski}
         {Set Theory with an Introduction to Descriptive Set Theory}
         {\jRN{PWN}}{1976}{#1}}
   \ITEE{#3}{GLewicki1992}{
      \BIB{#2}{G. Lewicki}
         {Bernstein's ``lethargy'' theorem in metrizable topological linear spaces}
         {\jRN{MonatM}}{113}{1992}{213--226}{#1}}
   \ITEE{#3}{ASLewis1996}{
      \BIB{#2}{A.S. Lewis}
         {Group invariance and convex matrix analysis}
         {\jRN{SIAMJMAA}}{17}{1996}{927--949}{#1}}
   \ITEE{#3}{C-KLi,N-KTsing1991}{
      \BIB{#2}{C.-K. Li and N.-K. Tsing}
         {$G$\hyp{}invariant norms and $G(c)$\hyp{}radii}
         {\jRN{LAA}}{150}{1991}{179--194}{#1}}
   \ITEE{#3}{AJLazar,JLindenstrauss1971}{
      \BIB{#2}{A.J. Lazar and J. Lindenstrauss}
         {Banach spaces whose duals are $L_1$ spaces and their representing matrices}
         {\jRN{ActaM}}{126}{1971}{165--193}{#1}}
   \ITEE{#3}{SGLee1980}{
      \BIB{#2}{S.G. Lee}
         {Remarks on reducing operator valued spectrum}
         {\jRN{JKoreanMS}}{16}{1980}{131--136}{#1}}
   \ITEE{#3}{EHLieb,MLoss1997}{
      \BIb{#2}{E.H. Lieb and M. Loss}
         {Analysis \textup{(Graduate Studies in Mathematics, vol. 14)}}
         {Amer. Math. Soc., Providence, RI}{1997}{#1}}
   \ITEE{#3}{HLin2001}{
      \BIB{#2}{H. Lin}
         {Residually finite\hyp{}dimensional and AF\hyp{}embeddable $\CCc^*$\hyp{}algebras}
         {\jRN{PAMS}}{129}{2001}{1689--1696}{#1}}
   \ITEE{#3}{ALindenbaum1926}{
      \BIB{#2}{A. Lindenbaum}
         {Contributions \`{a} l'\'{e}tude de l'espace m\'{e}trique I}
         {\jRN{FM}}{8}{1926}{209--222}{#1}}
   \ITEE{#3}{DLindenstrauss,LTzafriri1971}{
      \BIB{#2}{D. Lindenstrauss and L. Tzafriri}
         {On the complemented subspaces problem}
         {\jRN{IsraelJM}}{9}{1971}{263--269}{#1}}
   \ITEE{#3}{RILoebl1986}{
      \BIB{#2}{R.I. Loebl}
         {A note on containment of operators}
         {\jRN{BAustrMS}}{33}{1986}{279--291}{#1}}
   \ITEE{#3}{LHLoomis1945}{
      \BIB{#2}{L.H. Loomis}
         {Abstract congruence and the uniqueness of Haar measure}
         {\jRN{AnnM}}{46}{1945}{348--355}{#1}}
   \ITEE{#3}{LHLoomis1949}{
      \BIB{#2}{L.H. Loomis}
         {Haar measure in uniform structures}
         {\jRN{DMJ}}{16}{1949}{193--208}{#1}}
   \ITEE{#3}{ERLorch1939}{
      \BIB{#2}{E.R. Lorch}
         {Bicontinuous linear transformation in certain vector spaces}
         {\jRN{BAMS}}{45}{1939}{564--569}{#1}}
   \ITEE{#3}{KLowner1934}{
      \BIB{#2}{K. L\"{o}wner}
         {\"{U}ber monotone Matrixfunctionen}
         {\jRN{MZ}}{38}{1934}{177--216}{#1}}
   \ITEE{#3}{FLuft1968}{
      \BIB{#2}{F. Luft}{The two-sided 
         closed ideals of the algebra of bounded linear operators of a Hilbert space}
         {\jRN{CzMJ}}{18}{1968}{595--605}{#1}}
   \ITEE{#3}{ATLundell,SWeingram1969}{
      \BIb{#2}{A.T. Lundell and S. Weingram}
         {The topology of CW\hyp{}complexes}
         {Litton Educ. Publ.}{1969}{#1}}
   \ITEE{#3}{WLusky1976}{
      \BIB{#2}{W. Lusky}
         {The Gurarij spaces are unique}
         {\jRN{ArchM}}{27}{1976}{627--635}{#1}}
   \ITEE{#3}{WLusky1977}{
      \BIB{#2}{W. Lusky}
         {On separable Lindenstrauss spaces}
         {\jRN{JFA}}{26}{1977}{103--120}{#1}}
   \ITEE{#3}{DMaharam1942}{
      \BIB{#2}{D. Maharam}
         {On homogeneous measure algebras}
         {\jRN{PNAS}}{28}{1942}{108--111}{#1}}
   \ITEE{#3}{MMalicki,SSolecki2009}{
      \BIB{#2}{M. Malicki and S. Solecki}
         {Isometry groups of separable metric spaces}
         {\jRN{MProcCambPhS}}{146}{2009}{67--81}{#1}}
   \ITEE{#3}{PMankiewicz1972}{
      \BIB{#2}{P. Mankiewicz}
         {On extension of isometries in normed linear spaces}
         {\jRN{BAPolSSSM}}{20}{1972}{367--371}{#1}}
   \ITEE{#3}{AManoussos,PStrantzalos2003}{
      \BIB{#2}{A. Manoussos and P. Strantzalos}
         {On the group of isometries on a locally compact metric space}
         {\jRN{JLieTh}}{13}{2003}{7--12}{#1}}
   \ITEE{#3}{JMartinezMaurica,MTPellon1987}{
      \BIB{#2}{J. Martinez\hyp{}Maurica and M.T. Pell\'{o}n}
         {Non\hyp{}archimedean Chebyshev centers}
         {\jRN{IndagMP}}{90}{1987}{417--421}{#1}}
   \ITEE{#3}{KMaurin1980}{
      \BIb{#2}{K. Maurin}
         {Analysis, Part II}
         {D. Reidel, Dordrecht\hyp{}Boston\hyp{}London}{1980}{#1}}
   \ITEE{#3}{EMayer-Wolf1981}{
      \BIB{#2}{E. Mayer-Wolf}
         {Isometries between Banach spaces of Lipschitz functions}
         {\jRN{IsraelJM}}{38}{1981}{58--74}{#1}}
   \ITEE{#3}{SMazur,SUlam1932}{
      \BIB{#2}{S. Mazur and S. Ulam}
         {Sur les transformationes isom\'{e}triques d'espaces vectoriels norm\'{e}s}
         {\jRN{CRASParis}}{194}{1932}{946--948}{#1}}
   \ITEE{#3}{SMazurkiewicz1920}{
      \BIB{#2}{S. Mazurkiewicz}
         {Sur les lignes de Jordan}
         {\jRN{FM}}{1}{1920}{166--209}{#1}}
   \ITEE{#3}{SMazurkiewicz,WSierpinski1920}{
      \BIB{#2}{S. Mazurkiewicz and W. Sierpi\'{n}ski}
         {Contributions a la topologie des ensembles denombrables}
         {\jRN{FM}}{1}{1920}{17--27}{#1}}
   \ITEE{#3}{MMbekhta1992}{
      \BIB{#2}{M. Mbekhta}
         {Sur la structure des composantes connexes semi\hyp{}Fredholm de $B(H)$}
         {\jRN{PAMS}}{116}{1992}{521--524}{#1}}
   \ITEE{#3}{JEMcCarthy1996}{
      \BIB{#2}{J.E. McCarthy}
         {Boundary values and Cowen\hyp{}Douglas curvature}
         {\jRN{JFA}}{137}{1996}{1--18}{#1}}
   \ITEE{#3}{JMelleray2007}{
      \BIB{#2}{J. Melleray}
         {Computing the complexity of the relation of isometry between separable Banach spaces}
         {\jRN{MLQ}}{53}{2007}{128--131}{#1}}
   \ITEE{#3}{JMelleray2007a}{
      \BIB{#2}{J. Melleray}
         {On the geometry of Urysohn's universal metric space}
         {\jRN{TopA}}{154}{2007}{384--403}{#1}}
   \ITEE{#3}{JMelleray2008}{
      \BIB{#2}{J. Melleray}
         {Some geometric and dynamical properties of the Urysohn space}
         {\jRN{TopA}}{155}{2008}{1531--1560}{#1}}
   \ITEE{#3}{JMelleray2008a}{
      \BIB{#2}{J. Melleray}
         {Compact metrizable groups are isometry groups of compact metric spaces}
         {\jRN{PAMS}}{136}{2008}{1451--1455}{#1}}
   \ITEE{#3}{JMelleray,FVPetrov,AMVershik2008}{
      \BIB{#2}{J. Melleray, F.V. Petrov, A.M. Vershik}
         {Linearly rigid metric spaces and the embedding problem}
         {\jRN{FM}}{199}{2008}{177--194}{#1}}
   \ITEE{#3}{EMichael1953}{
      \BIB{#2}{E. Michael}
         {Some extension theorems for continuous functions}
         {\jRN{PacJM}}{3}{1953}{789--806}{#1}}
   \ITEE{#3}{EMichael1954}{
      \BIB{#2}{E. Michael}
         {Local properties of topological spaces}
         {\jRN{DMJ}}{21}{1954}{163--171}{#1}}
   \ITEE{#3}{EMichael1956}{
      \BIB{#2}{E. Michael}
         {Selected selection theorems}
         {\jRN{AmMMon}}{58}{1956}{233--238}{#1}}
   \ITEE{#3}{EMichael1956a}{
      \BIB{#2}{E. Michael}
         {Continuous selections. I}
         {\jRN{AnnM}}{63}{1956}{361--382}{#1}}
   \ITEE{#3}{EMichael1956b}{
      \BIB{#2}{E. Michael}
         {Continuous selections. II}
         {\jRN{AnnM}}{64}{1956}{562--580}{#1}}
   \ITEE{#3}{EMichael1959}{
      \BIB{#2}{E. Michael}
         {A theorem on semi\hyp{}continuous set\hyp{}valued functions}
         {\jRN{DMJ}}{26}{1959}{647--652}{#1}}
   \ITEE{#3}{EMichael1964}{
      \BIB{#2}{E. Michael}
         {A short proof of the Arens-Eells embedding theorem}
         {\jRN{PAMS}}{15}{1964}{415--416}{#1}}
   \ITEE{#3}{JVanMill1986}{
      \BIB{#2}{J. van Mill}
         {Another counterexample in ANR theory}
         {\jRN{PAMS}}{97}{1986}{136--138}{#1}}
   \ITEE{#3}{JVanMill2001}{
      \BIb{#2}{J. van Mill}
         {The Infinite\hyp{}Dimensional Topology of Function Spaces 
         \textup{(North\hyp{}Holland Mathematical Library, vol. 64)}}
         {Elsevier, Amsterdam}{2001}{#1}}
   \ITEE{#3}{KMine2006}{
      \BIB{#2}{K. Mine}
         {Universal spaces of non\hyp{}separable absolute Borel classes}
         {\jRN{TsukubaJM}}{30}{2006}{137--148}{#1}}
   \ITEE{#3}{WMlak1991}{
      \BIb{#2}{W. Mlak}{Hilbert Spaces and Operator Theory}
         {PWN --- Polish Scientific Publishers and Kluwer Academic Publishers, 
         Warszawa\hyp{}Dordrecht}{1991}{#1}}
   \ITEE{#3}{JMogilski1979}{
      \BIB{#2}{J. Mogilski}
         {$CE$\hyp{}decomposition of $\ell_2$\hyp{}manifolds}
         {\jRN{BAPolSSSM}}{27}{1979}{309--314}{#1}}
   \ITEE{#3}{RLMoore1916}{
      \BIB{#2}{R.L. Moore}
         {On the foundations of plane analysis situs}
         {\jRN{TAMS}}{17}{1916}{131--164}{#1}}
   \ITEE{#3}{KMorita1955}{
      \BIB{#2}{K. Morita}
         {A condition for the metrizability of topological spaces and for $n$\hyp{}dimensionality}
         {\jRN{SciRepTokyoA}}{5}{1955}{33--36}{#1}}
   \ITEE{#3}{AMukherjea,NATserpes1976}{
      \BIb{#2}{A. Mukherjea and N.A. Tserpes}
         {Measures on topological semigroups}
         {Springer Lecture Notes in Math. Vol. 547, Berlin}{1976}{#1}}
   \ITEE{#3}{JMycielski1974}{
      \BIB{#2}{J. Mycielski}
         {Remarks on invariant measures in metric spaces}
         {\jRN{CollM}}{32}{1974}{105--112}{#1}}
   \ITEE{#3}{SNNaboko1984}{
      \BIB{#2}{S.N. Naboko}
         {Conditions for similarity to unitary and selfadjoint operators}
         {\jRN{FunkAnalPril}}{18}{1984}{16--27}{#1}}
   \ITEE{#3}{LNachbin1965}{
      \BIb{#2}{L. Nachbin}{The Haar Integral}
         {D. Van Nostrand Company, Inc., 
         Princeton\hyp{}New Jersey\hyp{}Toronto\hyp{}New York\hyp{}London}{1965}{#1}}
   \ITEE{#3}{TDNarang,SKGarg1991}{
      \BIB{#2}{T.D. Narang and S.K. Garg}
         {On the uniqueness of best approximation in non\hyp{}archimedian spaces}
         {\jRN{PeriodMHung}}{22}{1991}{121--124}{#1}}
   \ITEE{#3}{JVonNeumann1930}{
      \BIB{#2}{J. von Neumann}
         {Zur Algebra der Funktionaloperationen und Theorie der normalen Operatoren}
         {\jRN{MAnn}}{102}{1930}{370--427}{#1}}
   \ITEE{#3}{JVonNeumann1934}{
      \BIB{#2}{J. von Neumann}
         {Zum Haarschen Mass in topologischen Gruppen}
         {\jRN{ComposM}}{1}{1934}{106--114}{#1}}
   \ITEE{#3}{JVonNeumann1937}{
      \BiB{#2}{J. von Neumann}{Some matrix\hyp{}inequalities 
         and metrization of matrix\hyp{}space}{\jRN{TomskUnivRev}{} \textbf{1} (1937), 
         286--300; in }{Collected Works}{Pergamon, New York}{1962}{Vol. 4, 205--219}{#1}}
   \ITEE{#3}{JVonNeumann1949}{
      \BIB{#2}{J. von Neumann}
         {On Rings of Operators. Reduction Theory}
         {\jRN{AnnM}}{50}{1949}{401--485}{#1}}
   \ITEE{#3}{ONielson1973}{
      \BIB{#2}{O. Nielson}
         {Borel sets of von Neumann algebras}
         {\jRN{AmJM}}{95}{1973}{145--164}{#1}}
   \ITEE{#3}{pn2002}{\bibITEM{#2}{#1} \mypaplist{pn1}{}}
   \ITEE{#3}{pn2006a}{\bibITEM{#2}{#1} \mypaplist{pn2}{}}
   \ITEE{#3}{pn2006b}{\bibITEM{#2}{#1} \mypaplist{pn3}{}}
   \ITEE{#3}{pn2007}{\bibITEM{#2}{#1} \mypaplist{pn4}{}}
   \ITEE{#3}{pn2008a}{\bibITEM{#2}{#1} \mypaplist{pn5}{}}
   \ITEE{#3}{pn2008b}{\bibITEM{#2}{#1} \mypaplist{pn6}{}}
   \ITEE{#3}{pn2009a}{\bibITEM{#2}{#1} \mypaplist{pn7}{}}
   \ITEE{#3}{pn2009b}{\bibITEM{#2}{#1} \mypaplist{pn8}{}}
   \ITEE{#3}{pn2009c}{\bibITEM{#2}{#1} \mypaplist{pn9}{}}
   \ITEE{#3}{pn2010a}{\bibITEM{#2}{#1} \mypaplist{pn10}{}}
   \ITEE{#3}{pn2010b}{\bibITEM{#2}{#1} \mypaplist{pn11}{}}
   \ITEE{#3}{pn2011a}{\bibITEM{#2}{#1} \mypaplist{pn12}{}}
   \ITEE{#3}{pn2011b}{\bibITEM{#2}{#1} \mypaplist{pn13}{}}
   \ITEE{#3}{pn2011c}{\bibITEM{#2}{#1} \mypaplist{pn14}{}}
   \ITEE{#3}{pn2011d}{\bibITEM{#2}{#1} \mypaplist{pn15}{}}
   \ITEE{#3}{pn2011e}{\bibITEM{#2}{#1} \mypaplist{pn16}{}}
   \ITEE{#3}{pn2011f}{\bibITEM{#2}{#1} \mypaplist{pn17}{}}
   \ITEE{#3}{pn2012a-}{\bibITEM{#2}{#1} \mypaplist[*]{pn18}{}}
   \ITEE{#3}{pn2012a}{\bibITEM{#2}{#1} \mypaplist{pn18}{}}
   \ITEE{#3}{pn2012b}{\bibITEM{#2}{#1} \mypaplist{pn19}{}}
   \ITEE{#3}{pn2012c}{\bibITEM{#2}{#1} \mypaplist{pn20}{}}
   \ITEE{#3}{pn2012d}{\bibITEM{#2}{#1} \mypaplist{pn21}{}}
   \ITEE{#3}{pn2012e}{\bibITEM{#2}{#1} \mypaplist{pn22}{}}
   \ITEE{#3}{pn2012f}{\bibITEM{#2}{#1} \mypaplist{pn23}{}}
   \ITEE{#3}{pnXXXXb}{
      \bibITEM{#2}{#1} \mypaplist{pnX2}{}}
   \ITEE{#3}{pnXXXXc}{
      \bibITEM{#2}{#1} \mypaplist{pnX3}{}}
   \ITEE{#3}{pnXXXXd}{
      \bibITEM{#2}{#1} \mypaplist{pnX17}{}}
   \ITEE{#3}{MNiezgoda1998}{
      \BIB{#2}{M. Niezgoda}
         {Group majorization and Schur type inequalities}
         {\jRN{LAA}}{268}{1998}{9--30}{#1}}
   \ITEE{#3}{MNiezgoda1998a}{
      \BIB{#2}{M. Niezgoda}{An analytical 
         characterization of effective and of irreducible groups inducing cone orderings}
         {\jRN{LAA}}{269}{1998}{105--114}{#1}}
   \ITEE{#3}{MNiezgoda,TYTam2001}{
      \BIB{#2}{M. Niezgoda and T.Y. Tam}
         {On norm property of $G(c)$\hyp{}radii and Eaton triples}
         {\jRN{LAA}}{336}{2001}{119--130}{#1}}
   \ITEE{#3}{LNNikolskaya1974}{
      \BIB{#2}{L.N. Nikol'skaya}{Structure of the point spectrum 
         of a linear operator \textup{(Russian)}}{\jRN{MZamet}}{15}{1974}{149--158; 
         English translation in \jRN{MNotes} \textbf{15} (1974), 83--87}{#1}}
   \ITEE{#3}{APazy1983}{
      \BIb{#2}{A. Pazy}{Semigroups of Linear Operators and Applications 
         to Partial Differential Equations \textup{(Applied Mathematical Sciences, vol. 44)}}
         {Springer\hyp{}Verlag, New York}{1983}{#1}}
   \ITEE{#3}{CPearcy,NSalinas1974}{
      \BIB{#2}{C. Pearcy and N. Salinas}
         {Finite-dimensional representations of separable $\CCc^*$-algebras}
         {\jRN{NAMS}}{21}{1974}{A-376}{#1}}
   \ITEE{#3}{APelc1982}{
      \BIB{#2}{A. Pelc}
         {Semiregular invariant measures on abelian groups}
         {\jRN{PAMS}}{86}{1982}{423--426}{#1}}
   \ITEE{#3}{RPenrose1955}{
      \BIB{#2}{R. Penrose}
         {A generalized inverse for matrices}
         {\jRN{ProcCambPhS}}{51}{1955}{406--413}{#1}}
   \ITEE{#3}{VPestov2006}{
      \BIb{#2}{V. Pestov}{Dynamics of infinite\hyp{}dimensional 
         groups. The Ramsey\hyp{}Dvoretzky\hyp{}Milman phenomenon}
         {University Lecture Series \textbf{40}, AMS, Providence, RI}{2006}{#1}}
   \ITEE{#3}{VPestov2007}{
      \BiB{#2}{V. Pestov}{Forty\hyp{}plus annotated 
         questions about large topological groups}{in:}{Open Problems in Topology II}
         {Elliot Pearl (editor), Elsevier B.V., Amsterdam}{2007}{439--450}{#1}}
   \ITEE{#3}{PVPetersen1993}{
      \BiB{#2}{P.V. Petersen}{Gromov\hyp{}Hausdorff convergence 
         of metric spaces}{in book:}{Differential Geometry: Riemannian Geometry 
         (Los Angeles, CA, 1990)}{Amer. Math. Soc., Providence, RI}{1993}{489--504}{#1}}
   \ITEE{#3}{HPfister1985}{
      \BIB{#2}{H. Pfister}
         {Continuity of the inverse}
         {\jRN{PAMS}}{95}{1985}{312--314}{#1}}
   \ITEE{#3}{LPontrjagin1946}{
      \BIb{#2}{L. Pontrjagin}
         {Topological Groups}
         {Princeton University Press, Princeton}{1946}{#1}}
   \ITEE{#3}{DRamachandran,MMisiurewicz1982}{
      \BIB{#2}{D. Ramachandran and M. Misiurewicz}
         {Hopf's theorem on invariant measures for a group of transformations}
         {\jRN{SM}}{74}{1982}{183--189}{#1}}
   \ITEE{#3}{WRoelcke,SDierolf1981}{
      \BIb{#2}{W. Roelcke and S. Dierolf}
         {Uniform Structures on Topological Groups and Their Quotients}
         {McGraw Hill, New York}{1981}{#1}}
   \ITEE{#3}{JMRosenblatt1974}{
      \BIB{#2}{J.M. Rosenblatt}
         {Equivalent invariant measures}
         {\jRN{IsraelJM}}{17}{1974}{261--270}{#1}}
   \ITEE{#3}{SRosset1976}{
      \BIB{#2}{S. Rosset}
         {A new proof of the Amitsur-Levitski identity}
         {\jRN{IsraelJM}}{23}{1976}{187--188}{#1}}
   \ITEE{#3}{HLRoyden1963}{
      \BIb{#2}{H.L. Royden}
         {Real Analysis}
         {The Macmillan Co., New York}{1963}{#1}}
   \ITEE{#3}{WRudin1962}{
      \BIb{#2}{W. Rudin}{Fourier Analysis on Groups 
         \textup{(Interscience Tracts in Pure and Applied Mathematics, Number 12)}}
         {Interscience Publishers, New York}{1962}{#1}}
   \ITEE{#3}{WRudin1991}{
      \BIb{#2}{W. Rudin}
         {Functional Analysis}
         {McGraw\hyp{}Hill Science}{1991}{#1}}
   \ITEE{#3}{TSaito1972}{
      \BiB{#2}{T. Sait\^{o}}{Generations of von Neumann algebras}
         {Lecture Notes in Math. vol. 247}{\textup{(}Lecture on Operator Algebras\textup{)}}
         {Springer, Berlin\hyp{}Heidelberg\hyp{}New York}{1972}{435--531}{#1}}
   \ITEE{#3}{KSakai,MYaguchi2003}{
      \BIB{#2}{K. Sakai and M. Yaguchi}{Characterizing 
         manifolds modeled on certain dense subspaces of non\hyp{}separable Hilbert spaces}
         {\jRN{TsukubaJM}}{27}{2003}{143--159}{#1}}
   \ITEE{#3}{SSakai1971}{
      \BIb{#2}{S. Sakai}
         {$\CCc^*$\hyp{}Algebras and $\WWw^*$\hyp{}Algebras}
         {Springer\hyp{}Verlag, Berlin\hyp{}Heidelberg\hyp{}New York}{1971}{#1}}
   \ITEE{#3}{RSchori1971}{
      \BIB{#2}{R. Schori}
         {Topological stability for infinite\hyp{}dimensional manifolds}
         {\jRN{ComposM}}{23}{1971}{87--100}{#1}}
   \ITEE{#3}{JTSchwartz1967}{
      \BIb{#2}{J.T. Schwartz}
         {$\WWw^*$\hyp{}algebras}
         {Gordon and Breach, Science Publishers Inc., New York\hyp{}London\hyp{}Paris}{1967}{#1}}
   \ITEE{#3}{ZSemadeni1971}{
      \BIb{#2}{Z. Semadeni}
         {Banach Spaces of Continuous Functions (Vol. I)}
         {\jRN{PWN}}{1971}{#1}}
   \ITEE{#3}{JPSerre1951}{
      \BIB{#2}{J.-P. Serre}
         {Homologie singuli\`{e}re des espaces fibr\'{e}s}
         {\jRN{AnnM}}{54}{1951}{425--505}{#1}}
   \ITEE{#3}{DSherman2007}{
      \BIB{#2}{D. Sherman}
         {On the dimension theory of von Neumann algebras}
         {\jRN{MScand}}{101}{2007}{123--147}{#1}}
   \ITEE{#3}{DSherman2007a}{
      \BIB{#2}{D. Sherman}
         {Unitary orbits of normal operators in von Neumann algebras}
         {\jRN{JReinAngM}}{605}{2007}{95--132}{#1}}
   \ITEE{#3}{SAShkarin1999}{
      \BIB{#2}{S.A. Shkarin}
         {Universal Abelian topological groups}
         {\jRN{SbM}}{190}{1999}{1059--1076}{#1}}
   \ITEE{#3}{WSierpinski1920}{
      \BIB{#2}{W. Sierpi\'{n}ski}
         {Sur une propri\'{e}t\'{e} topologique des ensembles d\'{e}nombrables denses en soi}
         {\jRN{FM}}{1}{1920}{11--16}{#1}}
   \ITEE{#3}{WSierpinski1928}{
      \BIB{#2}{W. Sierpi\'{n}ski}
         {Sur les projections des ensembles compl\'{e}mentaires aux ensembles \textup{(A)}}
         {\jRN{FM}}{11}{1928}{117--122}{#1}}
   \ITEE{#3}{MSlocinski1980}{
      \BIB{#2}{M. S\l{}oci\'{n}ski}
         {On the Wold\hyp{}type decomposition of a pair of commuting isometries}
         {\jRN{APM}}{37}{1980}{255--262}{#1}}
   \ITEE{#3}{OGSmolyanov,SAShkarin1999}{
      \BiB{#2}{O.G. Smolyanov and S.A. Shkarin}{On the structure of spectra of linear operators 
         in Hilbert space}{in:}{Selected questions in mathematics, mechanics, and their 
         applications}{Moscow State University, Moscow}{1999}{289-–302}{#1}}
   \ITEE{#3}{OGSmolyanov,SAShkarin2001}{
      \BIB{#2}{O.G. Smolyanov and S.A. Shkarin}{On the structure of the spectra 
         of linear operators in Banach spaces \textup{(Russian)}}{\jRN{MSb}}{192}
         {2001}{99--114}{#1} English translation: \jRN{SbM} \textbf{192} (2001), 577--591.}
   \ITEE{#3}{RCSteinlage1975}{
      \BIB{#2}{R.C. Steinlage}
         {On Haar measure in locally compact $T_2$ spaces}
         {\jRN{AmJM}}{97}{1975}{291--307}{#1}}
   \ITEE{#3}{JStochel,FHSzafraniec1989}{
      \BIB{#2}{J. Stochel and F.H. Szafraniec}
         {On normal extensions of unbounded operators. III. Spectral properties}
         {\jRN{PublRIMSKyoto}}{25}{1989}{105--139}{#1}}
   \ITEE{#3}{JStochel,FHSzafraniec1989a}{
      \BIB{#2}{J. Stochel and F.H. Szafraniec}
         {The normal part of an unbounded operator}
         {\jRN{ProcKonink}}{92}{1989}{495--503}{#1}}
   \ITEE{#3}{AHStone1962}{
      \BIB{#2}{A.H. Stone}
         {Absolute $\FFf_{\sigma}$\hyp{}spaces}
         {\jRN{PAMS}}{13}{1962}{495--499}{#1}}
   \ITEE{#3}{AHStone1962a}{
      \BIB{#2}{A.H. Stone}
         {Non\hyp{}separable Borel sets}
         {\jRN{DissM}}{28}{1962}{41 pages}{#1}}
   \ITEE{#3}{AHStone1972}{
      \BIB{#2}{A.H. Stone}
         {Non\hyp{}separable Borel sets II}
         {\jRN{GTopA}}{2}{1972}{249--270}{#1}}
   \ITEE{#3}{MHStone1937}{
      \BIB{#2}{M.H. Stone}
         {Application of the theory of Boolean rings to general topology}
         {\jRN{TAMS}}{41}{1937}{375--481}{#1}}
   \ITEE{#3}{MHStone1948}{
      \BIB{#2}{M.H. Stone}
         {The generalized Weierstrass approximation theorem}
         {\jRN{MMag}}{21}{1948}{167--184}{#1}}
   \ITEE{#3}{RAStruble1974}{
      \BIB{#2}{R.A. Struble}
         {Metrics in locally compact groups}
         {\jRN{ComposM}}{28}{1974}{217--222}{#1}}
   \ITEE{#3}{RGSwan1963}{
      \BIB{#2}{R.G. Swan}
         {An application of graph theory to algebra}
         {\jRN{PAMS}}{14}{1963}{367--373}{#1}}
   \ITEE{#3}{RGSwan1969}{
      \BIB{#2}{R.G. Swan}
         {Correction to ``An application of graph theory to algebra''}
         {\jRN{PAMS}}{21}{1969}{379--380}{#1}}
   \ITEE{#3}{BSz-Nagy1947}{
      \BIB{#2}{B. Sz.\hyp{}Nagy}
         {On uniformly bounded linear transformations in Hilbert space}
         {\jRN{ActaSM}}{11}{1947}{152--157}{#1}}
   \ITEE{#3}{WSzymanski1974}{
      \BIB{#2}{W. Szyma\'{n}ski}
         {Decompositions of operator-valued functions in Hilbert spaces}
         {\jRN{SM}}{50}{1974}{265--280}{#1}}
   \ITEE{#3}{WTakahashi1970}{
      \BIB{#2}{W. Takahashi}
         {A convexity in metric space and nonexpansive mappings, I}
         {\jRN{KodaiMSemRep}}{22}{1970}{142--149}{#1}}
   \ITEE{#3}{MTakesaki2002}{
      \BIb{#2}{M. Takesaki}{Theory 
         of Operator Algebras I \textup{(Encyclopaedia of Mathematical Sciences, Volume 124)}}
         {Springer\hyp{}Verlag, Berlin\hyp{}Heidelberg\hyp{}New York}{2002}{#1}}
   \ITEE{#3}{MTakesaki2003}{
      \BIb{#2}{M. Takesaki}{Theory 
         of Operator Algebras II \textup{(Encyclopaedia of Mathematical Sciences, Volume 125)}}
         {Springer\hyp{}Verlag, Berlin\hyp{}Heidelberg\hyp{}New York}{2003}{#1}}
   \ITEE{#3}{MTakesaki2003a}{
      \BIb{#2}{M. Takesaki}{Theory 
         of Operator Algebras III \textup{(Encyclopaedia of Mathematical Sciences, Volume 127)}}
         {Springer\hyp{}Verlag, Berlin\hyp{}Heidelberg\hyp{}New York}{2003}{#1}}
   \ITEE{#3}{TYTam1999}{
      \BIB{#2}{T.Y. Tam}
         {An extension of a result of Lewis}
         {\jRN{ELA}}{5}{1999}{1--10}{#1}}
   \ITEE{#3}{TYTam2000}{
      \BIB{#2}{T.Y. Tam}
         {Group majorization, Eaton triples and numerical range}
         {\jRN{LMLA}}{47}{2000}{11--28}{#1}}
   \ITEE{#3}{TYTam2002}{
      \BIB{#2}{T.Y. Tam}
         {Generalized Schur\hyp{}concave functions and Eaton triples}
         {\jRN{LMLA}}{50}{2002}{113--120}{#1}}
   \ITEE{#3}{TYTam,WCHill2001}{
      \BIB{#2}{T.Y. Tam and W.C. Hill}
         {On $G$\hyp{}invariant norms}
         {\jRN{LAA}}{331}{2001}{101--112}{#1}}
   \ITEE{#3}{AFTiman,IAVestfrid1983}{
      \BIB{#2}{A.F. Timan and I.A. Vestfrid}
         {Any separable ultrametric space can be isometrically imbedded in $\ell_2$}
         {\jRN{FAA}}{17}{1983}{70--71}{#1}}
   \ITEE{#3}{VTimofte2005}{
      \BIB{#2}{V. Timofte}
         {Stone\hyp{}Weierstrass theorems revisited}
         {\jRN{JAT}}{136}{2005}{45--59}{#1}}
   \ITEE{#3}{JTomiyama1958}{
      \BIB{#2}{J. Tomiyama}
         {Generalized dimension function for $\WWw^*$\hyp{}algebras of infinite type}
         {\jRN{TohokuMJ} (2)}{10}{1958}{121--129}{#1}}
   \ITEE{#3}{HTorunczyk1970}{
      \BIB{#2}{H. Toru\'{n}czyk}
         {Remarks on Anderson's paper ``On topological infinite deficiency''}
         {\jRN{FM}}{66}{1970}{393--401}{#1}}
   \ITEE{#3}{HTorunczyk1970a}{
      \BIb{#2}{H. Toru\'{n}czyk}
         {$G$\hyp{}$K$\hyp{}absorbing and skeletonized sets in metric spaces}
         {Ph.D. thesis, Inst. Math. Polish Acad. Sci., Warszawa}{1970}{#1}}
   \ITEE{#3}{HTorunczyk1972}{
      \BIB{#2}{H. Toru\'{n}czyk}
         {A short proof of Hausdorff's theorem on extending metrics}
         {\jRN{FM}}{77}{1972}{191--193}{#1}}
   \ITEE{#3}{HTorunczyk1974}{
      \BIB{#2}{H. Toru\'{n}czyk}
         {Absolute retracts as factors of normed linear spaces}
         {\jRN{FM}}{86}{1974}{53--67}{#1}}
   \ITEE{#3}{HTorunczyk1975}{
      \BIB{#2}{H. Toru\'{n}czyk}
         {On Cartesian factors and the topological classification of linear metric spaces}
         {\jRN{FM}}{88}{1975}{71--86}{#1}}
   \ITEE{#3}{HTorunczyk1978}{
      \BIB{#2}{H. Toru\'{n}czyk}{Concerning 
         locally homotopy negligible sets and characterization of $\ell_2$\hyp{}manifolds}
         {\jRN{FM}}{101}{1978}{93--110}{#1}}
   \ITEE{#3}{HTorunczyk1980}{
      \BiB{#2}{H. Toru\'{n}czyk}{Characterization of infinite\hyp{}dimensional manifolds}{in:}
         {Proceedings of the International Conference on Geometric Topology (Warsaw, 1978)}
         {\jRN{PWN}}{1980}{431--437}{#1}}
   \ITEE{#3}{HTorunczyk1981}{
      \BIB{#2}{H. Toru\'{n}czyk}
         {Characterizing Hilbert space topology}
         {\jRN{FM}}{111}{1981}{247--262}{#1}}
   \ITEE{#3}{HTorunczyk1985}{
      \BIB{#2}{H. Toru\'{n}czyk}
         {A correction of two papers concerning Hilbert manifolds}
         {\jRN{FM}}{125}{1985}{89--93}{#1}}
   \ITEE{#3}{KTsuda1985}{
      \BIB{#2}{K. Tsuda}
         {A note on closed embeddings of finite dimensional metric spaces}
         {\jRN{BLondMS}}{17}{1985}{273--278}{#1}}
   \ITEE{#3}{PSUrysohn1925}{
      \BIB{#2}{P.S. Urysohn}
         {Sur un espace m\'{e}trique universel}
         {\jRN{CRASParis}}{180}{1925}{803--806}{#1}}
   \ITEE{#3}{PSUrysohn1927}{
      \BIB{#2}{P.S. Urysohn}
         {Sur un espace m\'{e}trique universel}
         {\jRN{BullSM}}{51}{1927}{43--64, 74--96}{#1}}
   \ITEE{#3}{VVUspenskij1986}{
      \BIB{#2}{V.V. Uspenskij}
         {A universal topological group with a countable basis}
         {\jRN{FAA}}{20}{1986}{86--87}{#1}}
   \ITEE{#3}{VVUspenskij1990}{
      \BIB{#2}{V.V. Uspenskij}
         {On the group of isometries of the Urysohn universal metric space}
         {\jRN{CMUC}}{31}{1990}{181--182}{#1}}
   \ITEE{#3}{VVUspenskij2004}{
      \BIB{#2}{V.V. Uspenskij}
         {The Urysohn universal metric space is homeomorphic to a Hilbert space}
         {\jRN{TopA}}{139}{2004}{145--149}{#1}}
   \ITEE{#3}{VVUspenskij2008}{
      \BIB{#2}{V.V. Uspenskij}
         {On subgroups of minimal topological groups}
         {\jRN{TopA}}{155}{2008}{1580--1606}{#1}}
   \ITEE{#3}{VSVaradarajan1963}{
      \BIB{#2}{V.S. Varadarajan}
         {Groups of automorphisms of Borel spaces}
         {\jRN{TAMS}}{109}{1963}{191--220}{#1}}
   \ITEE{#3}{AMVershik1998}{
      \BIB{#2}{A.M. Vershik}{The universal Urysohn space, 
         Gromov's metric triples, and random metrics on the series of natural numbers}
         {\jRN{UspekhiMN}}{53}{1998}{57--64}{#1} English translation: \jRN{RussMS}{} \textbf{53} 
         (1998), 921--928. Correction: \jRN{UspekhiMN}{} \textbf{56} (2001), p. 207. English 
         translation: \jRN{RussMS}{} \textbf{56} (2001), p. 1015.}
   \ITEE{#3}{AMVershik2002}{
      \BIb{#2}{A.M. Vershik}{Random metric spaces and the universal Urysohn space}
         {Fundamental Mathematics Today. 10th anniversary of the Independent Moscow University. 
         MCCME Publ.}{2002}{#1}}
   \ITEE{#3}{NWeaver1999}{
      \BIb{#2}{N. Weaver}
         {Lipschitz Algebras}
         {World Scientific}{1999}{#1}}
   \ITEE{#3}{JWeidmann1980}{
      \BIb{#2}{J. Weidmann}
         {Linear Operators in Hilbert Spaces}
         {(Graduate Texts in Mathematics, vol. 68) Springer\hyp{}Verlag New York Inc.}{1980}{#1}}
   \ITEE{#3}{JEWest1969}{
      \BIB{#2}{J.E. West}
         {Approximating homotopies by isotopies in Fr\'{e}chet manifolds}
         {\jRN{BAMS}}{75}{1969}{1254--1257}{#1}}
   \ITEE{#3}{JEWest1969a}{
      \BIB{#2}{J.E. West}
         {Fixed\hyp{}point sets of transformation groups on infinite\hyp{}product spaces}
         {\jRN{PAMS}}{21}{1969}{575--582}{#1}}
   \ITEE{#3}{JEWest1970}{
      \BIB{#2}{J.E. West}
         {The ambient homeomorphy of infinite\hyp{}dimensional Hilbert spaces}
         {\jRN{PacJM}}{34}{1970}{257--267}{#1}}
   \ITEE{#3}{JHCWhitehead1949}{
      \BIB{#2}{J.H.C. Whitehead}
         {Combinatorial homotopy I}
         {\jRN{BAMS}}{55}{1949}{213--245}{#1}}
   \ITEE{#3}{GTWhyburn1942}{
      \BIb{#2}{G. T. Whyburn}
         {Analytic Topology}
         {Amer. Math. Soc. Colloquium Publications (vol. XXVIII), New York}{1942}{#1}}
   \ITEE{#3}{RWilliamson,LJanos1987}{
      \BIB{#2}{R. Williamson and L. Janos}
         {Constructing metrics with the Heine\hyp{}Borel property}
         {\jRN{PAMS}}{100}{1987}{567--573}{#1}}
   \ITEE{#3}{WWogen1969}{
      \BIB{#2}{W. Wogen}
         {On generators for von Neumann algebras}
         {\jRN{BAMS}}{75}{1969}{95--99}{#1}}
   \ITEE{#3}{RYTWong1967}{
      \BIB{#2}{R.Y.T. Wong}
         {On homeomorphisms of certain infinite dimensional spaces}
         {\jRN{TAMS}}{128}{1967}{148--154}{#1}}
   \ITEE{#3}{LYang,JZhang1987}{
      \BIB{#2}{L. Yang and J. Zhang}
         {Average distance constants of some compact convex space}
         {\jRN{JChinUST}}{17}{1987}{17--23}{#1}}
   \ITEE{#3}{PZakrzewski1993}{
      \BIB{#2}{P. Zakrzewski}
         {The existence of invariant $\sigma$\hyp{}finite measures for a group of transformations}
         {\jRN{IsraelJM}}{83}{1993}{275--287}{#1}}
   \ITEE{#3}{PZakrzewski2002}{
      \BIb{#2}{P. Zakrzewski}
         {Measures on Algebraic\hyp{}Topological Structures, Handbook of Measure Thoery}
         {E. Pap, ed., Elsevier, Amsterdam}{2002, 1091--1130}{#1}}
   \ITEE{#3}{WZelazko1960}{
      \BIB{#2}{W. \.{Z}elazko}
         {A theorem on $B_0$ division algebras}
         {\jRN{BPAS}}{8}{1960}{373--375}{#1}}
   \ITEE{#3}{KZhu2000}{
      \BIB{#2}{K. Zhu}
         {Operators in Cowen\hyp{}Douglas classes}
         {\jRN{IllinoisJM}}{44}{2000}{767--783}{#1}}
   \ITEE{#3}{JDieudonne1939}{
      \BIB{#2}{J. Dieudonn\'{e}}
         {Sur les espaces uniformes complets}
         {\jRN{AnnSciEcNormSup}}{56}{1939}{277--291}{#1}}
   }

\newcommand{\mypaplist}[3][]{
   \ITEE{#2}{pn1}{
      \myBIB{Separate and joint similarity to families of normal operators}
         {\jRN{SM}}{149}{2002}{39--62}{#3}}
   \ITEE{#2}{pn2}{
      \myBIB{Locally arcwise connected metrizable spaces with the fixed point property are 
         complete\hyp{}metrizable}{\jRN{TopA}}{153}{2006}{1639--1642}{#3}}
   \ITEE{#2}{pn3}{
      \myBIB{Invariant measures for equicontinuous semigroups of continuous transformations 
         of a compact Hausdorff space}{\jRN{TopA}}{153}{2006}{3373--3382}{#3}}
   \ITEE{#2}{pn4}{
      \myBIB{Approximation of the Hausdorff distance by the distance of continuous surjections}
         {\jRN{TopA}}{154}{2007}{655--664}{#3}}
   \ITEE{#2}{pn5}{
      \myBIB{Generalized Haar integral}
         {\jRN{TopA}}{155}{2008}{1323--1328}{#3}}
   \ITEE{#2}{pn6}{
      \myBIB{Integration and Lipschitz functions}
         {\jRN{RCMP}}{57}{2008}{391--399}{#3}}
   \ITEE{#2}{pn7}{
      \myBIB{Canonical Banach function spaces generated by Urysohn universal spaces. Measures 
         as Lipschitz maps}{\jRN{SM}}{192}{2009}{97--110}{#3}}
   \ITEE{#2}{pn8}{
      \myBIB{Urysohn universal spaces as metric groups of exponent $2$}
         {\jRN{FM}}{204}{2009}{1--6}{#3}}
   \ITEE{#2}{pn9}{
      \myBIB{Central subsets of Urysohn universal spaces}
         {\jRN{CMUC}}{50}{2009}{445--461}{#3}}
   \ITEE{#2}{pn10}{
      \myBIB{Ultra\hyp{}$\mM$\hyp{}separability}
         {\jRN{TopA}}{157}{2010}{669--673}{#3}}
   \ITEE{#2}{pn11}{
      \myBIB{Functor of extension of $\Lambda$\hyp{}isometric maps between central subsets 
         of the unbounded Urysohn universal space}{\jRN{CMUC}}{51}{2010}{541--549}{#3}}
   \ITEE{#2}{pn12}{
      \myBIB[P. Niemiec and T.Y. Tam]{A representation of $G$\hyp{}invariant norms for Eaton 
         triple}{\jRN{JCA}}{18}{2011}{59--65}{#3}}
   \ITEE{#2}{pn13}{
      \myBIB{Topological structure of Urysohn universal spaces}
         {\jRN{TopA}}{158}{2011}{352--359}{#3}}
   \ITEE{#2}{pn14}{
      \myBIB{A note on invariant measures}
         {\jRN{OpusM}}{31}{2011}{425--431}{#3}}
   \ITEE{#2}{pn15}{
      \myBIB{Strengthened Stone\hyp{}Weierstrass type theorem}
         {\jRN{OpusM}}{31}{2011}{645--650}{#3}}
   \ITEE{#2}{pn16}{
      \myBIB{Generalized absolute values and polar decompositions of a bounded operator}
         {\jRN{IEOT}}{71}{2011}{151--160}{#3}}
   \ITEE{#2}{pn17}{
      \myBIB{Functor of extension of contractions on Urysohn universal spaces}
         {\jRN{ACS}}{19}{2011}{959--967}{#3}}
   \ITEE{#2}{pn18}{
      \myBIB{A note on ANR's}{\jRN{TopA}}{159}{2012}{\ITEE{#1}{}{315--321; erratum: p.~2232}
         \ITEE{#1}{*}{315--321}\ITEE{#1}{**}{p.~2232}}{#3}}
   \ITEE{#2}{pn19}{
      \myBIB{Unitary equivalence and decompositions of finite systems of closed densely defined 
         operators in Hilbert spaces}{\jRN{DissM}}{482}{2012}{106~pp}{#3}}
   \ITEE{#2}{pn20}{
      \myBIB{Problem with almost everywhere equality}
         {\jRN{APM}}{104}{2012}{105--108}{#3}}
   \ITEE{#2}{pn21}{
      \myBIB{Borel parts of the spectrum of an operator and of the operator algebra 
         of a separable Hilbert space}{\jRN{SM}}{208}{2012}{77--85}{#3}}
   \ITEE{#2}{pn22}{
      \myBIB{Normed topological pseudovector groups}
         {\jRN{ACS}}{20}{2012}{303--322}{#3}}
   \ITEE{#2}{pn23}{
      \myBIB{Normal systems over ANR's, rigid embeddings and nonseparable absorbing sets}
         {\jRN{ActaMSinES}}{28}{2012}{1531--1552}{#3}}
   \ITEE{#2}{pnX2}{
      \myBAPP{Functor of continuation in Hilbert cube and Hilbert space}
         {submitted to \jRN{FM}\oNlINE{#1}{(\texttt{http://arxiv.org/abs/1107.1386})}}{#3}}
   \ITEE{#2}{pnX3}{
      \myBAPP{Norm closures of orbits of bounded operators}{accepted for publication 
         in \jRN{JOT}\oNlINE{#1}{(\texttt{http://arxiv.org/abs/1107.1505})}}{#3}}
   \ITEE{#2}{pnX6}{
      \myBAPP{Extending maps in Hilbert manifolds}
         {accepted for publication in \jRN{BPAS}}{#3}}
   \ITEE{#2}{pnX7}{
      \myBAPP{Spaces of measurable functions}
         {submitted to \jRN{CEurJM}\oNlINE{#1}{(\texttt{http://arxiv.org/abs/1107.1495})}}{#3}}
   \ITEE{#2}{pnX10}{
      \myBAPP{Central points and measures and dense subsets of compact metric spaces}
         {accepted for publication in \jRN{TMNA}}{#3}}
   \ITEE{#2}{pnX12}{
      \myBAPP{Ultrametrics, extending of Lipschitz maps and nonexpansive selections}
         {accepted for publication in \jRN{HJM}}{#3}}
   \ITEE{#2}{pnX15}{
      \myBAPP{Universal valued Abelian groups}
         {submitted to \jRN{AdvM}\oNlINE{#1}{(\texttt{http://arxiv.org/abs/1103.1623})}}{#3}}
   \ITEE{#2}{pnX17}{
      \myBAPP{Isometry groups of proper metric spaces}
         {submitted to \jRN{TAMS}\oNlINE{#1}{(\texttt{http://arxiv.org/abs/1201.5675})}}{#3}}
   \ITEE{#2}{pnX18}{
      \myBAPP{Isometry groups among topological groups}
         {submitted to \jRN{ComposM}\oNlINE{#1}{(\texttt{http://arxiv.org/abs/1202.3368})}}{#3}}
   \ITEE{#2}{pnX19}{
      \myBAPP{$\CCc^*$\hyp{}algebras of pure type $I_n$}
         {submitted to \jRN{IsraelJM}\oNlINE{#1}{(\texttt{http://arxiv.org/abs/1203.0857})}}{#3}}
   }

\begin{document}

\title{Isometry groups among topological groups}
\myData
\begin{abstract}
It is shown that a topological group $G$ is topologically isomorphic to the isometry group
of a (complete) metric space iff $G$ coincides with its $\GGg_{\delta}$-closure in the Ra\v{\i}kov
completion of $G$ (resp. if $G$ is Ra\v{\i}kov-complete). It is also shown that for every Polish
(resp. compact Polish; locally compact Polish) group $G$ there is a complete (resp. proper) metric
$d$ on $X$ inducing the topology of $X$ such that $G$ is isomorphic to $\Iso(X,d)$ where $X =
\ell_2$ (resp. $X = [0,1]^{\omega}$; $X = [0,1]^{\omega} \setminus \{\textup{point}\}$). It is
demonstrated that there are a separable Banach space $E$ and a nonzero vector $e \in E$ such that
$G$ is isomorphic to the group of all (linear) isometries of $E$ which leave the point $e$ fixed.
Similar results are proved for arbitrary Ra\v{\i}kov-complete topological groups.
\end{abstract}
\subjclass[2010]{Primary 22A05, 54H11; Secondary 57N20.}
\keywords{Polish group; isometry group; Hibert cube; Hilbert space; Hilbert cube manifold; 
   Ra\v{\i}kov-complete group; isometry group of a Banach space.}
\maketitle

\SECT{Introduction}

In \cite{g-k} Gao and Kechris proved that every Polish group is isomorphic to the (full) isometry
group of some separable complete metric space. Later Melleray \cite{mel} and Malicki and Solecki
\cite{m-s} improved this result in the context of compact and, respectively, locally compact Polish
groups by showing that every such group is isomorphic to the isometry group of a compact and,
respectively, a proper metric space. (A metric space is \textit{proper} iff each closed ball in this
space is compact). All their proofs were complicated and based on the techniques of the so-called
\textit{Kat\v{e}tov} maps. Very recently \cite{pn0} we involved quite a new method to characterize
groups of homeomorphisms of a locally compact Polish space which coincide with the isometry groups
of the space with respect to some proper metrics. As a consequence, we showed that every (separable)
Lie group is isomorphic to the isometry group of another Lie group equipped with some proper metric
and that every finite-dimensional [locally] compact Polish group is isomorphic to the isometry group
of a finite-dimensional [proper locally] compact metric space. One of the aims of this paper is
to give Gao's-Kechris', Melleray's and Malicki's-Solecki's results more ``explicit'' and unified
form:

\begin{thm}{iso}
Let $G$ be a Polish group.
\begin{enumerate}[\upshape(a)]
\item There is a complete compatible metric $d$ on $\ell_2$ such that $G$ is isomorphic to
   $\Iso(\ell_2,d)$.
\item If $G$ is compact, there is a compatible metric $d$ on the Hilbert cube $Q$ such that $G$ is
   isomorphic to $\Iso(Q,d)$.
\item If $G$ is locally compact, there is a proper compatible metric $d$ on $Q \setminus
   \{\textup{point}\}$ such that $G$ is isomorphic to $\Iso(Q \setminus \{\textup{point}\},d)$.
\end{enumerate}
\end{thm}

We shall also prove the following

\begin{thm}{Banach}
For every Polish group $G$ there exist a separable real Banach space $E$ and a nonzero vector
$e \in E$ such that $G$ is isomorphic to the group of all linear isometries of $E$ (endowed with
the pointwise convergence topology) which leave the point $e$ fixed.
\end{thm}

Our methods can be adapted to general settings and give a characterization of topological groups
which are isomorphic to isometry groups of complete as well as incomplete metric spaces. To this
end, we recall that a topological group $G$ is \textit{Ra\v{\i}kov-complete} (or \textit{upper
complete}) iff it is complete with respect to the upper uniformity (cf. \cite[\S3.6]{a-t}
or \cite{r-d}; see also the remarks on page~1581 in \cite{usp}). In other words, $G$ is upper
complete if every net $\{x_{\sigma}\}_{\sigma\in\Sigma} \subset G$ satisfying the following
condition (C) is convergent in $G$.
\begin{itemize}
\item[(C)] For every neighbourhood $U$ of the neutral element of $G$ there is $\sigma_0 \in \Sigma$
   such that both $x_{\sigma} x_{\sigma'}^{-1}$ and $x_{\sigma}^{-1} x_{\sigma'}$ belong to $U$ for
   any $\sigma, \sigma' \geqsl \sigma_0$.
\end{itemize}
Equivalently, the net $\{x_{\sigma}\}_{\sigma\in\Sigma}$ satisfies (C) if both the nets
$\{x_{\sigma}\}_{\sigma\in\Sigma}$ and $\{x_{\sigma}^{-1}\}_{\sigma\in\Sigma}$ are fundamental with
respect to the left uniformity of $G$. We call Ra\v{\i}kov-complete groups briefly
\textit{complete}. With this terminology we follow Uspenskij \cite{usp}. The class of all complete
topological groups coincides with the class of all absolutely closed topological groups
(a topological group is \textit{absolutely closed} if it is closed in every topological group
containing it as a topological subgroup). It is well-known that for every topological group $G$
there exists a unique (up to topological isomorphism) complete topological group containing $G$
as a dense subgroup (see e.g. \cite[\S3.6]{a-t} or \cite{r-d}). This complete group is called
the \textit{Ra\v{\i}kov completion} of $G$ and we shall denote it by $\overline{G}$.\par
Less classical are topological groups which we call $\GGg_{\delta}$-complete. To define them, let us
agree with the following general convention. Whenever $\tau$ is a topology on a set $X$,
$\tau_{\delta}$ stands for the topology on $X$ whose base is formed by all $\GGg_{\delta}$-sets
(with respect to $\tau$) in $X$. (In particular, $\GGg_{\delta}$-sets in $(X,\tau)$ are open
in $(X,\tau_{\delta})$.) Subsets of $X$ which are closed or dense in the topology $\tau_{\delta}$
are called $\GGg_{\delta}$-closed and $\GGg_{\delta}$-dense, respectively (see e.g.
\cite[p.~268]{a-t} or \cite{arh}).\par
It may be easily verified that if $(G,\tau)$ is a topological group, so is $(G,\tau_{\delta})$.
We now introduce

\begin{dfn}{G-complete}
A topological group $G$ is \textit{$\GGg_{\delta}$-complete} if $(G,\tau_{\delta})$ is a complete
topological group (where $\tau$ is the topology of $G$).
\end{dfn}

Equivalently, a topological group $G$ is $\GGg_{\delta}$-complete iff $G$ is $\GGg_{\delta}$-closed
in $\overline{G}$. The class of all $\GGg_{\delta}$-complete groups is huge (see \PRO{exs} below)
and contains all complete as well as metrizable topological groups (more detailed discussion on this
class is included in Section~4). However, there are topological groups which are not
$\GGg_{\delta}$-complete (see \EXM{non-G-complete} below).\par
$\GGg_{\delta}$-complete groups turn out to characterize isometry groups of metric spaces,
as shown by

\begin{thm}{G-complete}
Let $G$ be a topological group.
\begin{enumerate}[\upshape(A)]
\item \TFCAE
   \begin{enumerate}[\upshape({A}1)]
   \item there exists a metric space $(X,d)$ such that $G$ is isomorphic to $\Iso(X,d)$,
   \item $G$ is $\GGg_{\delta}$-complete.
   \end{enumerate}
   Moreover, if $G$ is $\GGg_{\delta}$-complete, the space $X$ witnessing \textup{(A1)} may be
   chosen so that $w(X) = w(G)$.
\item \TFCAE
   \begin{enumerate}[\upshape(B1)]
   \item there exists a complete metric space $(X,d)$ such that $G$ is isomorphic to $\Iso(X,d)$,
   \item $G$ is complete.
   \end{enumerate}
   Moreover, if $G$ is complete, the space $X$ witnessing \textup{(B1)} may be chosen so that
   $w(X) = w(G)$.
\end{enumerate}
\end{thm}

(By $w(X)$ we denote the topological weight of a topological space $X$.)\par
One concludes from \THM{G-complete} that the isometry group of an arbitrary metric space is (always)
Dieudonn\'{e} complete (see \COR{dieudonne} below). This solves an open problem, posed e.g.
by Arhangel'skii and Tkachenko (see Open Problem~3.5.4 on page~181 in \cite{a-t}).\par
A generalization of Theorems~\ref{thm:iso} and \ref{thm:Banach} has the following form:

\begin{pro}{non}
Let $G$ be a complete topological group of topological weight not greater than $\beta \geqsl
\aleph_0$.
\begin{enumerate}[\upshape(a)]
\item There is a complete compatible metric $\varrho$ on $\HHh_{\beta}$ such that $G$ is isomorphic
   to $\Iso(\HHh_{\beta},\varrho)$, where $\HHh_{\beta}$ is a real Hilbert space of (Hilbert space)
   dimension equal to $\beta$.
\item There are an infinite-dimensional real Banach space $E$ whose topological weight is equal
   to $\beta$ and a nonzero vector $e \in E$ such that $G$ is isomorphic to the group of all linear
   isometries of $E$ which leave the point $e$ fixed.
\end{enumerate}
\end{pro}

As an immediate consequence of \THM{G-complete} and \PRO{non} we obtain

\begin{cor}{beta}
Let $\HHh$ be a Hilbert space of Hilbert space dimension $\beta \geqsl \aleph_0$ and let
$$\ggG = \{\Iso(\HHh,\varrho)\dd\ \varrho \textup{ is a complete compatible metric on } \HHh\}.$$
Then, up to isomorphism, $\ggG$ consists precisely of all complete topological groups of topological
weight not exceeding $\beta$.
\end{cor}

The paper is organized as follows. In Section~2 we give a new proof of the Gao-Kechris theorem
mentioned above. We consider our proof more transparent, more elementary and less complicated.
The techniques of this part are adapted in Section~3 where we demonstrate that every closed subgroup
of the isometry group of a [complete] metric space $(X,d)$ is actually (isomorphic to) the isometry
group of a certain [complete] metric space, closely `related' to $(X,d)$. This theorem is applied
in the next part, where we establish basic properties of the class of all $\GGg_{\delta}$-complete
groups and prove \THM{G-complete}. Fifth part contains proofs of \THM{Banach}, \PRO{non}, point (a)
of \THM{iso} and \COR{beta}. In Section~6 we study topological groups isomorphic to isometry groups
of completely metrizable metric spaces. The last section is devoted to the proofs of points (b) and
(c) of \THM{iso}.

\subsection*{Notation and terminology} In this paper $\NNN = \{0,1,2,\ldots\}$ (and it is equipped
with the discrete topology). All isomorphisms between topological groups are topological, all
topological groups are Hausdorff and all isometries between metric spaces are, by definition,
bijective. All normed vector spaces are assumed to be real. The topological weight of a topological
space $X$ is denoted by $w(X)$ and it is understood as an infinite cardinal number. Isometry groups
(and all their subsets) of metric as well as normed vector spaces are endowed with the pointwise
convergence topology, which makes them topological groups. A \textit{Polish} space (resp. group)
is a completely metrizable separable topological space (resp. group). A metric on a topological
space is \textit{compatible} iff it induces the topology of the space. It is \textit{proper} if all
closed balls with respect to this metric are compact (in the topology induced by this metric).
Whenever $(X,d)$ is a metric space, $a \in X$ and $r > 0$, $B_X(a,r)$ and $\bar{B}_X(a,r)$ stand
for, respectively, the open and the closed $d$-balls with center at $a$ and of radius $r$.
The Hilbert cube, that is, the countable infinite Cartesian power of $[0,1]$, is denoted by $Q$ and
$\ell_2$ stands for the separable Hilbert space. A \textit{map} means a continuous function.

\SECT{The Gao-Kechris theorem revisited}

This part is devoted to the proof of the Gao-Kechris theorem \cite{g-k} mentioned
in the introductory part and stated below. Another proof may be found in \cite{mel}.

\begin{thm}{g-k}
Every Polish group is isomorphic to the isometry group of a certain separable complete metric space.
\end{thm}

For the purpose of this and the next section, let us agree with the following conventions. For every
nonempty collection $\{X_s\}_{s \in S}$ of topological spaces, $\bigsqcup_{s \in S} X_s$ denotes
the topological disjoint union of these spaces. In particular, whenever the notation
$\bigsqcup_{s \in S} X_s$ appears, the sets $X_s\ (s \in S)$ are assumed to be pairwise disjoint
(the same rule for the symbol `$\sqcup$'). Further, for a function $f\dd X \to X$, an integer number
$n \geqsl 1$ and an arbitrary point $w$, $f^{\cIRC{n}}\dd X^n \to X^n$ is the $n$-th Cartesian power
of $f$ (that is, $f^{\cIRC{n}}(x_1,\ldots,x_n) = (f(x_1),\ldots,f(x_n))$) and $f \times w\dd X
\times \{w\} \to X \times \{w\}$ sends $(x,w)$ to $(f(x),w)$ for any $x \in X$. Similarly, if $d$ is
a metric on $X$, $d^{\cIRC{n}}$ denotes the maximum metric on $X^n$ induced by $d$, i.e.
$$d^{\cIRC{n}}((x_1,\ldots,x_n),(y_1,\ldots,y_n)) = \max_{j=1,\ldots,n} d(x_j,y_j);$$
and $d \times w$ is the metric on $X \times \{w\}$ such that $(d \times w)((x,w), (y,w)) = d(x,y)$.
Finally, for a topological space $V$ and a map $v\dd V \to V$ we put $\widehat{V} = (V \times \NNN)
\sqcup (\bigsqcup_{n=2}^{\infty} V^n) \sqcup \NNN$ and define $\widehat{v}\dd \widehat{V} \to
\widehat{V}$ by the rules: $\widehat{v}(x,m) = (v(x),m)$, $\widehat{v}\bigr|_{V^n} = v^{\cIRC{n}}$
and $\widehat{v}(m) = m$ for any $x \in V$, $m \in \NNN$ and $n \in \NNN \setminus \{0\}$. To avoid
repetitions, for a metric space $(X,d)$ and arbitrary sets $A, B \subset \NNN$ and $C \subset \NNN
\setminus \{0,1\}$, let us say a metric $\varrho$ on $(X \times A) \sqcup (\bigsqcup_{j \in C} X^j)
\sqcup B (\subset \widehat{X})$ \textit{respects} $d$ iff the following three conditions are
satisfied:
\begin{enumerate}[({A}X1)]
\item $\varrho$ coincides with $d \times m$ on $X \times \{m\}$ for each $m \in A$,
\item $\varrho$ coincides with $d^{\cIRC{n}}$ on $X^n$ for each $n \in C$,
\item $\varrho(x,y) \geqsl 1$ whenever $x$ and $y$ belong to distinct members of the collection
   $\{X \times \{m\}\dd\ m \in A\} \cup \{X^n\dd\ n \in C\} \cup \{\{k\}\dd\ k \in B\}$.
\end{enumerate}
Observe that (AX1)--(AX3) imply that
\begin{enumerate}[({A}X1)]\addtocounter{enumi}{3}
\item if $\varrho$ respects $d$, then: $\bullet\ \varrho$ is compatible; $\bullet\ \varrho$ is
   complete if so is $d$.
\end{enumerate}

The main result of this section is the following

\begin{pro}{closed}
Let $(X,d)$ be a separable bounded complete metric space and let $G$ be a closed subgroup
of $\Iso(X,d)$. There exists a metric $\varrho$ on $\widehat{X}$ such that $\varrho$ respects $d$
and the function
$$G \ni u \mapsto \widehat{u} \in \Iso(\widehat{X},\varrho)$$
is a well defined isomorphism of topological groups.
\end{pro}

The proof of \PRO{closed} will be preceded by a few auxiliary results. The first of them is a kind
of folklore and we leave its (simple) proof to the reader.

\begin{lem}{0}
Let $\{(X_s,d_s)\}_{s \in S}$ be a nonempty family of metric spaces such that for $A =
\bigcap_{s \in S} X_s$ we have:
\begin{itemize}
\item $X_s \cap X_{s'} = A$ and $d_s\bigr|_{A \times A} = d_{s'}\bigr|_{A \times A}$ for any two
   distinct indices $s$ and $s'$ of $S$,
\item $A$ is nonempty and closed in $(X_s,d_s)$ for each $s \in S$.
\end{itemize}
Let $X = \bigcup_{s \in S} X_s$ and let $d\dd X \times X \to [0,\infty)$ be given by the rules:
\begin{itemize}
\item $d$ coincides with $d_s$ on $X_s \times X_s$ for every $s \in S$,
\item $d(x,y) = \inf \{d_s(x,a) + d_{s'}(a,y)\dd\ a \in A\}$ whenever $x \in X_s \setminus X_{s'}$
   and $y \in X_{s'} \setminus X_s$ for distinct indices $s$ and $s'$.
\end{itemize}
Then $d$ is a well defined metric on $X$ with the following property. Whenever $f_s \in
\Iso(X_s,d_s)\ (s \in S)$ are maps such that $f_s\bigr|_A = f_{s'}\bigr|_A$ and $f_s(A) = A$ for any
$s, s' \in S$, then their union $f := \bigcup_{s \in S} f_s$ (that is, $f = f_s$ on $X_s$) is a well
defined function such that $f \in \Iso(X,d)$.
\end{lem}

The above result will be the main tool for constructing the metric $\varrho$ appearing
in \PRO{closed}.\par
In the next two results, $(X,p)$ is a complete nonempty metric space with
\begin{equation}\label{eqn:p}
p < 1.
\end{equation}

\begin{lem}{1}
Let $J \subset \NNN \setminus \{0\}$ be a finite set such that $n = \card(J) > 1$. There is a metric
$\lambda$ on $F := [X \times (J \cup \{0\})] \sqcup X^n$ which has the following properties:
\begin{enumerate}[\upshape(a)]
\item $\lambda$ respects $p$ and $\lambda \leqsl 5$,
\item for every $u \in \Iso(X,p)$, $\widehat{u}\bigr|_F \in \Iso(F,\lambda)$,
\item if $g \in \Iso(F,\lambda)$ is such that $g(X \times \{j\}) = X \times \{j\}$ for each $j \in J
   \cup \{0\}$, then $g = \widehat{u}\bigr|_F$ for some $u \in \Iso(X,p)$.
\end{enumerate}
\end{lem}
\begin{proof}
With no loss of generality, we may assume that $J = \{1,\ldots,n\}$. Let $A = \{(x_1,\ldots,x_n) \in
X^n\dd\ x_1 = \ldots = x_n\}$ and $\lambda_0'$ be the metric on $X_0 := (X \times \{0\}) \sqcup A$
such that $\lambda_0'$ coincides with $p \times 0$ on $X \times \{0\}$, with $p^{\cIRC{n}}$ on $A$,
and $\lambda_0'((x,0),(a,\ldots,a)) = 1 + p(x,a)$ for $x \in X$ and $(a,\ldots,a) \in A$. Now apply
\LEM{0} for $\{(X_0,\lambda_0'),(X^n,p^{\cIRC{n}})\}$ to obtain a metric $\lambda_0$ on $(X \times
\{0\}) \sqcup X^n$ which extends both $\lambda_0'$ and $p^{\cIRC{n}}$. Observe that $\lambda_0$
respects $p$, $\lambda_0 \leqsl 3$ (by \eqref{eqn:p}),
$\widehat{u}\bigr|_{(X \times \{0\}) \sqcup X^n} \in \Iso((X \times \{0\}) \sqcup X^n,\lambda_0)$
for each $u \in \Iso(X,p)$ and for arbitrary $x,x_1,\ldots,x_n \in X$:
\begin{equation}\label{eqn:0}
\lambda_0((x,0),(x_1,\ldots,x_n)) = 1 \iff x_1 = \ldots = x_n = x.
\end{equation}
Further, for $j \in J$ let $\lambda_j$ be the metric on $(X \times \{j\}) \sqcup X^n$ such that
$\lambda_j$ coincides with $p \times j$ on $X \times \{j\}$, with $p^{\cIRC{n}}$ on $X^n$ and
$\lambda_j((x,j),(x_1,\ldots,x_n)) = 1 + p(x,x_j)$ for any $x,x_1,\ldots,x_n \in X$ ($\lambda_j$ is
indeed a metric thanks to \eqref{eqn:p}). Similarly as before, notice that $\lambda_j$ respects
$p$, $\lambda_j \leqsl 2$ and for any $x,x_1,\ldots,x_n \in X$:
\begin{equation}\label{eqn:1}
\lambda_j((x,j),(x_1,\ldots,x_n)) = 1 \iff x_j = x
\end{equation}
Now again apply \LEM{0} for the family $\{(X_j,\lambda_j)\dd\ j \in J \cup \{0\}\}$ to obtain
a metric $\lambda$ on $F$ which extends each of $\lambda_j\ (j \in J \cup \{0\})$. It follows from
the construction and \LEM{0} that points (a) and (b) are satisfied. We turn to be. Let $g$ be
as there. Let $u\dd X \to X$ be such that $u \times 0 = g\bigr|_{X \times \{0\}}$ and, similarly,
for $j \in J$ let $u_j\dd X \to X$ be such that $u_j \times j = g\bigr|_{X \times \{j\}}$. Finally,
put $f = g\bigr|_{X^n}\dd X^n \to X^n$. Since $\lambda$ respects $p$, $u \in \Iso(X,p)$. So,
we only need to check that $u_1 = \ldots u_n = u$ and $f = u^{\cIRC{n}}$. Let $\pi_j\dd X^n \to X$
be the projection onto the $j$-th coordinate $(j=1,\ldots,n$). For any $x = (x_1,\ldots,x_n) \in
X^n$ and $j \in J$ we have, by \eqref{eqn:1}:
$$1 = \lambda((\pi_j(x),j),x) = \lambda(g(\pi_j(x),j),g(x)) =
\lambda(((u_j \circ \pi_j)(x),j),f(x))$$
and therefore, again by \eqref{eqn:1}, $u_j \circ \pi_j = \pi_j \circ f$. Consequently,
$f(x_1,\ldots,x_n) = (u_1(x_1),\ldots,u_n(x_n))$. Finally, for any $z \in X$ we have,
by \eqref{eqn:0}:
$$1 = \lambda((z,0),(z,\ldots,z)) = \lambda(g(z,0),g(z,\ldots,z)) =
\lambda((u(z),0),(u_1(z),\ldots,u_n(z)))$$
and hence, again by \eqref{eqn:0}, $u_1(z) = \ldots = u_n(z) = u(z)$.
\end{proof}

\begin{lem}{2}
Let $G$ be a subgroup of $\Iso(X,p)$ and let $z \in X^n$ and $J \subset \NNN \setminus \{0\}$ be
such that $\card(J) = n > 1$. Let $D$ denote the closure (in $X^n$) of the set
$\{u^{\cIRC{n}}(z)\dd\ u \in G\}$. There exists a metric $\mu$ on $F := [X \times (J \cup \{0\})]
\sqcup X^n \sqcup \{n-1\}$ which has the following properties:
\begin{enumerate}[\upshape(a)]
\item $\mu$ respects $p$ and $\mu \leqsl 11$,
\item $\widehat{u}\bigr|_F \in \Iso(F,\mu)$ for every $u \in G$,
\item for any $g \in \Iso(F,\mu)$ there is $u \in \Iso(X,p)$ such that $g = \widehat{u}\bigr|_F$ and
   $u^{\cIRC{n}}(z) \in D$.
\end{enumerate}
\end{lem}
\begin{proof}
Without loss of generality, we may assume that $J = \{1,\ldots,n\}$. Let $\lambda$ be as in \LEM{1}
(so, $\lambda$ is a metric on $F \setminus \{n-1\}$). Let $c_0,\ldots,c_{n+1}$ be such that
\begin{equation}\label{eqn:2}
5 < c_0 < c_1 \ldots < c_{n+1} < 6.
\end{equation}
Put $A = [X \times (J \cup \{0\})] \sqcup D$ and denote by $\mu_0$ the metric on $A \sqcup \{n-1\}$
such that $\mu_0$ coincides with $\lambda$ on $A$, $\mu_0((x,j),n-1) = c_j$ for $x \in X$ and
$j=0,\ldots,n$, and $\mu(y,n-1) = c_{n+1}$ for $y \in D$ ($\mu_0$ is a metric thanks to point (a)
of \LEM{1}, (AX3) and \eqref{eqn:2}). Now apply \LEM{0} for the family
$\{(A \cup \{n-1\},\mu_0),(F \setminus \{n-1\},\lambda)\}$ to obtain a metric $\mu$ which extends
both $\mu_0$ and $\lambda$. We infer from \eqref{eqn:2} and point (a) of \LEM{1} that point (a) is
satisfied. Further, since $u^{\cIRC{n}}(D) = D$ for each $u \in G$ and thanks to point (b)
of \LEM{1}, condition (b) is fulfilled as well (see \LEM{0}). We turn to (c).\par
Let $g \in \Iso(F,\mu)$. Since $n-1$ is a unique point $q \in F$ such that $\mu(q,x) = c_0$ and
$\mu(q,y) = c_1$ for some $x, y \in F$ (since $\lambda \leqsl 5 < c_0 < c_1$), we conclude that
$g(n-1) = n-1$. Further, observe that for each $x \in X^n$, $\mu(x,n-1) \geqsl c_{n+1}$, because
of (AX3) and \eqref{eqn:2}. Consequently, $X \times \{j\} = \{x \in F\dd\ \mu(x,n-1) = c_j\}$ for
$j = 0,1,\ldots,n$. Thus, we see that $g(X \times \{j\}) = X \times \{j\}$ for such $j$'s. Since
$g\bigr|_{F \setminus \{n-1\}} \in \Iso(F \setminus \{n-1\},\lambda)$ and $g(n-1) = n-1$, point (c)
of \LEM{1} implies that there is $u \in \Iso(X,p)$ such that $g = \widehat{u}\bigr|_F$. Finally,
$g(z) = u^{\cIRC{n}}(z) \in X^n$ and for $y \in X^n$, $\mu(y,n-1) = c_{n+1}$ iff $y \in D$
(by \eqref{eqn:2} and (AX3)) which gives $u^{\cIRC{n}}(z) \in D$ and finishes the proof.
\end{proof}

\begin{proof}[Proof of \PRO{closed}]
Let $r \geqsl 1$ be such that $d < r$. Put $p = \frac1r d < 1$ and notice that $\Iso(X,p) =
\Iso(X,d)$. Let $X_0 = \{x_n\dd\ n \geqsl 1\}$ be a dense subset of $X$. Let $J_1,J_2,\ldots$ be
pairwise disjoint sets such that $\bigcup_{n=1}^{\infty} J_n = \NNN \setminus \{0\}$ and $\card(J_n)
= n+1$. For each $n \geqsl 2$ put $z_n = (x_1,\ldots,x_n) \in X^n$, $F_n = [X \times (J_{n-1} \cup
\{0\})] \sqcup X^n \sqcup \{n-1\}$ and let $D_n$ be the closure (in $X^n$)
of $\{u^{\cIRC{n}}(z_n)\dd\ u \in G\}$. Further, let $\mu_n$ be a metric on $F_n$ obtained from
\LEM{2} (applied for $z_n$ and $J_{n-1}$). Now apply \LEM{0} for the collection $\{(F_n,\mu_n)\dd\
n \geqsl 2\}$ to get a metric $\lambda_0$ on $\widehat{X} \setminus \{0\}$ which extends each
of $\mu_n\ (n \geqsl 2)$. In particular, $\lambda_0$ respects $p$ and $\lambda_0 \leqsl 22$.
Finally, we extend the metric $\lambda_0$ to a metric $\lambda$ on $\widehat{X}$ in such a way that
for $k \geqsl 0$, $\lambda(x,0) = c_{k,1}$ for $x \in X \times \{k\}$, $\lambda(x,0) = c_{k,2}$ for
$x \in X^{k+2}$ and $\lambda(k+1,0) = c_{k,3}$ where the numbers
\begin{equation}\label{eqn:3}
c_{0,1},c_{0,2},c_{0,3},c_{1,1},c_{1,2},c_{1,3},\ldots \textup{ are all different}
\end{equation}
and belong to $(22,23)$ ($\lambda$ is a metric thanks to (AX3)). It follows from \LEM{0} and point
(b) of \LEM{2} that $\widehat{u} \in \Iso(\widehat{X},\lambda)$ for any $u \in G$. It is clear that
the function $G \ni u \mapsto \widehat{u} \in \Iso(\widehat{X},\lambda)$ is a group homomorphism and
a topological embedding. We shall now show that it is also surjective.\par
Let $g \in \Iso(\widehat{X},\lambda)$. Since $0$ is a unique point $q \in \widehat{X}$ such that
$\lambda(q,x) = c_{0,1}$ and $\lambda(q,y) = c_{0,2}$ for some $x, y \in \widehat{X}$, we see that
$g(0) = 0$. Consequently, $g(X \times \{k\}) = X \times \{k\}$, $g(X^{k+2}) = X^{k+2}$ and $g(k+1) =
k+1$ for each $k \geqsl 0$, by \eqref{eqn:3}. So, taking into account that $g\bigr|_{F_n} \in
\Iso(F_n,\mu)$, point (c) of \LEM{2} yields that there is $u \in \Iso(X,p)$ such that $g =
\widehat{u}$ and $u(z_n) \in D_n$. The latter condition implies that there are elements
$u_1,u_2,\ldots$ of $G$ which converge pointwise to $u$ on $X_0$. We now infer from the density
of $X_0$ in $X$ that $u = \lim_{n\to\infty} u_n$ and in fact $u \in G$, by the closedness
of $G$.\par
To end the proof, it suffices to put $\varrho = r \lambda$.
\end{proof}

\begin{proof}[Proof of \THM{g-k}]
Let $(H,\cdot)$ be a Polish group. First we involve a standard argument, used e.g. by Melleray
\cite{mel} in his proof of this theorem: take a left invariant metric $d_0 \leqsl 1$ on $H$ and
denote by $(X,d)$ the completion of $(H,d_0)$. Then, of course, $X$ is separable and for every
$h \in H$ there is a unique $u_h \in \Iso(X,d)$ such that $u_h(x) = hx$ for $x \in H$. Observe that
the function $H \ni h \mapsto u_h \in \Iso(X,d)$ is a group homomorphism as well as a topological
embedding. Therefore its image $G$ is isomorphic to $H$. Since $G$ is a Polish subgroup of a Polish
group, $G$ is closed in $\Iso(X,d)$. Now it suffices to apply \PRO{closed} and to use (AX4)
to deduce the completeness of the metric obtained by that result.
\end{proof}

\SECT{Closed subgroups of isometry groups}

In this section we generalize the ideas of the previous part to the context of all isometry groups.
Our aim is to show that a closed subgroup of the isometry group of a metric space is isomorphic
to the isometry group of another metric space. We have decided to discuss the separable case
separately, because in that case the proofs are more transparent and easier. Actually all tools were
prepared in the previous section, except the following one:

\begin{lem}{id}
Let $X$ be a set with $\card(X) \neq 2$ and $I \subset (0,\infty)$ be a nondegenerate interval.
There is a metric $d\dd X \times X \to I \cup \{0\}$ such that the identity map of $X$ is a unique
member of $\Iso(X,d)$.
\end{lem}

We shall prove a stronger version of \LEM{id} at the end of the section. Now we generalize
the concepts involved in Section~3. Let $\beta$ be an infinite cardinal number and let $D_{\beta}$
denote a fixed discrete topological space of cardinality $\beta$. For a metrizable space $X$ and
a function $f\dd X \to X$ let $X^0$ be a one-point space and $f^{\cIRC{\textup{o}}}\dd X^0 \to X^0$
denote the identity map. Further, we put $T(X) = \bigsqcup_{n\in\NNN} X^n$ (recall that $0 \in
\NNN$). Finally, denote by $\widehat{X}_{\beta}$ and $\widehat{f}_{\beta}$ (resp.) the product $T(X)
\times D_{\beta}$ and the function of $\widehat{X}_{\beta}$ into itself such that
$\widehat{f}_{\beta} = f^{\cIRC{n}} \times \xi$ on $X^n \times \{\xi\}$ for any $n \in \NNN$ and
$\xi \in D_{\beta}$. (Notice that $w(\widehat{X}_{\beta}) = \beta$ provided $\beta \geqsl w(X)$.)
For any $J \subset \NNN$ and a collection $\{A_n\}_{n \in J}$ of subsets of $D_{\beta}$, we say
a metric $\varrho$ on $\bigsqcup_{n \in J} (X^n \times A_n) \subset \widehat{X}_{\beta}$
\textit{respects} a compatible metric $d$ on $X$ iff the following two conditions are fulfilled:
\begin{enumerate}[\upshape(PR1)]
\item $\varrho$ coincides with $d^{\cIRC{n}} \times \xi$ on $X^n \times \{\xi\}$ for any $n \in J
   \setminus \{0\}$ and $\xi \in A_n$,
\item $\varrho(x,y) \geqsl 1$ whenever $x$ and $y$ belong to distinct members of the collection
   $\{X^n \times \{\xi\}\dd\ n \in J,\ \xi \in A_n\}$.
\end{enumerate}
As before, we see that (AX4) is satisfied.\par
A counterpart of \PRO{closed} in the general case is the following

\begin{thm}{closed}
Let $\beta$ be an infinite cardinal number and $(X,d)$ be a nonempty bounded metric space such that
$w(X) \leqsl \beta$. For any closed subgroup $G$ of $\Iso(X,d)$ there exists a metric $\varrho$
on $\widehat{X}_{\beta}$ such that $\varrho$ respects $d$ and the function
\begin{equation}\label{eqn:isomorph}
G \ni u \mapsto \widehat{u}_{\beta} \in \Iso(\widehat{X}_{\beta},\varrho)
\end{equation}
is a well defined isomorphism of topological groups.
\end{thm}
\begin{proof}
In what follows, we shall (naturally) identify $X^0 \times D_{\beta}$ with $D_{\beta}$. It follows
from the proof of \PRO{closed} that we may assume $d < 1$. Let $Z$ be a dense set in $X$ such that
$\card(Z) \leqsl \beta$. Fix arbitrary $\theta \in D_{\beta}$ and write $D_{\beta} \setminus
\{\theta\}$ in the form $\bigcup_{n=0}^{\infty} S_n$ where $\card(S_n) = \beta$ for any $n$ and
\begin{equation}\label{eqn:S}
S_n \cap S_m = \varempty \qquad (n \neq m).
\end{equation}
The set $S_0$ and the point $\theta$ will be employed in the last part of the proof. For simplicity,
put $S_* = \bigcup_{n=1}^{\infty} S_n$ and $X_* := \widehat{X}_{\beta} \setminus (S_0  \cup
\{\theta\}) = [\bigsqcup_{n\geqsl1} (X^n \times D_{\beta})] \sqcup S_*$. It follows from
\eqref{eqn:S} that for any $\xi \in D_{\beta} \setminus \{\theta\}$ there is a unique number $n(\xi)
\in \NNN$ such that $\xi \in S_{n(\xi)}$. Further, for every $n \geqsl 1$ there are a surjection
$\kappa_n\dd S_n \to Z^n$ and a bijection $\tau_n\dd S_n \to D_{\beta}$. Take a collection
$\{J_{\xi}\dd\ \xi \in S_*\}$ such that for any $\xi, \eta \in S_*$:
\begin{enumerate}[(S1)]
\item $\card(J_{\xi}) = n(\xi)$,
\item $J_{\xi} \cap J_{\eta} = \varempty$ whenever $\xi \neq \eta$,
\item $\bigcup_{\zeta \in S_*} J_{\zeta} = D_{\beta} \setminus \{\theta\}$.
\end{enumerate}
We deduce from (S1) and \LEM{2} that for each $\xi \in S_*$ there exists a metric $\mu_{\xi}$
on $F_{\xi} := [X \times (J_{\xi} \cup \{\theta\})] \sqcup (X^{n(\xi)} \times
\{\tau_{n(\xi)}(\xi)\}) \sqcup \{\xi\}$ which has the following properties:
\begin{enumerate}[(D1)]
\item $\mu_{\xi}$ respects $d$ and $\mu_{\xi} \leqsl 11$,
\item $\widehat{u}_{\beta}\bigr|_{F_{\xi}} \in \Iso(F_{\xi},\mu_{\xi})$ for every $u \in G$,
\item for any $g \in \Iso(F_{\xi},\mu_{\xi})$ there is $u \in \Iso(X,d)$ such that $g =
   \widehat{u}_{\beta}\bigr|_{F_{\xi}}$ and $u^{\cIRC{n}}(\kappa_n(\xi))$ belongs to the closure
   $B_{\xi}$ of $\{f^{\cIRC{n}}(\kappa_n(\xi))\dd\ f \in G\}$ in $X^n$.
\end{enumerate}
Observe that (S2)--(S3), (D1) and the bijectivity of $\tau_n$'s imply that
\begin{equation}\label{eqn:inter}
F_{\xi} \cap F_{\eta} = X \times \{\theta\}
\end{equation}
for distinct $\xi, \eta \in S_*$, and $X \times \{\theta\}$ is closed in $(F_{\xi},\mu_{\xi})$ (cf.
(PR2)). Moreover, it follows from (D1) that we may apply \LEM{0} for the family
$\{(F_{\xi},\mu_{\xi})\dd\ \xi \in S_*\}$. Let $\mu$ be the metric on $\bigcup_{\xi \in S_*} F_{\xi}
= X_*$ obtained by that lemma. Then
\begin{equation}\label{eqn:mu}
\widehat{u}_{\beta}\bigr|_{X_*} \in \Iso(X_*,\mu) \qquad (u \in G)
\end{equation}
(see (D2) and the last claim in \LEM{0}) and
\begin{equation}\label{eqn:mu'}
\mu \textup{ respects } d \textup{ and } \mu \leqsl 22
\end{equation}
(by (D1)). What is more,
\begin{itemize}
\item[($\star$)] if $g \in \Iso(X_*,\mu)$ is such that $g(A) = A$ for any $A \in \SsS := \{X^n
   \times \{\xi\}\dd\ n \geqsl 1,\ \xi \in D_{\beta}\} \cup \{\{\xi\}\dd\ \xi \in S_*\}$, then $g =
   \widehat{u}_{\beta}\bigr|_{X_*}$ for some $u \in G$.
\end{itemize}
Let us briefly show ($\star$). If $g$ is as there, then $g(F_{\xi}) = F_{\xi}$ for any $\xi \neq
\theta$. So, we infer from (D3) that $g = \widehat{(u_{\xi})}_{\beta}$ on $F_{\xi}$ for some
$u_{\xi} \in \Iso(X,d)$ with
\begin{equation}\label{eqn:kappa}
(u_{\xi})^{\cIRC{n}}(\kappa_n(\xi)) \in B_{\xi}
\end{equation}
where $n = n(\xi)$. We conclude from \eqref{eqn:inter} that $u := u_{\xi}$ is independent
of the choice of $\xi \neq \theta$. Consequently, $g = \widehat{u}_{\beta}\bigr|_{X_*}$. To end
the proof of ($\star$), it remains to check that $u \in G$. Since $\kappa_n$'s are surjective,
\eqref{eqn:kappa} yields that $(u(z_1),\ldots,u(z_n))$ belongs to the closure (in $X^n$)
of $\{(f(z_1),\ldots,f(z_n))\dd\ f \in G\}$ for any $n \geqsl 1$ and $z_1,\ldots,z_n \in Z$. But
this, combined with the fact that the function $\Iso(X,d) \ni f \mapsto f\bigr|_Z \in X^Z$ is
an embedding (when $X^Z$ is equipped with the pointwise convergence topology), yields that $u$
belongs to the closure of $G$ in $\Iso(X,d)$. But $G$ is a closed subgroup and we are done.\par
By \LEM{id}, there is a metric
\begin{equation}\label{eqn:lambda}
\lambda\dd S_0 \times S_0 \to \{0\} \cup [1,2]
\end{equation}
for which $\Iso(S_0,\lambda_0) = \{\id_{S_0}\}$ ($\id_{S_0}$ is the identity map on $S_0$). Let
$\SsS$ be as in ($\star$). Since $\card(\SsS) = \beta = \card(S_0)$, there is a one-to-one function
$v\dd \SsS \to \{11,12\}^{S_0}$. We define a metric $\varrho$ on $X_* \sqcup S_0 =
\widehat{X}_{\beta} \setminus \{\theta\}$ by the rules:
\begin{itemize}
\item $\varrho = \mu$ on $X_* \times X_*$,
\item $\varrho = \lambda$ on $S_0 \times S_0$,
\item $\varrho(\xi,\eta) = \varrho(\eta,\xi) = [v(A)](\eta)$ for $\xi \in X_*$ and $\eta \in S_0$,
   where $A \in \SsS$ is such that $\xi \in A$ (such $A$ is unique).
\end{itemize}
That $\varrho$ is indeed a metric it follows from \eqref{eqn:lambda}, \eqref{eqn:mu'}, axiom (PR2)
for $\mu$ and the fact that for any $\eta \in S_0$, $\varrho(\cdot,\eta)$ is constant on each member
of $\SsS$. Finally, we extend the metric $\varrho$ to $\widehat{X}_{\beta}$ by putting
$\varrho(\xi,\theta) = 22$ for $\xi \in X_*$ and $\varrho(\xi,\theta) = 23$ for $\xi \in S_0$.
Direct calculations show that $\varrho$ is indeed a metric on $\widehat{X}_{\beta}$ and that
$\varrho$ respects $d$. It remains to check that \eqref{eqn:isomorph} is a well defined surjection
(cf. the proof of \PRO{closed}). We infer from \eqref{eqn:mu} and the fact that
$\varrho(\cdot,\eta)$ is constant on each member of $\SsS$ for any $\eta \in S_0 \cup \{\theta\}$
that the function \eqref{eqn:isomorph} is well defined. Now let $g \in
\Iso(\widehat{X}_{\beta},\varrho)$. Since $\theta$ is a unique point $\omega \in
\widehat{X}_{\beta}$ such that $\card(\{\xi \in \widehat{X}_{\beta}\dd\ \varrho(\xi,\omega) = 23\})
> 1$, we obtain $g(\theta) = \theta$. Consequently, $g(X_*) = X_*$ and $g(S_0) = S_0$. The latter
yields that
\begin{equation}\label{eqn:aux50}
g\bigr|_{S_0} = \id_{S_0}
\end{equation}
(because $\varrho$ extends $\lambda$). Now if $\xi, \eta \in X_*$ are arbitrary, the injectivity
of $v$ and the definition of $\varrho$ imply that $\varrho(\xi,\cdot) = \varrho(\eta,\cdot)$
on $S_0$ iff $\xi$ and $\eta$ belong to a common member of $\SsS$. But this, combined with
\eqref{eqn:aux50}, allows us to conclude that $g(A) = A$ for any $A \in \SsS$. Now an application
of ($\star$) (recall that $\varrho$ extends $\mu$) provides us the existence of $u \in G$ for which
$g = \widehat{u}_{\beta}$ on $X_*$. Since $g(\xi) = \xi$ for $\xi \notin X_*$, we see that $g =
\widehat{u}_{\beta}$, which finishes the proof.
\end{proof}

\begin{rem}{better}
Under the notation and the assumptions of \THM{closed}, if $M \geqsl 1$ is such that $d \leqsl M$
and $\epsi > 0$ is arbitrary, the metric $\varrho$ appearing in the assertion of that theorem may be
chosen so that $\varrho \leqsl M + \epsi$. Indeed, the above proof provides the existence
of a bounded metric $\varrho$, say $\varrho \leqsl C$ where $M < C < \infty$. Now it suffices
to replace $\varrho$ by $\omega \circ \varrho$, where $\omega\dd [0,C] \to [0,M + \epsi]$ is affine
on $[0,M]$ and $[M,C]$, and $\omega(0) = 0$, $\omega(M) = M$ and $\omega(C) = M + \epsi$.
\end{rem}

Let $(X,d)$ be a nonempty metric space and $(Y,\varrho)$ denote the completion of $(X,d)$. Since
every isometry of $(X,d)$ extends to a unique isometry of $(Y,\varrho)$, thus the topological group
$\Iso(X,d)$ may naturally be identified with the subgroup $\{u \in \Iso(Y,\varrho)\dd\ u(X) = X\}$
of $\Iso(Y,\varrho)$. If we follow this idea, \THM{closed} may be strengthened as follows:

\begin{pro}{closed2}
Let $(X,d)$ be a nonempty bounded metric space and $(Y,\varrho)$ denote its completion. Let $\beta$
be an infinite cardinal not less than $w(X)$. Further, let $G$ be a closed subgroup of $\Iso(X,d)
\subset \Iso(Y,\varrho)$ and $\bar{G}$ denote its closure in $\Iso(Y,\varrho)$. There are a complete
metric $\lambda$ on $\widehat{Y}_{\beta}$ which respects $\varrho$ and a dense set $X_{\beta}
\subset \widehat{Y}_{\beta}$ such that $(\widehat{Y}_{\beta} \setminus X_{\beta},\lambda)$ is
isometric to $(Y \setminus X,\varrho)$ and the function
\begin{equation}\label{eqn:cl}
\bar{G} \ni u \mapsto \widehat{u}_{\beta} \in \Iso(\widehat{Y}_{\beta},\lambda)
\end{equation}
is a well defined isomorphisms of topological groups which transforms $G$ onto the group of all
$u \in \Iso(\widehat{Y}_{\beta},\lambda)$ with $u(X_{\beta}) = X_{\beta}$.
\end{pro}
\begin{proof}
Fix $\theta \in D_{\beta}$ and put $X_{\beta} = \widehat{Y}_{\beta} \setminus [(Y \setminus X)
\times \{\theta\}]$. By \THM{closed}, there is a metric $\lambda$ on $\widehat{Y}_{\beta}$ which
respects $d$ and for which \eqref{eqn:cl} is a well defined isomorphism. Note that then $\lambda$
is complete (see (AX4)), $X_{\beta}$ is dense in $\widehat{Y}_{\beta}$ and $(\widehat{Y}_{\beta}
\setminus X_{\beta},\lambda)$ is isometric to $(Y \setminus X,\varrho)$ (since $\lambda$ respects
$\varrho$). Finally, if $u \in \bar{G}$, then $\widehat{u}_{\beta}(X_{\beta}) = X_{\beta}$ iff
$u(X) = X$ (which follows from the formulas for $\widehat{u}_{\beta}$ and $X_{\beta}$).
Equivalently, $\widehat{u}_{\beta}(X_{\beta}) = X_{\beta}$ iff $u \in \Iso(X,d) \cap \bar{G} = G$,
by the closedness of $G$ in $\Iso(X,d)$. This shows the last claim of the theorem and finishes
the proof.
\end{proof}

The above result shall be applied in Section~6 devoted to isometry groups of completely metrizable
metric spaces.\par
To complete the proof of \THM{closed}, we need to show \LEM{id}. But the latter result immediately
follows from the following much stronger

\begin{pro}{id}
Let $a$ and $b$ be two reals such that
\begin{equation}\label{eqn:ab}
0 < a < b \leqsl 2a.
\end{equation}
For every set $X$ having more than $5$ points there is a metric $d\dd X \times X \to \{0,a,b\}$ such
that
\begin{equation}\label{eqn:triv}
\Iso(X,d) = \{\id_X\}.
\end{equation}
\end{pro}
\begin{proof}
First of all, observe that any function $d\dd X \times X \to \{0,a,b\}$ which is symmetric and
vanishes precisely on the diagonal of $X$ is automatically a complete metric, which follows from
\eqref{eqn:ab}. So, we only need to take care of \eqref{eqn:triv}. For the same reason, we may (and
do) assume, with no loss of generality, that $a = 1$ and $b = 2$. We shall make use of transfinite
induction with respect to $\beta = \card(X) > 5$. Everywhere below in this proof, for $x \in X$,
by $S(x)$ we denote the set of all $y \in X$ with $d(x,y) = 1$. Since we have to define a metric
taking values in $\{0,1,2\}$, it is readily seen that it suffices to describe the sets $S(x)\ (x \in
X)$.\par
First assume $\beta = n \geqsl 6$ is finite. We may assume that $X = \{1,\ldots,n\}$. Our metric $d$
is defined by the following rules: $S(1) = \{2\}$, $S(2) = \{1,3,4,5\}$, $S(3) = \{2,4\}$, $S(4) =
\{2,3,5\}$, $S(5) = \{2,4,6\}$, $S(n) = \{n-1\}$ and $S(j) = \{j-1,j+1\}$ if $5 < j < n$. Take
$g \in \Iso(X,d)$ and observe that:
\begin{itemize}
\item $g(2) = 2$, since $2$ is the only point $j \in X$ such that $\card(S(j)) = 4$;
\item $g(1) = 1$, because $1$ is the unique point $j \in X$ for which $S(j) = \{2\}$;
\item $g(3) = 3$, since $3$ is the only point $j \in X$ such that $\card(S(j)) = 2$ and $2 \in
   S(j)$;
\item $g(4) = 4$, because $4$ is the unique point $j \in X$ for which $2\neq j \in S(3)$;
\item $g(5) = 5$, since $5$ is the only point $j \in X$ such that $j \in S(4) \setminus \{2,3\}$.
\end{itemize}
Now it is easy to check, using induction, that $g(j) = j$ for $j = 6,\ldots,n$.\par
When $\beta = \aleph_0$, we may assume $X = \NNN$. Define a metric $d\dd \NNN \times \NNN \to
\{0,1,2\}$ by $d(n,m) = \min(|m-n|,2)$. It is left to the reader that $\Iso(\NNN,d) =
\{\id_{\NNN}\}$ (use induction to show that $g(n) = n$ for any $n \in \NNN$ and $g \in
\Iso(\NNN,d)$). Below we assume that $\beta > \aleph_0$ is such that for every infinite $\alpha <
\beta$ the proposition holds for an arbitrary set $X$ of cardinality $\alpha$. For simplicity, for
any uncountable cardinal $\gamma$ we denote by $I_{\gamma}$ the set of all cardinals $\alpha$ for
which $\aleph_0 \leqsl \alpha < \gamma$. To get the assertion, we consider three cases.\par
First assume $\beta$ is not limit, i.e. $\beta$ is the immediate successor of an infinite cardinal
$\alpha$. We may assume that $X$ is the union of three pairwise disjoint sets $X'$, $X' \times Y$
and $\{a\}$ where $\card(X') = \alpha$ and $\card(Y) = \beta$. It follows from the transfinite
induction hypothesis that there is a metric $d'\dd X' \times X' \to \{0,1,2\}$ such that
$\Iso(X',d') = \{\id_{X'}\}$. Since $\beta \leqsl 2^{\alpha}$, there is a one-to-one function
$\mu\dd X' \times Y \to \{1,2\}^{X'}$ such that
\begin{equation}\label{eqn:sph}
[\mu(x,y)](x) = 1 \qquad (x \in X',\ y \in Y)
\end{equation}
(such a function $\mu$ may easily be constructed by transfinite induction with respect to an initial
well order on $X' \times Y$). We now define a metric $d$ on $X$ (with values in $\{0,1,2\}$)
by the rules:
\begin{enumerate}[(d1)]
\item $d = d'$ on $X' \times X'$,
\item $d((x,y),(x',y')) = 1$ if $(x,y)$ and $(x',y')$ are distinct elements of $X' \times Y$,
\item $d((x,y),x') = [\mu(x,y)](x')$ if $x, x' \in X'$ and $y \in Y$, and
\item $d(x,a) = 1$ and $d((x,y),a) = 2$ for any $x \in X'$ and $y \in Y$.
\end{enumerate}
Observe that $S(x) \supset \{x\} \times Y$ and $S(x,y) \supset (X' \times Y) \setminus \{(x,y)\}$
for any $x \in X'$ and $y \in Y$ (thanks to \eqref{eqn:sph} and (d2)--(d3)); and
\begin{equation}\label{eqn:X'}
S(a) = X'
\end{equation}
(by (d4)). We infer from these facts that $a$ is a unique point $x \in X$ such that $\card(S(x)) =
\alpha$. Consequently, if $g \in \Iso(X,d)$, then $g(a) = a$ and $g(X') = X'$ (because
of \eqref{eqn:X'}). So, $g\bigr|_{X'} \in \Iso(X',d')$ and therefore $g(x) = x$ for any $x \in X'$.
Finally, if $x, x' \in X'$ and $y \in Y$ are arbitrary, then $g(x,y) \notin X'$ and $d(g(x,y),x') =
d((x,y),x')$, which yields that $\mu(g(x,y)) = \mu(x,y)$. So, $g(x,y) = (x,y)$ (since $\mu$ is
one-to-one) and we are done.\par
Now we assume that $\beta$ is limit and $\card(I_{\beta}) < \beta$. For simplicity, put $I =
I_{\beta}$. Let $\{X_{\alpha}\}_{\alpha \in I}$ be a family of pairwise disjoint sets such that
\begin{equation}\label{eqn:alph}
X_{\alpha} \cap I = \varempty \quad \textup{and} \quad \card(X_{\alpha}) = \alpha < \beta \qquad
(\alpha \in I).
\end{equation}
Note that the set $X_* = \bigsqcup_{\alpha \in I} X_{\alpha}$ is of cardinality $\beta$ and
therefore we may assume $X = \{\omega\} \sqcup I \sqcup X_*$ (recall that this notation means that
$\omega \notin I \cup X_*$). It follows from the transfinite induction hypothesis that there are
metrics $d_I\dd I \times I \to \{0,1,2\}$ and $d_{\alpha}\dd X_{\alpha} \times X_{\alpha} \to
\{0,1,2\}\ (\alpha \in I)$ for which the groups $\Iso(I,d_I)$ and $\Iso(X_{\alpha},d_{\alpha})$ are
trivial. We define a metric $d$ on $X$ as follows:
\begin{enumerate}[\upshape(d1')]
\item $d = d_I$ on $I \times I$ and $d = d_{\alpha}$ on $X_{\alpha} \times X_{\alpha}$ for any
   $\alpha \in I$,
\item $d(x,y) = 1$ if $x$ and $y$ belong to different members of the collection
   $\{X_{\alpha}\}_{\alpha \in I}$,
\item $d(\alpha,x) = 2$ for $x \in X_{\alpha}$ and $d(\alpha,x) = 1$ for $x \in X_* \setminus
   X_{\alpha}\ (\alpha \in I)$,
\item $d(\alpha,\omega) = 1$ and $d(x,\omega) = 2$ for any $\alpha \in I$ and $x \in X_*$.
\end{enumerate}
Observe that for any $\alpha \in I$ and $x \in X_{\alpha}$, $S(\alpha) \supset X_* \setminus
X_{\alpha}$ (cf. (d3')) and $S(x) \supset X_* \setminus X_{\alpha}$ as well (cf. (d2')). At the same
time, $S(\omega) = I$ (by (d4')) and hence $\omega$ is a unique point $x \in X$ such that
$\card(S(x)) < \beta$ (cf. \eqref{eqn:alph}). Consequently, if $g \in \Iso(X,d)$, then $g(\omega) =
\omega$, $g(I) = I$ and $g(X_*) = X_*$. Then $g\bigr|_I \in \Iso(I,d_I)$ (cf. (d1')) and hence
$g(\alpha) = \alpha$ for each $\alpha \in I$. Now use (d3') to conclude that $g(X_{\alpha}) =
X_{\alpha}$ for any $\alpha \in I$. So, according to (d1'), $g\bigr|_{X_{\alpha}} \in
\Iso(X_{\alpha},d_{\alpha})$ for every $\alpha \in I$ and consequently $g(x) = x$ for all $x \in
\bigcup_{\alpha \in I} X_{\alpha} = X_*$, and we are done.\par
Finally, assume $\beta$ is limit and $\card(I_{\beta}) = \beta$. Then we may assume $X = I_{\beta}$.
Since
\begin{equation}\label{eqn:aux58}
\card(I_{\alpha}) \leqsl \alpha < \beta
\end{equation}
whenever $\alpha \in X$, for every $\alpha \in X$ there is a cardinal $\gamma(\alpha) \in X$ such
that
\begin{equation}\label{eqn:aux56}
\card(\{\xi\dd\ \alpha < \xi \leqsl \gamma(\alpha)\}) = \alpha.
\end{equation}
Now define a metric $d\dd X \times X \to \{0,1,2\}$ by the following rule: if $\aleph_0 \leqsl
\alpha_1 < \alpha_2 < \beta$, then $d(\alpha_1,\alpha_2) = 1$ iff $\alpha_2 \leqsl
\gamma(\alpha_1)$. It is easy to check that then $\card(S(\alpha)) = \alpha$ for any $\alpha \in X$
(thanks to \eqref{eqn:aux58} and \eqref{eqn:aux56}) and hence the identity map is a unique isometry
on $(X,d)$.
\end{proof}

\SECT{Models for $\GGg_{\delta}$-complete groups}

We begin this section with a useful characterization of $\GGg_{\delta}$-complete groups.

\begin{pro}{G-complete}
For a topological group $G$ all conditions stated below are equivalent.
\begin{enumerate}[\upshape(I)]
\item $G$ is $\GGg_{\delta}$-complete.
\item $G$ is isomorphic to a $\GGg_{\delta}$-closed subgroup of a complete topological group.
\item $G$ is $\GGg_{\delta}$-closed in every topological group which contains $G$ as a topological
   subgroup.
\item Every net $\{x_{\sigma}\}_{\sigma\in\Sigma}$ of elements of $G$ satisfying the following
   condition \textup{(CC)} is convergent in $G$.
   \begin{itemize}
   \item[(CC)] For every sequence $U_1,U_2,\ldots$ of neighbourhoods of the neutral element of $G$
      there exist points $y, z \in G$ and a sequence $\sigma_1,\sigma_2,\ldots\in\Sigma$ such that
      both $x_{\sigma}^{-1} y$ and $x_{\sigma} z^{-1}$ belong to $U_n$ whenever $n \geqsl 1$ and
      $\sigma \geqsl \sigma_n$.
   \end{itemize}
\item Every net $\{x_{\sigma}\}_{\sigma\in\Sigma}$ of elements of $G$ satisfying the following
   condition \textup{(CC')} is convergent in $G$.
   \begin{itemize}
   \item[(CC')] For every continuous left invariant pseudometric $d$ on $G$ there are points $y, z
      \in G$ such that $\lim_{\sigma\in\Sigma} d(x_{\sigma},y) = \lim_{\sigma\in\Sigma}
      d(x_{\sigma}^{-1},z^{-1}) = 0$.
   \end{itemize}
\end{enumerate}
\end{pro}
\begin{proof}
Everywhere below $\tau$ is the topology of $G$ and $e$ is its neutral element.\par
First assume $G$ is a $\GGg_{\delta}$-closed subgroup of a complete group $H$.
We want to show that $G$ is $\GGg_{\delta}$-complete. Let $\{x_{\sigma}\}_{\sigma\in\Sigma} \subset
G$ be a net which satisfies condition (C) with respect to the topology $\tau_{\delta}$. It then
satisfies this condition with respect to $\tau$ as well. Since $H$ is complete, there is $y \in H$
such that $\lim_{\sigma\in\Sigma} x_{\sigma} = y$. It suffices to check that $y$ belongs
to the $\GGg_{\delta}$-closure of $G$ in $H$. Take a $\GGg_{\delta}$-subset of $H$ containing $y$
and write $A y^{-1}$ in the form $A y^{-1} = \bigcap_{n=1}^{\infty} U_n$ where each $U_n$ is an open
in $H$ and contains $e$. Let $V_1,V_2,\ldots$ be a sequence of open (in $H$) neighbourhoods of $e$
such that the closure (in $H$) of $V_n$ is contained in $V_{n-1} \cap U_n$ for each $n$. where $V_0
= H$. Then $F := \bigcap_{n=1}^{\infty} V_n$ is a closed $\GGg_{\delta}$-subset of $H$ and $e \in F
\subset A y^{-1}$. It follows from our assumption about the net that there is $\sigma_0 \in \Sigma$
such that $x_{\sigma_0} x_{\sigma}^{-1} \in F$ for any $\sigma \geqsl \sigma_0$. We now infer from
the closedness of $F$ in $H$ that $x_{\sigma_0} y^{-1} \in F$ as well and consequently $x_{\sigma_0}
\in A$, which shows that $y$ belongs to the $\GGg_{\delta}$-closure of $G$. This proves that (II) is
followed by (I). Conversely, if $G$ is a $\GGg_{\delta}$-complete subgroup of a topological group
$K$ and $\OOo$ denotes the topology of $K$, then $\tau_{\delta}$ coincides with the topology
(on $G$) of a subspace inherited from $(K,\OOo_{\delta})$. It now follows from the completeness
of $(G,\tau_{\delta})$ that $G$ is closed in $(K,\OOo_{\delta})$ or, equivalently, that $G$ is
$\GGg_{\delta}$-closed in $K$, which proves that (III) is implied by (I). Since (II) obviously
follows from (III), in this way we have shown that conditions (I), (II) and (III) are equivalent.
We shall now show that (II) is equivalent to (IV) and then that (IV) is equivalent to (V).\par
If (II) is fulfilled, then $G$ is $\GGg_{\delta}$-closed in $\overline{G}$. Let
$\{x_{\sigma}\}_{\sigma\in\Sigma}$ be a net of elements of $G$ which satisfies condition (CC). Then
it fulfills condition (C) as well and hence there is $w \in \overline{G}$ such that
$\lim_{\sigma\in\Sigma} x_{\sigma} = w$. It suffices to check that $w \in G$ or, equivalently, that
$w$ belongs to the $\GGg_{\delta}$-closure of $G$. To this end, take any $\GGg_{\delta}$-subset $A$
of $\overline{G}$ which contains $w$. Write $A w^{-1} = \bigcap_{n=1}^{\infty} V_n$ where $V_1,V_2,
\ldots$ are open neighbourhoods of $e$. For each $n \geqsl 1$ take a neighbourhood $U_n$ of $e$ with
$U_n = U_n^{-1}$ and $U_n \cdot U_n \subset V_n$. Now let $y$, $z$ and $\sigma_1,\sigma_2,\ldots$ be
as in (CC), applied for the sequence $$U_1 \cap G,U_2 \cap G,\ldots$$ Fix for a moment $n \geqsl 1$.
Choose $\sigma \geqsl \sigma_n$ such that $x_{\sigma} w^{-1} \in U_n$. Then $z w^{-1} = (x_{\sigma}
z^{-1})^{-1} (x_{\sigma} w^{-1}) \subset U_n^{-1} \cdot U_n \subset V_n$. So, $z w^{-1} \in
\bigcap_{n=1}^{\infty} V_n = A w^{-1}$, which implies that $w \in A$. Consequently, $A \cap G \neq
\varempty$ and we are done.\par
The converse implication goes similarly: when (IV) is satisfied, we show that $G$ is
$\GGg_{\delta}$-closed in $\overline{G}$. Let $w \in \overline{G}$ belong
to the $\GGg_{\delta}$-closure of $G$. Then, of course, $w$ is in the closure of $G$ and thus there
is a net $\{x_{\sigma}\}_{\sigma\in\Sigma} \subset G$ which converges to $w$. To prove that $w \in
G$, it is enough to verify that (CC) is fulfilled. To this end, fix a sequence $U_1,U_2,\ldots$
of neighbourhoods of $e$ and choose its open symmetric neighbourhoods $V_1,V_2,\ldots$ such that
$V_n \cdot V_n \subset U_n\ (n \geqsl 1)$. We conclude from the fact that $w$ is
in the $\GGg_{\delta}$-closure of $G$ that there is $y \in G$ such that $y \in
\bigcap_{n=1}^{\infty} (V_n w \cap w V_n)$. Fix $n \geqsl 1$. There is $\sigma_n \in \Sigma$ such
that both $x_{\sigma} w^{-1}$ and $w^{-1} x_{\sigma}$ belong to $V_n$ for $\sigma \geqsl \sigma_n$.
Then, for such $\sigma$'s, $x_{\sigma} y^{-1} = (x_{\sigma} w^{-1}) (y w^{-1})^{-1} \subset V_n
\cdot V_n^{-1} \subset U_n$ and $x_{\sigma}^{-1} y = (w^{-1} x_{\sigma})^{-1} (w^{-1} y) \subset
V_n^{-1} \cdot V_n \subset U_n$ as well. This shows that (CC) is satisfied for $z = y$, and we are
done.\par
Point (V) is easily implied by (IV) (for a fixed continuous left invariant pseudometric $d$ and
a net satisfying (CC') apply (CC) for $U_n = \{x \in G\dd\ d(x,e) < 2^{-n}\}$). The converse
implication follows from the well-known fact that for an arbitrary sequence $U_1,U_2,\ldots$
of neighbourhoods of $e$ there exists a left invariant pseudometric $d$ on $G$ such that $\{x \in
G\dd\ d(x,e) < 2^{-n}\} \subset U_n$ for every $n \geqsl 1$ (see e.g. the proof
of the Kakutani-Birkhoff theorem on the metrizability of topological groups presented
in \cite[Theorem~6.3]{ber}; or use Markov's theorem \cite[Theorem~3.3.9]{a-t} to deduce this
property).
\end{proof}

\begin{rem}{weak}
The proof of \PRO{G-complete} shows that points (IV) and (V) of that result may be weakened
by assuming that every net satisfying condition (CC) or (CC') with $z = y$ is convergent. However,
to prove \THM{G-complete}, we need (IV) in its present form.
\end{rem}

Now we can give many examples of $\GGg_{\delta}$-complete groups. We inform that by the Cartesian
product of a family $\{G_s\}_{s \in S}$ of topological groups we mean the `full' Cartesian product
$\prod_{s \in S} G_s$ of them and by the direct product of this family we mean the topological
subgroup $\bigoplus_{s \in S} G_s$ of $\prod_{s \in S} G_s$ consisting of all its finitely supported
elements.

\begin{pro}{exs}
Each of the following topological groups is $\GGg_{\delta}$-complete.
\begin{enumerate}[\upshape(a)]
\item A $\GGg_{\delta}$-closed subgroup of a $\GGg_{\delta}$-complete group. A complete group.
\item The Cartesian as well as the direct product of arbitrary collection
   of $\GGg_{\delta}$-complete groups.
\item A topological group which is the countable union of its subgroups each of which is
   $\GGg_{\delta}$-complete.
\item A topological group which is $\sigma$-compact as a topological space. In particular, all
   countable topological groups are $\GGg_{\delta}$-complete.
\item A topological group in which singletons are $\GGg_{\delta}$. In particular, metrizable groups
   are $\GGg_{\delta}$-complete.
\item $G = \Iso(X,d)$ for an arbitrary metric space $(X,d)$. Moreover, $w(G) \leqsl w(X)$, and $G$
   is complete provided $(X,d)$ is complete.
\end{enumerate}
\end{pro}
\begin{proof}
In each point we involve \PRO{G-complete}.\par
To prove point (a), use the equivalence between conditions (I) and (III) in \PRO{G-complete}.
We turn to (b). Let $\{G_s\}_{s \in S}$ be a nonempty collection of $\GGg_{\delta}$-complete groups
and let $G = \prod_{s \in S} G_s$. Let $x_{\sigma} = (x^{(s)}_{\sigma})_{s \in S} \in G\
(\sigma\in\Sigma)$ be a net satisfying condition (CC). It remains to check that for any $t \in S$,
the net $\{x^{(t)}_{\sigma}\}_{\sigma\in\Sigma} \subset G_t$ satisfies condition (CC) as well
(because then it will be convergent), which is immediate: if $V_1,V_2,\ldots$ is a sequence
of neighbourhoods of the neutral element of $G_t$, apply (CC) for the sequence $U_1,U_2,\ldots$ with
$U_j := \{(x^{(s)})_{s \in S} \in G\dd\ x^{(t)} \in V_j\}\ (j\geqsl1)$ to obtain two points $y, z
\in G$ and then use their $t$-coordinates. Now to prove that also $H := \bigoplus_{s \in S} G_s$ is
$\GGg_{\delta}$-complete, it suffices to check that it is $\GGg_{\delta}$-closed in $G$ (by (a)).
But if $y = (y_s)_{s \in S} \in G \setminus H$, there is a countable (infinite) set $S' \subset S$
such that $y_s \neq e_s$ for any $s \in S'$, where $e_s$ is the neutral element of $G_s$. Then
the set $A := \{(z_s)_{s \in S} \in G|\quad \forall s \in S'\dd\ z_s \neq e_s\}$ is
a $\GGg_{\delta}$-subset of $G$ which contains $y$ and is disjoint from $H$, which finishes
the proof of (b).\par
Since the proofs of points (c) and (d) are similar, we shall show only (c). Let $G =
\bigcup_{n=1}^{\infty} G_n$ where $G_n$ is $\GGg_{\delta}$-complete for any $n$. Let $y \in
\overline{G} \setminus G$. Then $y \notin G_n$ and $G_n$ is $\GGg_{\delta}$-closed
in $\overline{G}$. Consequently, there are $\GGg_{\delta}$-subsets $A_1,A_2,\ldots$
of $\overline{G}$ containing $y$ such that $A_n \cap G_n = \varempty$. Then $A :=
\bigcap_{n=1}^{\infty} A_n$ is also a $\GGg_{\delta}$-subset of $\overline{G}$ containing $y$, and
$A \cap G = \varempty$. This shows that $y$ is not in the $\GGg_{\delta}$-closure of $G$ and we are
done.\par
Further, if all singletons are $\GGg_{\delta}$ in $G$, then $\tau_{\delta}$ is discrete and hence
$G$ is $\GGg_{\delta}$-complete. This proves (e).\par
Finally, we turn to (f). The second and the third claims of (f) are well-known, but for the sake
of completeness we shall prove them too. Let $(X,d)$ be a metric space and $G = \Iso(X,d)$. Let
$\{u_{\sigma}\}_{\sigma\in\Sigma} \subset G$ be a net satisfying condition (CC'). Fix $x \in X$ and
put $\varrho\dd G \times G \ni (u,v) \mapsto d(u(x),v(x)) \in [0,\infty)$. Observe that $\varrho$ is
a left invariant continuous pseudometric on $G$. It follows from (CC') that there are $f, g \in G$
such that $\lim_{\sigma\in\Sigma} d(u_{\sigma}(x),f(x)) = \lim_{\sigma\in\Sigma}
d(u_{\sigma}^{-1}(x),g(x)) = 0$. We conclude that both the nets
$\{u_{\sigma}(x)\}_{\sigma\in\Sigma}$ and $\{u_{\sigma}^{-1}(x)\}_{\sigma\in\Sigma}$ converge
in $X$. So, we may define $u, v\dd X \to X$ by $u(x) = \lim_{\sigma\in\Sigma} u_{\sigma}(x)$ and
$v(x) = \lim_{\sigma\in\Sigma} u_{\sigma}^{-1}(x)\ (x \in X)$. It is readily seen that both $u$ and
$v$ are isometric. What is more, a standard argument proves that $u \circ v = v \circ u = \id_X$ and
hence $u \in G$ and $\lim_{\sigma\in\Sigma} u_{\sigma} = u$. So, $G$ is $\GGg_{\delta}$-complete.
Furthermore, if $D$ is a dense subset of $X$ such that $\card(D) = w(X)$, then the function $G \ni g
\mapsto g\bigr|_D \in X^D$ is a topological embedding (when $X^D$ is equipped with the pointwise
convergence topology) and therefore $w(G) \leqsl w(X^D) \leqsl w(X)$. Finally, if $(X,d)$ is
in addition complete and $\{u_{\sigma}\}_{\sigma\in\Sigma} \subset G$ is a net satisfying (C),
similar argument to that above shows that then for any $x \in X$ the nets
$\{u_{\sigma}(x)\}_{\sigma\in\Sigma}$ and $\{u_{\sigma}^{-1}(x)\}_{\sigma\in\Sigma}$ are fundamental
in $(X,d)$ and hence converge. It now follows from the previous part of the proof that
$\{u_{\sigma}\}_{\sigma\in\Sigma}$ is convergent in $G$, which finishes the proof.
\end{proof}

For the purpose of the next result, recall that a topological space $X$ is \textit{Dieudonn\'{e}
complete} iff there is a complete uniformity on $X$ inducing the topology of $X$ (see e.g.
\cite[Chapter~8]{en2}). Accordingly, a topological group is \textit{Dieudonn\'{e} complete} iff
it is Dieudonn\'{e} complete as a topological space (\cite{a-t}).

\begin{cor}{dieudonne}
For every metric space $(X,d)$ the topological group $\Iso(X,d)$ is Dieudonn\'{e} complete.
\end{cor}
\begin{proof}
By \PRO{exs}, $\Iso(X,d)$ is $\GGg_{\delta}$-complete and hence, thanks to \PRO{G-complete}, it is
$\GGg_{\delta}$-closed in $\overline{G}$. Consequently, $\Iso(X,d)$ is Dieudonn\'{e} complete (since
$\overline{G}$ is such and $\GGg_{\delta}$-closed subsets of Dieudonn\'{e} complete spaces are
Dieudonn\'{e} complete as well---see \cite{die} or Problem~8.5.13(f) on page~465 in \cite{en2};
cf. also the proof of Proposition~6.5.2 on page~366 in \cite{a-t}).
\end{proof}

The above result gives a full answer to the question of when the isometry group of a metric space is
Dieudonn\'{e} complete, posed by Arhangel'skii and Tkachenko in \cite{a-t} (see Open Problem~3.5.4
on page~181 there).

\begin{exm}{non-G-complete}
As we announced in the introductory part, not every topological group is absolutely
$\GGg_{\delta}$-closed. Let us briefly justify our claim. Let $S$ be an uncountable set and for each
$s \in S$ let $G_s$ be a nontrivial complete group with the neutral element $e_s$. Then $G :=
\prod_{s \in S} G_s$ is a complete group as well and
$$G_0 = \{(x_s)_{s \in S} \in G\dd\ \card(\{s \in S\dd\ x_s \neq e_s\}) \leqsl \aleph_0\}$$
is a proper subgroup of $G$ which is $\GGg_{\delta}$-dense in $G$ (and thus $G_0$ is not
$\GGg_{\delta}$-closed in $G$). Indeed, if $y =(y_s)_{s \in S} \in G$ and $A$ is
a $\GGg_{\delta}$-subset of $G$ containing $y$, then there is a countable set $S_0 \subset S$ such
that $\{(x_s)_{s \in S} \in G|\quad \forall s \in S_0\dd\ x_s = y_s\} \subset A$; then $z \in G_0
\cap A$ where $z_s = y_s$ for $s \in S_0$ and $z_s = e_s$ otherwise.
\end{exm}

We are almost ready for proving \THM{G-complete}. For the purpose of its proof and the nearest
result, let us introduce auxiliary notations and terminology. Whenever $d$ and $d'$ are two bounded
pseudometrics on a common nonempty set $X$, we put
$$\|d - d'\|_{\infty} := \sup_{x,y \in X} |d(x,y) - d'(x,y)|.$$
Further, the relation $R := \{(x,y) \in X \times X\dd\ d(x,y) = 0\}$ is an equivalence on $X$. Let
$\pi\dd X \to X/R$ be the canonical projection. The function $D\dd (X/R) \times (X/R) \ni
(\pi(x),\pi(y)) \mapsto d(x,y) \in [0,\infty)$ is a well defined metric on $X/R$. We call a triple
$(Y,\varrho;\xi)$ a \textit{metric space associated with} $(X,d)$ if $(Y,\varrho)$ is a metric space
isometric to $(X/R,D)$ and $\xi$ is a function of $X$ onto $Y$ such that there is an isometry $g\dd
(X/R,D) \to (Y,\varrho)$ for which $\xi = g \circ \pi$. Observe that then $\varrho(\xi(x),\xi(y)) =
d(x,y)$ for any $x, y \in X$.\par
With use of the following result we shall take care of condition $w(X) = w(G)$ in point (A)
of \THM{G-complete}.

\begin{lem}{aux}
Let $G$ be a topological group and $\{\varrho_s\}_{s \in S}$ be a collection of bounded continuous
left invariant metrics on $G$. For each $s \in S$, let $(X_s,d_s;\pi_s)$ be a metric space
associated with $(G,\varrho_s)$ chosen so that the sets $X_s$'s are pairwise disjoint. There exists
a metric $d$ on $X := \bigcup_{s \in S} X_s$ with the following properties:
\begin{enumerate}[\upshape(D1)]
\item $\frac12 \sqrt[3]{d_s(x,y)} \leqsl d(x,y) \leqsl \sqrt[3]{d_s(x,y)}$ for any $x, y \in X_s$
   and $s \in S$,
\item $d(\pi_s(a),\pi_t(a)) \leqsl \sqrt[3]{\|\varrho_s - \varrho_t\|_{\infty}}$ for any $a \in G$
   and $s, t \in S$,
\item each of the sets $X_s\ (s \in S)$ is closed in $(X,d)$,
\item $d(\pi_s(ag),\pi_t(ah)) = d(\pi_s(g),\pi_t(h))$ for any $a,g,h \in G$ and $s,t \in S$.
\end{enumerate}
\end{lem}
\begin{proof}
To simplify arguments, for each $x \in X$ denote by $\kappa(x)$ the unique index $s \in S$ such that
$x \in X_s$. Define a function $v\dd X \times X \to [0,\infty)$ as follows:
$$v(x,y) = \|\varrho_{\kappa(x)} - \varrho_{\kappa(y)}\|_{\infty} + \inf
\{d_{\kappa(x)}(x,\pi_{\kappa(x)}(g)) + d_{\kappa(y)}(\pi_{\kappa(y)}(g),y)\dd\ g \in G\}.$$
Observe that:
\begin{itemize}
\item $v(x,x) = 0$ and $v(y,x) = v(x,y)$ for any $x, y \in X$,
\item $v(x,y) = d_s(x,y)$ for any $x, y \in X_s$ and $s \in S$,
\item $v(\pi_s(g),\pi_t(g)) = \|\varrho_s - \varrho_t\|_{\infty}$ whenever $s, t \in S$ and $g \in
   G$,
\item $v(\pi_s(g),\pi_t(h)) \geqsl \|\varrho_s - \varrho_t\|_{\infty}$ for all $s, t \in S$ and $g,
   h \in G$,
\item $v(\pi_s(ag),\pi_t(ah)) = v(\pi_s(g),\pi_t(h))$ for any $a, g, h \in G$ and $s, t \in S$.
\end{itemize}
Let us now show that for any $x_0, x_1, x_2, x_3 \in X$ and each $\epsi > 0$:
\begin{equation}\label{eqn:e0}
\max_{j=1,2,3} v(x_{j-1},x_j) < \epsi \implies v(x_0,x_3) < 8 \epsi.
\end{equation}
Assume $v(x_{j-1},x_j) < \epsi\ (j=1,2,3)$. This means that there are $a_1,a_2,a_3 \in G$ for which
\begin{multline}\label{eqn:55}
\|\varrho_{\kappa(x_{j-1})} - \varrho_{\kappa(x_j)}\|_{\infty}
+ d_{\kappa(x_{j-1})}(x_{j-1},\pi_{\kappa(x_{j-1})}(a_j))\\
+ d_{\kappa(x_j)}(\pi_{\kappa(x_j)}(a_j),x_j) < \epsi.
\end{multline}
In particular, $\|\varrho_{\kappa(x_{j-1})} - \varrho_{\kappa(x_j)}\|_{\infty} < \epsi$ for
$j=1,2,3$ and thus
\begin{equation}\label{eqn:aux51}
\|\varrho_{\kappa(x_0)} - \varrho_{\kappa(x_3)}\|_{\infty} < 3 \epsi.
\end{equation}
For simplicity, for $j \in \{0,1,2,3\}$ put $s_j = \kappa(x_j)$ and take $b_j \in G$ such that
$\pi_{s_j}(b_j) = x_j$. Recall that $d_s(\pi_s(g),\pi_s(h)) = \varrho_s(g,h)$ for any $s \in S$ and
$g, h \in G$. Therefore we have
\begin{multline*}
\varrho_{s_j}(b_{j-1},b_j) \leqsl \varrho_{s_j}(b_{j-1},a_j) + \varrho_{s_j}(a_j,b_j) \leqsl\\
\leqsl \|\varrho_{s_{j-1}} - \varrho_{s_j}\|_{\infty} + \varrho_{s_{j-1}}(b_{j-1},a_j)
+ \varrho_{s_j}(a_j,b_j) < \epsi
\end{multline*}
(by \eqref{eqn:55}) and consequently
$$\varrho_{s_2}(b_0,b_2) \leqsl \varrho_{s_2}(b_0,b_1) + \varrho_{s_2}(b_1,b_2)
\leqsl \|\varrho_{s_2} - \varrho_{s_1}\|_{\infty} + \varrho_{s_1}(b_0,b_1) + \varrho_{s_2}(b_1,b_2)
< 3 \epsi.$$
Similarly,
\begin{multline}\label{eqn:52}
\varrho_{s_3}(b_0,b_3) \leqsl \varrho_{s_3}(b_0,b_2) + \varrho_{s_3}(b_2,b_3) \leqsl\\
\leqsl \|\varrho_{s_3} - \varrho_{s_2}\|_{\infty} + \varrho_{s_2}(b_0,b_2) + \varrho_{s_3}(b_2,b_3)
< 5 \epsi.
\end{multline}
Finally, by \eqref{eqn:aux51} and \eqref{eqn:52} we obtain:
\begin{multline*}
v(x_0,x_3) \leqsl \|\varrho_{s_0} - \varrho_{s_3}\|_{\infty}
+ d_{s_0}(\pi_{s_0}(b_0),\pi_{s_0}(b_0)) + d_{s_3}(\pi_{s_3}(b_0),\pi_{s_3}(b_3)) <\\
< 3 \epsi + v_{s_3}(b_0,b_3) < 8 \epsi
\end{multline*}
and the proof of \eqref{eqn:e0} is complete. Now let $f\dd X \times X \to [0,\infty)$ be given
by $f(x,y) = \sqrt[3]{v(x,y)}$. Below we collect all properties established for $v$ and translated
to the case of the function $f$:
\begin{enumerate}[\upshape(F1)]
\item $f(x,x) = 0$ and $f(x,y) = f(y,x) > 0$ for any distinct points $x$ and $y$ of $X$,
\item if $\epsi > 0$ and $\{f(x,y),f(y,z),f(z,w)\} \subset [0,\epsi]$ for some $x,y,z,w \in X$, then
   $f(x,w) \leqsl 2\epsi$ (thanks to \eqref{eqn:e0}),
\item $f(x,y) = \sqrt[3]{d_s(x,y)}$ and $f(\pi_s(g),\pi_t(g)) =
   \sqrt[3]{\|\varrho_s - \varrho_t\|_{\infty}}$ whenever $x, y \in X_s$, $g \in G$ and $s, t
   \in S$,
\item $f(\pi_s(g),\pi_t(h)) \geqsl \sqrt[3]{\|\varrho_s - \varrho_t\|_{\infty}}$ for all $g, h \in
   G$ and $s, t \in S$,
\item $f(\pi_s(ag),\pi_t(ah)) = f(\pi_s(g),\pi_t(h))$ for  any $a, g, h \in G$ and $s, t \in S$.
\end{enumerate}
Finally, we define $d\dd X \times X \to [0,\infty)$ as follows:
$$d(x,y) = \inf \bigl\{\sum_{j=1}^n f(z_{j-1},z_j)\dd\ n \geqsl 1,\ z_0,\ldots,z_n \in X,\ z_0 = x,\
z_n = y\bigr\}.$$
Lemma~6.2 of \cite{ber} asserts that for any function $f\dd X \times X \to [0,\infty)$ satisfying
conditions (F1)--(F2) the function $d$ defined above is a metric on $X$ such that
\begin{equation}\label{eqn:f-d}
\frac12 f(x,y) \leqsl d(x,y) \leqsl f(x,y) \qquad (x,y \in X).
\end{equation}
It follows from (F5) and the formula of $d$ that (D4) is fulfilled, while (D1) and (D2) may easily
be deduced from (F3) and \eqref{eqn:f-d}. Finally, (D3) is a consequence of (F4) and
\eqref{eqn:f-d}, and we are done.
\end{proof}

\begin{proof}[Proof of \THM{G-complete}]
Implications `(A1)$\implies$(A2)' and `(B1)$\implies$(B2)' follow from \PRO{exs}. It remains to show
the converse implications.\par
First assume that $G$ is complete (in this case the proof is much shorter). By a well-known result
(see e.g. \cite[Theorem~2.1]{usp}), there is a bounded metric space $(Y,\varrho)$ such that $w(G) =
w(Y)$ and $G$ is isomorphic to a subgroup $H$ of $\Iso(Y,\varrho)$. Since $\Iso(Y,\varrho)$ is
naturally isomorphic to a subgroup of $\Iso(\bar{Y},\bar{\varrho})$, where $(\bar{Y},\bar{\varrho})$
is the completion of $(Y,\varrho)$, we may assume that $(Y,\varrho)$ is a complete metric space.
Since $G$ is complete, $H$ is a closed subgroup of $(Y,\varrho)$. Now \THM{closed} implies that $H$
is isomorphic to $\Iso(X,d)$ where $X = \widehat{Y}_{\beta}$ with $\beta = w(Y)$, and $d$ is
a metric which respects $\varrho$. Notice that $d$ is complete (by (AX4)), $w(X) = w(G)$ (because
$\beta = w(Y) = w(G)$) and $G$ is isomorphic to $\Iso(X,d)$ (being isomorphic to $H$). This proves
the remainder of point (B).\par
We now turn to (A). Assume $G$ is $\GGg_{\delta}$-complete. Thanks to \THM{closed}, it suffices
to show that $G$ is isomorphic to a closed subgroup of $\Iso(X,d)$ for a metric space $(X,d)$
of topological weight equal to $w(G)$ (see the previous part of the proof). We shall do this
employing \LEM{aux} and improving a classical argument, presented e.g. in 1st proof
of \cite[Theorem~2.1]{usp}.\par
Let $\bbB$ be a base of open neighbourhoods of the neutral element $e$ of $G$ such that $\card(\bbB)
\leqsl w(G)$. Let $S$ be the set of all finite and all infinite sequences of members of $\bbB$. For
any $U \in \bbB$ there exists a continuous left invariant metric $\lambda_U$ on $G$ bounded by $1$
such that
\begin{equation}\label{eqn:ball}
\{x \in G\dd\ \lambda_U(x,e) < 1\} \subset U.
\end{equation}
We leave it as a simple exercise that the family $\{\lambda_U\}_{U \in \bbB}$ determines
the topology of $G$. Now for any $s = (U_j)_{j=1}^N \in S$ (where $N$ is finite or $N = \infty$) let
\begin{equation}\label{eqn:rho}
\varrho_s := \sum_{j=1}^N \frac{1}{2^j} \lambda_{U_j}.
\end{equation}
Notice that $\varrho_s$ is a continuous left invariant metric on $G$ bounded by $1$. What is more,
\begin{itemize}
\item[(T)] the family $\{\varrho_s\}_{s \in S}$ determines the topology of $G$
\end{itemize}
(since $\varrho_s = \lambda_U$ for $s = (U,U,\ldots) \in S$). Let $(X_s,d_s;\pi_s)\ (s \in S)$
as well as $(X,d)$ be as in \LEM{aux}. For each $s \in S$ let $\tilde{\varrho}_s\dd G \times G \to
[0,\infty)$ be given by $\tilde{\varrho}_s(g,h) = d(\pi_s(g),\pi_s(h))$. It is clear that
$\tilde{\varrho}_s$ is a pseudometric on $G$. What is more, it is left invariant (thanks to (D4))
and
\begin{equation}\label{eqn:til}
\frac12 \sqrt[3]{\varrho_s} \leqsl \tilde{\varrho}_s \leqsl \sqrt[3]{\varrho_s} \qquad (s \in S)
\end{equation}
(by (D1)). Consequently, each of the pseudometrics $\tilde{\varrho}_s$'s is continuous and
\begin{itemize}
\item[(T')] the family $\{\tilde{\varrho}_s\}_{s \in S}$ determines the topology of $G$
\end{itemize}
(see (T)). We infer from the continuity of $\tilde{\varrho}_s$ that $\pi_s$, as a function of $G$
into $(X,d)$, is continuous as well. We claim that $w(X) \leqsl w(G)$. To see this, let $S_f$
consists of all finite sequences of members of $\bbB$, and let $D$ be a dense subset of $G$ such
that $\card(D) \leqsl w(G)$. Observe that $Z := \bigcup_{s \in S_f} \pi_s(D)$ has cardinality not
exceeding $w(G)$. We will now show that $Z$ is dense in $X$. First of all, note that $\pi_s(D)$ is
dense in $X_s$, since $\pi_s$ is continuous. In particular, the closure of $Z$ contains all points
of $\bigcup_{s \in S_f} X_s$. Fix $s \notin S_f$ and $a \in G$. Then $s$ is of the form $s =
(U_j)_{n=1}^{\infty} \in S$. Put $s_n := (U_j)_{j=1}^n \in S_f$ and observe that, by (D2) and
\eqref{eqn:rho},
$$d(\pi_{s_n}(a),\pi_s(a)) \leqsl \sqrt[3]{\|\varrho_{s_n} - \varrho_s\|_{\infty}} \to 0
\quad (n\to\infty).$$
So, since $\pi_s(G) = X_s$, the above argument shows that $Z$ is indeed dense in $X$.\par
It remains to check that $G$ is isomorphic to a closed subgroup of $\Iso(X,d)$. For $g \in G$ let
$u_g\dd X \to X$ be such that $u_g(\pi_s(x)) = \pi_s(gx)$ for any $s \in S$ and $x \in G$. Then
$\Phi\dd G \ni g \mapsto u_g \in \Iso(X,d)$ is a well defined (by (D4)) group homomorphism as well
as a topological embedding (thanks to (T')). So, it follows from point (f) of \PRO{exs} that $w(X)
\geqsl w(G)$ and hence in fact $w(X) = w(G)$. We shall check that $\Phi(G)$ is closed, which will
finish the proof. Assume $\{x_{\sigma}\}_{\sigma\in\Sigma}$ is a net in $G$ such that the net
$\{u_{x_{\sigma}}\}_{\sigma\in\Sigma}$ converges in $\Iso(X,d)$ to some $u \in \Iso(X,d)$. It is
enough to prove that the net $\{x_{\sigma}\}_{\sigma\in\Sigma}$ is convergent in $G$. Since $G$ is
$\GGg_{\delta}$-complete, actually it suffices to verify condition (CC) (see \PRO{G-complete}).
To this end, let $V_1,V_2,\ldots$ be a sequence of neighbourhoods of $e$. For any $n \geqsl 1$
choose $U_n \in \bbB$ for which $U_n \subset V_n$. Now for $s := (U_j)_{j=1}^{\infty} \in S$ we have
\begin{eqnarray*}
\lim_{\sigma\in\Sigma} d(\pi_s(x_{\sigma}),u(\pi_s(e))) = \lim_{\sigma\in\Sigma}
d(u_{x_{\sigma}}(\pi_s(e)),u(\pi_s(e))) = 0,\\
\lim_{\sigma\in\Sigma} d(\pi_s(x_{\sigma}^{-1}),u^{-1}(\pi_s(e))) = \lim_{\sigma\in\Sigma}
d(u_{x_{\sigma}}^{-1}(\pi_s(e)),u^{-1}(\pi_s(e))) = 0.
\end{eqnarray*}
We infer from (D3) and the above convergences that there are $y, z \in G$ such that
$\lim_{\sigma\in\Sigma} d(\pi_s(x_{\sigma}),\pi_s(y)) = \lim_{\sigma\in\Sigma}
d(\pi_s(x_{\sigma}^{-1}),\pi_s(z^{-1})) = 0$. But $d(\pi_s(a),\pi_s(b)) = \tilde{\varrho}_s(a,b)$
for any $a, b \in G$ and thus, thanks to \eqref{eqn:til},
$$\lim_{\sigma\in\Sigma} \varrho_s(x_{\sigma},y) = \lim_{\sigma\in\Sigma}
\varrho_s(x_{\sigma}^{-1},z^{-1}) = 0.$$
For each $n \geqsl 1$ let $\sigma_n \in \Sigma$ be such that both the numbers
$\varrho_s(x_{\sigma},y) = \varrho_s(x_{\sigma}^{-1}y,e)$ and $\varrho_s(x_{\sigma}^{-1},z^{-1}) =
\varrho_s(x_{\sigma}z^{-1},e)$ are less than $2^{-n}$ for any $\sigma \geqsl \sigma_n$. We deduce
from the formula of $s$, \eqref{eqn:rho} and \eqref{eqn:ball} that
$\{x_{\sigma}^{-1}y,x_{\sigma}z^{-1}\} \subset U_n (\subset V_n)$ for all $\sigma \geqsl \sigma_n$
and we are done.
\end{proof}

The above proof provides the existence of a metric space (namely, $\widehat{X}_{\beta}$) whose
isometry group is isomorphic to a given $\GGg_{\delta}$-complete group $G$. This metric space is
highly disconnected (since it contains a clopen discrete set whose cardinality is equal
to the topological weight of the whole space). In the next section we shall improve \THM{G-complete}
by showing that $G$ is isomorphic to the isometry group of a contractible open set in a normed
vector space of the same topological weight as $G$.

\begin{cor}{weight}
\begin{enumerate}[\upshape(A)]
\item Let $G$ be a topological group and $\beta$ be an infinite cardinal number. There exists
   a [complete] metric space $(X,d)$ such that $w(X) = \beta$ and $\Iso(X,d)$ is isomorphic to $G$
   iff $G$ is $\GGg_{\delta}$-complete [resp. complete] and $\beta \geqsl w(G)$.
\item A topological group is isomorphic to the isometry group of some separable metric space iff
   it is second countable.
\end{enumerate}
\end{cor}
\begin{proof}
Both points (A) and (B) follow from Theorems~\ref{thm:G-complete} and \ref{thm:closed}, (AX4) and,
respectively, points (f) and (e) of \PRO{exs}.
\end{proof}

We call a metrizable space $X$ \textit{zero-dimensional} iff $X$ has a base consisting of clopen
(that is, simultaneously open and closed) sets; and $X$ is \textit{strongly zero-dimensional}
if the covering dimension of $X$ equals $0$.

\begin{cor}{0-dim}
Let $G$ be an infinite metrizable topological group and $\beta = w(G)$. There exists a compatible
metric $\varrho$ on $Y := \widehat{G}_{\beta}$ such that $G$ is isomorphic to $\Iso(Y,\varrho)$.
In particular:
\begin{enumerate}[\upshape(A)]
\item If $G$ is discrete, there is a complete compatible (possibly non-left invariant) metric $d$
   on $G$ such that $G$ is isomorphic to $\Iso(G,d)$.
\item If $G$ is countable and non-discrete, there is a compatible metric $d$ on $F := \QQQ \sqcup
   \ZZZ$ such that $G$ is isomorphic to $\Iso(F,d)$.
\item If $G$ is totally disconnected (zero-dimensional; strongly zero-dimensional), there is
   a metric space $(X,d)$ such that $X$ is totally disconnected (resp. zero\hyp{}dimensional;
   strongly zero-dimensional) as well, $w(X) = w(G)$ and $\Iso(X,d)$ is isomorphic to $G$.
\end{enumerate}
\end{cor}
\begin{proof}
Let $p$ be a bounded left invariant compatible metric on $G$ (if $G$ is discrete, we may
additionally assume that $p$ is complete). It is an easy exercise (and a well-known fact) that all
left translations on $G$ form a closed subgroup of $\Iso(G,p)$. Consequently, by \THM{closed}, $G$
is isomorphic to $\Iso(Y,\varrho)$ for some metric $\varrho$ which respects $p$. Note that if $G$ is
discrete, $Y$ is homeomorphic to $G$, which proves (A). Further, if $G$ is countable and
non-discrete, it is homeomorphic to the space of all rationals (e.g. by Sierpi\'{n}ski's theorem
\cite{sie} that every countable metrizable topological space without isolated points is homeomorphic
to $\QQQ$; cf. point (d) of Exercise~6.2.A on page~370 in \cite{en2}) and hence $Y$ is homeomorphic
to $F$, from which we deduce (B). Finally, (C) simply follows from the following remarks: if $G$ is
totally disconnected (resp. zero-dimensional; strongly zero-dimensional), so is $Y$ (cf.
Theorems~1.3.6, 4.1.25 and 4.1.3 in \cite{en1}).
\end{proof}

Point (A) of the above result may be generalized to the context of so-called non-archimedean
complete topological groups. Recall that a topological group is \textit{non-archimedean} iff it has
a base of neighbourhoods of the neutral element consisting of open subgroups. Non-archimedean Polish
groups play important role e.g. in model theory; see \S1.5 in \cite{b-k}. The equivalence between
points (i) and (ii) of the following result is taken from this book (there it was formulated only
for Polish groups, but the proof works in the general case).

\begin{cor}{nonarch}
Let $G$ be a topological group and $\beta = w(G)$. Let $D$ be a discrete topological space
of cardinality $\beta$ and $S_{\beta}$ be the symmetric group of the set $D$ (that is, $S_{\beta}$
consists of all permutations of $D$). \TFCAE
\begin{enumerate}[\upshape(i)]
\item $G$ is non-archimedean and complete;
\item $G$ is isomorphic to a closed subgroup of $S_{\beta}$;
\item $G$ is isomorphic to $\Iso(D,\varrho)$ for some metric $\varrho \geqsl \delta_D$ on $D$ where
   $\delta_D$ is the discrete metric on $D$.
\end{enumerate}
\end{cor}
\begin{proof}
First note that $S_{\beta} = \Iso(D,\delta)$. Thus, if (ii) is satisfied, \THM{closed} implies that
there is a metric $\varrho$ on $X = \widehat{D}_{\beta}$ which respects $\delta_D$ such that $G$ is
isomorphic to $\Iso(X,\varrho)$. Since $\varrho$ respects $\delta_D$, we readily have $\varrho
\geqsl \delta_X$. Finally, the notice that $X$ is homeomorphic to $D$ gives (iii).\par
Now assume (iii) is satisfied. Since $\varrho$ is complete, $G$ is complete as well (cf. \PRO{exs}).
Furthermore, the sets $U_A = \{u \in \Iso(D,\varrho)\dd\ u(a) = a \textup{ for any } a \in A\}$
where $A$ runs over all finite subsets of $D$ are open (because $\varrho \geqsl 1$) subgroups
of $\Iso(D,\varrho)$ which form a base of neighbourhoods of the identity map and thus $G$ is
non-archimedean.\par
Finally, assume (i) is fulfilled. Let $\bbB$ be a base of neighbourhoods of the neutral element
$e_G$ of $G$ which consists of open subgroups and has cardinality $\beta$. For any $H \in \bbB$,
the size of the collection $\{gH\dd\ g \in G\}$ does not exceed $\beta$ and hence $\card(\uuU) =
\beta = \card(D)$ for $\uuU = \{gH\dd\ g \in G,\ H \in \bbB\}$. Hence we may and do identify the set
$D$ with $\uuU$. For any $g \in G$ put $\pi_g\dd \uuU \ni U \mapsto gU \in \uuU$. Under the above
identification, we readily have $\pi_g \in S_{\beta}$. What is more, the function $\pi\dd G \ni g
\mapsto \pi_g \in S_{\beta}$ is easily seen to be a group homomorphism (possibly discontinuous).
It suffices to check that $\pi$ is an embedding (because then $\pi(G)$ be closed, thanks
to the completeness of $G$). Since $\bigcap \bbB = \{e_G\}$, $\pi$ is one-to-one. To complete
the proof, observe that for any net $\{g_{\sigma}\}_{\sigma\in\Sigma} \subset G$ one has
\begin{align*}
\lim_{\sigma\in\Sigma} g_{\sigma} = e_G &\iff \forall g \in G\dd\ \lim_{\sigma\in\Sigma} g^{-1}
g_{\sigma} g = e_G\\&\iff \forall g \in G\ \ \forall H \in \bbB\ \ \exists \sigma_0 \in \Sigma\ \
\forall \sigma \geqsl \sigma_0\dd\ \pi_{g_{\sigma}}(gH) = gH\\
&\iff \lim_{\sigma\in\Sigma} \pi_{g_{\sigma}} = \pi_{e_G}
\end{align*}
(recall that $\uuU$, as identified with $D$, has discrete topology; and that for any $x \in G$ and
$H \in \bbB$, $x \in H$ iff $x H = H$).
\end{proof}

With use of \PRO{closed2}, we may easily strengthen \THM{G-complete}:

\begin{pro}{completion}
For any $\GGg_{\delta}$-complete group $G$ there are a complete metric space $(X,d)$ with $w(X) =
w(G)$, a dense set $X' \subset X$ and an isomorphism $\Phi\dd \overline{G} \to \Iso(X,d)$ such that
$\Phi(G) = \{u \in \Iso(X,d)\dd\ u(X') = X'\} (= \Iso(X',d))$.
\end{pro}
\begin{proof}
First use \THM{G-complete} to get the isomorphism between $G$ and $\Iso(Y,\varrho)$ for some metric
space $(Y,\varrho)$ with $w(Y) = w(G)$ and then apply \PRO{closed2} to conclude the whole assertion
(recall that the isometry group of a complete metric space is complete).
\end{proof}

\begin{exm}{isom}
Let $(X,d)$ be an arbitrary metric space and $G$ be a subgroup of $\Iso(X,d)$. It follows from
\THM{G-complete} and \PRO{G-complete} that $G$ is isomorphic to the isometry group of some metric
space iff $G$ is $\GGg_{\delta}$-closed in $\Iso(X,d)$. Let us briefly show that
the $\GGg_{\delta}$-closure of $G$ coincides with the set of all $u \in \Iso(X,d)$ such that for any
separable subspace $A$ of $X$ there is $v \in G$ which agrees with $u$ on $A$. Indeed, there is
a countable set $D \subset A$ which is dense in $A$. Then the set $F(u,A) := \{v \in \Iso(X,d)\dd\
v\bigr|_A = u\bigr|_A\}$ coincides with $\{v \in \Iso(X,d)\dd\ v\bigr|_D = u\bigr|_D\}$. This
implies that $F(u,A)$ is $\GGg_{\delta}$ in $\Iso(X,d)$. Consequently, if $u$ belongs
to the $\GGg_{\delta}$-closure of $G$, then necessarily $G \cap F(u,A) \neq \varempty$. Conversely,
it may be easily shown that for every $\GGg_{\delta}$-set $P$ containing $u$ there is a countable
set $A$ such that $F(u,A) \subset P$ and hence the condition on $u$ under the question is also
sufficient.\par
According to the above remark, \THM{closed} may now be generalized as follows: a subgroup $G$
of $\Iso(X,d)$ is isomorphic to the isometry group of some metric space (of the same topological
weight as $X$) iff $G$ satisfies the following condition. Whenever $u \in \Iso(X,d)$ is such that
for any separable subspace $A$ of $X$ there exists $v \in G$ which agrees with $u$ on $A$, then
$u \in G$.
\end{exm}

We end the section with the concept of $\GGg_{\delta}$-completions. Similarly as in case
of Ra\v{\i}kov completion, any topological group $G$ has a unique $\GGg_{\delta}$-completion, i.e.
$G$ may be embedded in a $\GGg_{\delta}$-complete group as a $\GGg_{\delta}$-dense subgroup
in a unique way, as shown by

\begin{pro}{G-completion}
Let $G$ and $K$ be $\GGg_{\delta}$-complete groups and $H$ be a $\GGg_{\delta}$-dense subgroup
of $G$.
\begin{enumerate}[\upshape(a)]
\item Every continuous homomorphism of $H$ into $K$ extends uniquely to a continuous homomorphism
   of $G$ into $K$.
\item If $f\dd H \to K$ is a group homomorphism as well as a topological embedding and
   $\widetilde{G}$ denotes the $\GGg_{\delta}$-closure of $f(H)$ in $K$, then there is a unique
   (topological) isomorphism $F\dd G \to \widetilde{G}$ which extends $f$.
\end{enumerate}
\end{pro}
\begin{proof}
Let $f\dd H \to K$ be a continuous homomorphism. Since $H$ is dense in $G$, there is a unique
continuous group homomorphism $F\dd G \to \overline{K}$ which extends $f$. It suffices to show that
$F(G) \subset K$. But this follows from the fact that the preimage of a $\GGg_{\delta}$-closed set
under a continuous function is $\GGg_{\delta}$-closed too. This proves (a). If in addition $f$ is
a topological embedding, it follows from the above argument that there is a continuous group
homomorphism $\widetilde{F}\dd \widetilde{G} \to G$ which extends $f^{-1}$. We then readily see that
both $\widetilde{F} \circ F$ and $F \circ \widetilde{F}$ are the identity maps and hence $F$ is
an isomorphism (and $\widetilde{F} = F^{-1}$), which shows (b).
\end{proof}

\begin{dfn}{G-completion}
Let $G$ be a topological group. The \textit{$\GGg_{\delta}$-completion} of $G$ is
a $\GGg_{\delta}$-complete group which contains $G$ as a $\GGg_{\delta}$-dense (topological)
subgroup. It follows from \PRO{G-completion} that the $\GGg_{\delta}$-completion is unique
up to isomorphism fixing the points of $G$. It is also obvious that any group has
the $\GGg_{\delta}$-completion.
\end{dfn}

\SECT{Hilbert spaces as underlying topological spaces}

Our first aim of this section is to prove \THM{Banach} and point (b) of \PRO{non}. To this end,
we recall a classical construction due to Arens and Eells \cite{a-e} (see also
\cite[Chapter~2]{wea}).

\begin{dfn}{AE}
Let $(X,d)$ be a nonempty complete metric space. For every $p \in X$ let $\chi_p\dd X \to \{0,1\}$
be such that $\chi_p(x) = 1$ if $x = p$ and $\chi_p(x) = 0$ otherwise. A \textit{molecule} of $X$ is
any function $m\dd X \to \RRR$ which is supported on a finite set and satisfies $\sum_{p \in X} m(p)
= 0$. Denote by $\ArEe_0(X)$ the real vector space of all molecules of $X$ and for $m \in
\ArEe_0(X)$ put
$$\|m\|_{AE} = \inf \bigl\{\sum_{j=1}^n |a_j| d(p_j,q_j)\dd\ m = \sum_{j=1}^n a_j (\chi_{p_j}
- \chi_{q_j})\bigr\}.$$
Then $\|\cdot\|_{AE}$ is a norm and the completion of $(\ArEe_0(X),\|\cdot\|_{AE})$ is called
the \textit{Arens-Eells} space of $(X,d)$ and denoted by $(\ArEe(X),\|\cdot\|_{AE})$. Moreover,
$w(\ArEe(X)) = w(X)$.
\end{dfn}

It is an easy observation that every isometry $u\dd X \to Y$ between complete metric spaces $X$ and
$Y$ induces a unique linear isometry $\ArEe(u)\dd \ArEe(X) \to \ArEe(Y)$ such that $\ArEe(u)(\chi_p
- \chi_q) = \chi_{u(p)} - \chi_{u(q)}$ for any $p,q \in X$. The following result is surely
well-known, but probably nowhere explicitly stated. Therefore, for reader's convenience, we give its
short proof.

\begin{lem}{AE}
For every complete metric space $(X,d)$, the function
\begin{equation}\label{eqn:AE}
\Psi\dd \Iso(X,d) \ni u \mapsto \ArEe(u) \in \Iso(\ArEe(X),\|\cdot\|_{AE})
\end{equation}
is both a group homomorphism and a topological embedding.
\end{lem}
\begin{proof}
Continuity and homomorphicity of $\Psi$ is clear (note that $\ArEe_0(X)$ is dense in $\ArEe(X)$ and
$\ArEe_0(X)$ is the linear span of the set $\{\chi_p - \chi_q\dd\ p, q \in X\}$). Here we focus only
on showing that $\Psi$ is an embedding. We may and do assume that $\card(X) > 1$. Suppose
$\{u_{\sigma}\}_{\sigma\in\Sigma} \subset \Iso(X,d)$ is a net such that
\begin{equation}\label{eqn:aux300}
\lim_{\sigma\in\Sigma} \Psi(u_{\sigma}) = \Psi(u)
\end{equation}
for some $u \in \Iso(X,d)$. Let $x \in X$ be arbitrary. We only need to verify that
$\lim_{\sigma\in\Sigma} u_{\sigma}(x) = u(x)$. Let $y \in X$ be different from $x$. We infer from
\eqref{eqn:aux300} that
\begin{equation}\label{eqn:aux301}
\chi_{u_{\sigma}(x)} - \chi_{u_{\sigma}(y)} \to \chi_{u(x)} - \chi_{u(y)} \quad (\sigma\in\Sigma).
\end{equation}
For $\epsi \in (0,d(x,y))$ let $B_{\epsi} = \{z \in X\dd\ d(u(x),z) \geqsl \epsi\} (\neq \varempty)$
and let $v_{\epsi}\dd X \to [0,\infty)$ denote the distance function to $B_{\epsi}$, that is,
$$v_{\epsi}(z) = \inf \{d(z,b)\dd\ b \in B_{\epsi}\}.$$
Observe that $v_{\epsi}$ is Lipschitz. Since the dual Banach space to $\ArEe(X)$ is naturally
isomorphic to the Banach space of all Lipschitz real-valued functions on $X$ (see e.g. Chapter~2
of \cite{wea}), we deduce from \eqref{eqn:aux301} that $\lim_{\sigma\in\Sigma}
[v_{\epsi}(u_{\sigma}(x)) - v_{\epsi}(u_{\sigma}(y))] = v_{\epsi}(u(x)) - v_{\epsi}(u(y))$. But
$v_{\epsi}(u(x)) - v_{\epsi}(u(y)) = v_{\epsi}(u(x)) > 0$ (because $u(y) \in B_{\epsi}$,
by the isometricity of $u$). So, there is $\sigma_0 \in \Sigma$ such that $v_{\epsi}(u_{\sigma}(x))
> 0$ for any $\sigma \geqsl \sigma_0$. This means that for $\sigma \geqsl \sigma_0$, $u_{\sigma}(x)
\notin B_{\epsi}$ and consequently $d(u_{\sigma}(x),u(x)) < \epsi$.
\end{proof}

The homomorphism appearing in \eqref{eqn:AE} is \textbf{not} surjective, unless $\card(X) < 3$.
There is however a fascinating result discovered by Mayer-Wolf \cite{m-w} which characterizes
\textit{all} isometries of the space $\ArEe(X)$ under some additional conditions on the metric
of $X$. Below we formulate only a special case of it, enough for our considerations.

\begin{thm}{AE}
Let $d$ be a \textbf{bounded} complete metric on a set $X$. Let $\ArEe(X)$ denote the Arens-Eells
spaces of $(X,\sqrt{d})$. Every linear isometry of $\ArEe(X)$ onto itself is of the form
$\pm \ArEe(u)$ where $u \in \Iso(X,d) (= \Iso(X,\sqrt{d}))$.
\end{thm}
\begin{proof}
All parts of this proof come from \cite{m-w}. Alternatively, we give references to suitable results
from the book \cite{wea}. By \cite[Proposition~2.4.5]{wea}, the metric space $(X,\sqrt{d})$ is
so-called concave (for the definition, see the note preceding Lemma~2.4.4 in \cite{wea}, page~51
there). Now \cite[Theorem~2.7.2]{wea} implies that if $\Phi$ is a linear isometry of $\ArEe(X)$,
then there is $r \in \RRR \setminus \{0\}$ and a bijection $u\dd X \to X$ such that
$\sqrt{d(u(x),u(y))} = |r| \sqrt{d(x,y)}$ and $\Phi(\chi_x - \chi_y) = \frac1r (\chi_{u(x)}
- \chi_{u(y)})$ for any $x, y \in X$. The former of these relations, combined with the boundedness
of $d$, implies that $|r| = 1$ and $u \in \Iso(X,d)$. So, $\Phi = \pm \ArEe(u)$ and we are done.
\end{proof}

We shall also need quite an intuitive result stated below. Although its proof is not immediate,
we leave it to the reader as an exercise.

\begin{lem}{R2}
Let $X$ be a two-dimensional real vector space, $\|\cdot\|$ be any norm on $X$ and let $a$ and $b$
be two vectors in $X$.
\begin{enumerate}[\upshape(a)]
\item If $\bar{B}_X(0,2) \subset \bar{B}_X(b,2) \cup \bar{B}_X(a,1)$, then $b = 0$.
\item If $\|a\| = \|b\| = 2$ and $\bar{B}_X(b,1) \subset \bar{B}_X(0,2) \cup \bar{B}_X(a,1)$, then
   $a = b$.
\end{enumerate}
\end{lem}

Finally, we are ready to give

\begin{proof}[Proof of \THM{Banach} and point \textup{(}b\textup{)} of \PRO{non}]
For \THM{Banach} is a special case of point (b) of the proposition, we focus only on the proof
of the latter result. It follows from \COR{weight} that there is a complete metric space
$(Y,\varrho)$ such that $w(Y) = \beta$ and $\Iso(Y,\varrho)$ is isomorphic to $G$. Since
$\Iso(Y,\varrho) = \Iso(Y,\frac{\varrho}{2+2\varrho})$ and the metric $\frac{\varrho}{2+2\varrho}$
is complete (and compatible), we may and do assume that $\varrho < \frac12$. We also assume that
$Y \cap [0,1] = \varempty$. Let $X = Y \cup [0,1]$. We define a metric $d$ on $X$ by the rules:
\begin{itemize}
\item $d(s,t) = |s - t|$ for $s, t \in [0,1]$,
\item $d(x,y) = \varrho(x,y)$ for $x, y \in Y$,
\item $d(x,t) = d(t,x) = 1 + t$ for $x \in Y$ and $t \in [0,1]$.
\end{itemize}
We leave it as an exercise that $d$ is indeed a metric, that $d$ is complete and $w(X) = \beta$.
Notice that for any $a \in X$ and $t \in [0,1] (\subset X)$:
\begin{itemize}
\item $a = 1 \iff \exists b,c \in X\dd\ d(a,b) = \frac34\ \wedge\ d(a,c) = 2$,
\item $a = t \iff d(a,1) = 1-t$.
\end{itemize}
The above notices imply that for every $f \in \Iso(X,d)$ we have $f(t) = t$ for $t \in [0,1]$ and
$f\bigr|_Y \in \Iso(Y,\varrho)$. It is also easy to see that each isometry of $(Y,\varrho)$ extends
(uniquely) to an isometry of $(X,d)$. Hence the function $\Iso(X,d) \ni f \mapsto f\bigr|_Y \in
\Iso(Y,\varrho)$ is a (well defined) isomorphism. Now let $(E,\|\cdot\|) =
(\ArEe(X),\|\cdot\|_{AE})$ be the Arens-Eells space of $(X,\sqrt{d})$ and let $e = \chi_1 - \chi_0$.
We see that $w(E) = \beta$ and $E$ is infinite-dimensional, since $X$ is infinite. What is more,
it follows from \THM{AE} that every linear isometry of $E$ which leaves the point $e$ fixed is
of the form $\ArEe(u)$ for some $u \in \Iso(X,\sqrt{d})$. Since $\ArEe(u)(e) = e$ for any $u \in
\Iso(X,d)$ (because $u(0) = 0$ and $u(1) = 1$ for such $u$), the notice that $\Iso(X,d) =
\Iso(X,\sqrt{d})$ and \LEM{AE} finish the proof.
\end{proof}

Our next aim is to show points (a) of both \THM{iso} and \PRO{non}. In the proof we shall involve
the next three results. The first of them is due to Mankiewicz \cite{man}:

\begin{thm}{man}
Whenever $X$ and $Y$ are normed vector spaces, $U$ and $V$ are connected open subsets of,
respectively, $X$ and $Y$, then every isometry of $U$ onto $V$ extends to a unique affine isometry
of $X$ onto $Y$.
\end{thm}

We recall that a function $\Phi\dd X \to Y$ between real vector spaces $X$ and $Y$ is
\textit{affine} if $\Phi - \Phi(0)$ is linear.\par
The following result is a consequence e.g. of \cite[Theorem~VI.6.2]{b-p} and a famous theorem
of Toru\'{n}czyk \cite{tr1,tr2} which says that every Banach space is homeomorphic to a Hilbert
space.

\begin{thm}{H1}
Every closed convex set in an infinite-dimensional Banach space whose interior is nonempty is
homeomorphic to an infinite-dimensional Hilbert space.
\end{thm}

Our last tool is the next result which in the separable case was proved by Mogilski \cite{mog}.
The argument presented by him works also in nonseparable case. This theorem in its full generality
may also be briefly concluded from the results of Toru\'{n}czyk \cite{tr1,tr2}. For reader's
convenience, we give a sketch of its proof.

\begin{thm}{H2}
Let $X$ be a metrizable space. If $X$ is the union of its two closed subsets $A$ and $B$ such that
each of $A$, $B$ and $A \cap B$ is homeomorphic to an infinite-dimensional Hilbert space $\HHh$,
then $X$ itself is homeomorphic to $\HHh$.
\end{thm}
\begin{proof}
Put $C = A \cap B$.\par
First assume that $C$ is a $Z$-set in both $A$ and $B$ (for the definition of a $Z$-set see
Section~6). Then there exist homeomorphisms $h_A\dd A \to \HHh \times (-\infty,0]$ and $h_B\dd B \to
\HHh \times [0,\infty)$ which coincide on $C$ and $h_A(C) = h_B(C) = \HHh \times \{0\}$ (this
follows from the theorem on extending homeomorphisms between $Z$-sets in Hilbert manifolds, see
\cite{and}, \cite{a-m}, \cite{ch1} or Chapter~V in \cite{b-p}). Now it suffices to define $h\dd X
\to \HHh \times \RRR$ as the union of $h_A$ and $h_B$ to obtain a homeomorphism we searched for.\par
Now we consider a general case. Let $X' = (A \times [-1,0]) \cup (B \times [0,1]) \subset X \times
[-1,1]$. Observe that $A' = A \times [-1,0]$, $B' = B \times [0,1]$ and $C' = C \times \{0\}$ are
homeomorphic to $\HHh$ (by the assumptions of the theorem) and $C'$ is a $Z$-set in both $A'$ and
$B'$. Thus, we infer from the first part of the proof that $X'$ is homeomorphic to $\HHh$. Finally,
note that the function $X' \ni (x,t) \mapsto (x,0) \in X \times \{0\}$ is a proper retraction. So,
Toru\'{n}czyk's result \cite{tr1} implies that $X$ is a manifold modelled on $\HHh$. Since it is
contractible, it is homeomorphic to $\HHh$.
\end{proof}

We are now ready to give

\begin{proof}[Proof of points \textup{(}a\textup{)} of \THM{iso} and \PRO{non}]
Again, observe that point (a) of the theorem under the question is a special case of point (b)
of the proposition. Therefore we focus only on the latter result. Let $E$ and $e$ be as in point (b)
of \PRO{non}. Replacing, if needed, $e$ by $2e/\|e\|$, we may assume that $\|e\| = 2$. Denote
by $\EEe$ the group of all linear isometries which leave $e$ fixed. Let $W = \bar{B}_E(0,2) \cup
\bar{B}_E(e,1)$ be equipped with the metric $p$ induced by the norm of $E$. Notice that if $V \in
\EEe$, then $V(W) = W$ and $V\bigr|_W \in \Iso(W,p)$. Conversely, for each $g \in \Iso(W,p)$, $g =
V\bigr|_W$ for some linear isometry $V \in \EEe$. Let us briefly justify this claim. Let $x = g(0)$
and $y = g(e)$. Then $W \subset \bar{B}_E(x,2) \cup \bar{B}_E(y,1)$ and consequently $\bar{B}_X(0,2)
\subset \bar{B}_X(x,2) \cup \bar{B}_X(y,1)$ where $X$ is a two-dimensional linear subspace of $E$
which contains $x$ and $y$. We infer from point (a) of \LEM{R2} that $x = 0$. So, $\|y\| = 2$ (since
$g$ is an isometry) and thus $\bar{B}_Y(e,1) \subset \bar{B}_Y(0,2) \cup \bar{B}_Y(y,1)$ where $Y$
is a two-dimensional linear subspace of $E$ such that $e, y \in Y$. Now point (b) of \LEM{R2} yields
that $y = e$. We then have $g(B_E(0,2) \cup B_E(e,1)) = B_E(0,2) \cup B_E(e,1)$. So, an application
of \THM{man} gives our assertion: there is a linear (since $g(0) = 0$) isometry $V$ of $E$ which
extends $g$.\par
Having the above fact, we easily see that the function $\EEe \ni V \mapsto V\bigr|_W \in \Iso(W,p)$
is an isomorphism. Consequently, $G$ is isomorphic to $\Iso(W,p)$. So, to finish the proof,
it suffices to show that $W$ is homeomorphic to $\HHh_{\beta}$. But this immediately follows from
Theorems~\ref{thm:H1} and \ref{thm:H2}, since $w(E) = \beta$ and the sets $\bar{B}_E(0,2)$,
$\bar{B}_E(e,1)$ and $\bar{B}_E(0,2) \cap \bar{B}_E(1,2)$ are closed, convex and have nonempty
interiors.
\end{proof}

\begin{proof}[Proof of \COR{beta}]
It suffices to apply \PRO{non} and point (f) of \PRO{exs}.
\end{proof}

The arguments used in the proofs of both points of \PRO{non} may simply be employed to show
the following

\begin{cor}{incomplete}
Let $G$ be a $\GGg_{\delta}$-complete topological group of topological weight not exceeding $\beta
\geqsl \aleph_0$. There are an infinite-dimensional normed vector space $E$ of topological weight
$\beta$, a contractible open set $U \subset E$ and a nonzero vector $e \in E$ such that
the topological groups $G$, $\Iso(U,d)$ and $\Iso(E|e)$ are isomorphic where $d$ is the metric
on $U$ induced by the norm of $E$ and $\Iso(E|e)$ is the group of all linear isometries of $E$ which
leave the point $e$ fixed.
\end{cor}
\begin{proof}
Let $(Y_0,\varrho_0)$ be a metric space such that $w(Y_0) = \beta$ and $\Iso(Y_0,\varrho_0)$ is
isomorphic to $G$ (cf. Theorems~\ref{thm:G-complete} and \ref{thm:closed}). Denote by $(Y,\varrho)$
the completion of $(Y_0,\varrho_0)$. Now let $(X,d)$, $(\ArEe(X),\|\cdot\|_{AE})$ and $e$ be
as in the proof of point (b) of \PRO{non}. We know that the function $$\Iso(X,d) \ni u \mapsto
\ArEe(u) \in \Iso(\ArEe(X)|e)$$ is an isomorphism. Denote by $E$ the linear span of the set
$\{\chi_a - \chi_b\dd\ a, b \in Y_0 \cup [0,1]\} (\subset \ArEe(X))$ (recall that $X = Y \sqcup
[0,1]$). Observe that if $u \in \Iso(X,d)$ is such that $\ArEe(u)(E) = E$, then $u(Y_0 \cup [0,1]) =
Y_0 \cup [0,1]$ and consequently $u(Y_0) = Y_0$ (see the proof of point (b) of \PRO{non}). This
yields that also the function
$$\Iso(Y_0,\varrho_0) \ni u \mapsto \ArEe(u)\bigr|_E \in \Iso(E|e)$$
is an isomorphism. Now it suffices to put $U = B_E(0,\|e\|) \cup B_E(e,\frac12\|e\|)$ and repeat
the proof of point (a) of \PRO{non} (involving \LEM{R2} and \THM{man}) to get the whole assertion.
($U$ is contractible as the union of two intersecting convex sets.)
\end{proof}

\SECT{Isometry groups of completely metrizable metric spaces}

Taking into account \COR{beta}, the following question naturally arises: given an infinite cardinal
$\beta$, how to characterize topological groups isomorphic to $\Iso(\HHh,d)$ for some compatible
metric $d$ on a Hilbert space $\HHh$ of Hilbert space dimension $\beta$~? In this part we give
a partial answer to this question. In fact, we will deduce our main result in this topic from
the following general fact.

\begin{pro}{completion-H}
Let $(S,p)$ be a bounded complete metric space, $\beta$ an infinite cardinal not less than $w(S)$
and $\HHh$ a Hilbert space of Hilbert space dimension $\beta$. Let $S'$ be a dense subset of $S$ and
$G$ a closed subgroup of $\Iso(S',p)$. There exist a compatible complete metric $\lambda$ on $\HHh$,
a set $\HHh' \subset \HHh$ and an isomorphism $\Psi\dd \overline{G} \to \Iso(\HHh,\lambda)$ such
that $\Psi(G) = \{u \in \Iso(\HHh,\lambda)\dd\ u(\HHh') = \HHh'\}$, $(\HHh \setminus \HHh',\lambda)$
is isometric to $(S \setminus S',\sqrt{p})$ and the closure of $\HHh \setminus \HHh'$ is a $Z$-set
in $\HHh$. In particular, if $S'$ is completely metrizable, then $\HHh'$ is homeomorphic to $\HHh$.
\end{pro}

Recall that a closed set $K$ in a metric space $X$ is a \textit{$Z$-set} iff every map
of the Hilbert cube $Q$ into $X$ may uniformly be approximated by maps of $Q$ into $X \setminus K$.

\begin{proof}
By \PRO{closed2}, there is a bounded complete metric space $(Y,\varrho)$, a dense set $Y' \subset Y$
and an isomorphism $F_1\dd \overline{G} \to \Iso(Y,\varrho)$ such that $w(Y) = \beta$,
$(Y \setminus Y',\varrho)$ is isometric to $(S \setminus S',p)$ and $F_1(G) = \{u \in
\Iso(Y,\varrho)\dd\ u(Y') = Y'\}$. Now we shall mimic the proof of \PRO{non}.\par
Replacing, if applicable, $p$ and $\varrho$ by $t \cdot p$ and $t \cdot \varrho$ with small enough
$t > 0$ (and the final metric $\lambda$ obtained from this proof by $\frac{1}{\sqrt{t}} \lambda$),
we may assume that $\varrho < \frac12$. Now let $(X,d) \supset (Y,\varrho)$ be as in the proof
of \THM{Banach}. Further, let $(E,\|\cdot\|) = (\ArEe(X),\|\cdot\|_{AE})$ be the Arens-Eells space
of $(X,\sqrt{d})$ and $e = \chi_1 - \chi_0 \in E$. We denote by $\lambda$ the metric on $W =
\bar{B}_E(0,1) \cup \bar{B}_E(e,\frac12)$ induced by the norm of $E$. Arguments involved
in the proofs of \THM{Banach} and \PRO{non} show that
\begin{enumerate}[({I}1)]
\item the function $F_2\dd \Iso(X,d) \ni u \mapsto u\bigr|_Y \in \Iso(Y,\varrho)$ is an isomorphism
   and there is a dense set $X' \subset X$ such that $X \setminus X' = Y \setminus Y'$ and
   $F_2^{-1}(F_1(G))$ constists precisely of all $u \in \Iso(X,d)$ with $u(X') = X'$;
\item the function $F_3\dd \Iso(X,d) \ni u \mapsto \ArEe(u)\bigr|_W \in \Iso(W,\lambda)$ is
   an isomorphism and $u(0) = 0$ for any $u \in \Iso(X,d)$;
\item $W$ is homeomorphic to $\HHh$.
\end{enumerate}
Point (I3) asserts that we may identify $\HHh$ with $W$. Put $\Psi = F_3 \circ F_2^{-1} \circ F_1\dd
\overline{G} \to \Iso(W,\lambda)$ and $W' = W \setminus \{\chi_x - \chi_0\dd\ x \in X \setminus
X'\}$ (recall that $\|\chi_y - \chi_0\|_{AE} = \sqrt{d(y,0)} = 1$ for every $y \in Y \supset X
\setminus X'$) and note that $\Psi$ is an isomorphism. What is more, we claim that $\Psi(G)$
consists of all $v \in \Iso(W,\lambda)$ for which $v(W') = W'$. Indeed, it follows from (I2) that
each $v \in \Iso(W,\lambda)$ has the form $v = \ArEe(u)\bigr|_W$ for some $u \in \Iso(X,d)$. So,
taking into account (I1), we only need to check that $u(X') = X'$ iff $(\ArEe(u))(W') = W'$. But
this follows from the fact that $u(0) = 0$.\par
Further, observe that the function $(X \setminus X',\sqrt{d}) \ni x \mapsto \chi_x - \chi_0 \in
(W \setminus W',\lambda)$ is an isometry, which implies that the latter metric space is isometric
to $(S \setminus S',\sqrt{p})$. To show that the closure $D$ of $W \setminus W'$ is a $Z$-set
in $W$, note that $D \subset C := \{\chi_y - \chi_0\dd\ y \in Y\}$ and the maps $W \ni w \mapsto
(1 - \frac1n)w + \frac1n e \in W$ converge uniformly to the identity map with respect to $\lambda$
and their images are disjoint from $C$.\par
So, to complete the proof, we only need to check that if $S'$ is completely metrizable, then $W'$ is
homeomorphic to $\HHh$. But this is an immediate consequence of the fact that $W \setminus W'$ is
a set isometric to $(S \setminus S',\sqrt{p})$ and contained in a $Z$-set in $W$, and a known fact
on $\sigma$-$Z$-sets in Hilbert spaces: $S'$, being completely metrizable, is a $\GGg_{\delta}$-set
in $S$, hence $S \setminus S'$ is $\FFf_{\sigma}$ in $S$. Consequently, $W \setminus W'$ is
a countable union of sets complete in the metric $\lambda$, thus an $\FFf_{\sigma}$-set in $W$.
Since the closure of $W \setminus W'$ is a $Z$-set in $W$, thus $W \setminus W'$ is
a $\sigma$-$Z$-set, that is, it is a countable union of $Z$-sets. Now the assertion follows from
the well-known result that the complement of a $\sigma$-$Z$-set in a Hilbert space is homeomorphic
to the whole space, which simply follows from Toru\'{n}czyk's characterization of Hilbert manifolds
\cite{tr1,tr2} (for the separable case one may also consult Theorem~6.3 in \cite[Chapter~V]{b-p}).
\end{proof}

As a conclusion, we obtain

\begin{thm}{incomplete}
Let $\HHh$ be a Hilbert space of Hilbert space dimension $\beta \geqsl \aleph_0$ and
$$\ggG = \{\Iso(\HHh,\varrho)\dd\ \varrho \textup{ is a compatible metric on } \HHh\}.$$
Then, up to isomorphism, $\ggG$ consists precisely of all topological groups $G$ which are
isomorphic to closed subgroups of $\Iso(X,d)$ for some completely metrizable spaces $(X,d)$ with
$w(X) \leqsl \beta$.
\end{thm}
\begin{proof}
If $G$ is a closed subgroup of $\Iso(X,d)$ for a completely metrizable space $(X,d)$ with $w(X)
\leqsl \beta$, we may assume that $d$ is bounded (replacing, if necessary, $d$ by $\frac{d}{1+d}$).
Then \PRO{completion-H}, applied for $(S,p) =$ the completion of $(X,d)$ and $S' = X$, yields
a complete compatible metric $\lambda$ on $\HHh$ and a dense set $\HHh' \subset \HHh$ homeomorphic
to $\HHh$ such that $G$ is isomorphic to $G' := \{u \in \Iso(\HHh,\lambda)\dd\ u(\HHh') = \HHh'\}$.
But $G'$ is naturally isomorphic to $\Iso(\HHh',\lambda)$ and we are done.
\end{proof}

For an infinite cardinal number $\alpha$, let us denote by $\IGH(\alpha)$ the class of all
topological groups which are isomorphic to $\Iso(\HHh,\varrho)$ for some compatible metric $\varrho$
on a Hilbert space $\HHh$ of Hilbert space dimension $\alpha$ (`IGH' is the abbreviation
of `isometry group of a Hilbert space'). Additionally, let $\IGH$ stand for the union of all classes
$\IGH(\alpha)$'s. \THM{incomplete} implies that $\IGH$ is a \textit{variety}; that is, closed
subgroups as well as topological products of members of $\IGH$ belong to $\IGH$ as well.\par
We are mainly interested in the class $\IGH(\aleph_0)$. It is clear that all groups belonging
to this class are second countable. In the sequel we shall see that the axiom of second countability
is insufficient for a topological group to belong to $\IGH(\aleph_0)$ (see \PRO{real} below).\par
As a simple consequence of \THM{incomplete} we obtain

\begin{cor}{sigma-comp}
If $G$ is a second countable, $\sigma$-compact topological group, then $G \in \IGH(\aleph_0)$.
\end{cor}
\begin{proof}
Let $p$ be a compatible left invariant metric on $G$ and let $(Y,\varrho) \supset (G,p)$ denote
the completion of $(G,p)$. If the interior of $G$ in $Y$ is nonempty, then $G$ is locally completely
metrizable and thus $G$ is Polish. In that case the assertion follows from \THM{iso}. On the other
hand, if $G$ is a boundary set in $Y$, then $Y \setminus G$ is dense in $Y$ and therefore
$\Iso(Y \setminus G,\varrho)$ is isomorphic to $\{u \in \Iso(Y,\varrho)\dd\ u(Y \setminus G) = Y
\setminus G\}$ and the latter group coincides with $\{u \in \Iso(Y,\varrho)\dd\ u(G) = G\}$ which
is isomorphic to $\Iso(G,\varrho) = \Iso(G,p)$. Since $G$ is $\sigma$-compact, $Y \setminus G$ is
completely metrizable. Finally, since $p$ is left invariant, all left translations of $G$ form
a closed subgroup of $\Iso(G,p)$ isomorphic to $G$. So, to sum up, $G$ is isomorphic to a closed
subgroup of the isometry group of the Polish space $(Y \setminus G,\varrho)$. Now it suffices
to apply \THM{incomplete}.
\end{proof}

To formulate our next result, we remind the reader that a separable metrizable space $X$ is said
to be \textit{coanalytic} iff $X$ is homeomorphic to a space of the form $Y \setminus Z$ where $Y$
is a Polish space and $Z \subset Y$ is \textit{analytic}, that is, $Z$ is the continuous image
of a Polish metric space. We also recall that continuous images of Borel subsets of Polish spaces
are analytic.

\begin{pro}{coanal}
Each member of $\IGH(\aleph_0)$ is coanalytic as a topological space.
\end{pro}
\begin{proof}
Let $(X,\varrho)$ be a Polish metric space and $G = \Iso(X,\varrho)$. Denote by $(Y,d)$
the completion of $(X,\varrho)$. Since $X$ is completely metrizable, it is a $\GGg_{\delta}$-set
in $Y$. Observe that $G$ is naturally isomorphic to the subgroup $\{u \in \Iso(Y,d)\dd\ u(X) = X\}$
of $\Iso(Y,d)$. Since $\Iso(Y,d)$ is a Polish group, it suffices to show that $A := \{u \in
\Iso(Y,d)\dd\ u(X) \neq X\}$ is the continuous image of a Borel subset of a Polish metric space.
To this end, notice that the set $W := \{(u,x) \in \Iso(Y,d) \times X\dd\ u(x) \in Y \setminus X\}$
is Borel in the Polish space $\Iso(Y,d) \times X$ and $\pi(W) = A$ where $\pi\dd \Iso(Y,d) \times X
\to \Iso(Y,d)$ is the natural projection.
\end{proof}

\begin{pro}{real}
There exists a topological subgroup of the additive group of reals which does not belong
to $\IGH(\aleph_0)$.
\end{pro}
\begin{proof}
We consider $\RRR$ as a vector space over the field $\QQQ$ of rationals. There exists a vector
subspace $G$ of $\RRR$ such that $\QQQ \cap G = \{0\}$ and $G + \QQQ = \RRR$. We claim that
$G \notin \IGH(\aleph_0)$. To convince of that, it is enough to prove that $G$ is not coanalytic
(thanks to \PRO{coanal}). Since analytic spaces are absolutely measurable (see e.g. Theorem~A.13
in \cite[Appendix]{tak}), it suffices to show that $G$ is not Lebesgue measurable. But this follows
from the following two observations:
\begin{itemize}
\item $\RRR = \bigcup_{q \in \QQQ} (q + G)$ and hence the outer measure of $G$ is positive;
\item $G - G (= G)$ has empty interior and thus its inner measure is $0$, which simply follows from
   Theorem~A in \S61 of \cite[Chapter~XII]{hal} (page~266).
\end{itemize}
\end{proof}

Propositions~\ref{pro:coanal} and \ref{pro:real} make the issue of characterizing members of $\IGH$
complicated. The following problems seem to be most interesting:

\begin{prb}{1}
If $G \in \IGH$, is it true that $G \in \IGH(w(G))$~?
\end{prb}

\begin{prb}{2}
Does the class $\IGH$ contain all $\GGg_{\delta}$-complete topological groups?
\end{prb}

\begin{prb}{3}
Characterize members of $\IGH(\aleph_0)$.
\end{prb}

\SECT{Compact and locally compact Polish groups}

This section is devoted to the proofs of points (b) and (c) of \THM{iso}. Our main tool will be
the following result, very recently shown by us \cite{pn0}.

\begin{thm}{prop}
Let $G$ be a locally compact Polish group, $X$ be a locally compact Polish space and let $G \times X
\ni (g,x) \mapsto g.x \in X$ be a continuous proper action of $G$ on $X$. Assume there is a point
$\omega \in X$ such that the set $G.\omega = \{g.\omega\dd\ g \in G\}$ is non-open and $G$ acts
freely at $\omega$ (that is, $g.\omega = \omega$ implies $g =$ the neutral element of $G$). Then
there exists a proper compatible metric $d$ on $X$ such that $\Iso(X,d)$ consists precisely of all
maps of the form $x \mapsto g.x\ (g \in G)$. In particular, the topological groups $\Iso(X,d)$ and
$G$ are isomorphic.
\end{thm}

We recall that (under the above notation) the action is \textit{proper} if for every compact set
$K \subset X$ the set $\{g \in G\dd\ g.K \cap K \neq \varempty\}$ is compact as well (where $g.K =
\{g.x\dd\ x \in K\}$).\par
Our next tool is the following classical result due to Keller \cite{kel} (see also
\cite[Theorem~III.3.1]{b-p}).

\begin{thm}{Q}
Every infinite-dimensional compact convex subset of a Fr\'{e}chet space is homeomorphic
to the Hilbert cube.
\end{thm}

We recall that a Fr\'{e}chet space is a completely metrizable locally convex topological vector
space.\par
We call a function $u\dd (X,d) \to \RRR$ (where $(X,d)$ is a metric space) \textit{nonexpansive} iff
$|u(x) - u(y)| \leqsl d(x,y)$ for all $x, y \in X$. The function $u$ is a \textit{Kat\v{e}tov map}
iff $u$ is nonexpansive and additionally $d(x,y) \leqsl u(x) + u(y)$ for any $x, y \in X$.
Kat\v{e}tov maps correspond to one-point extensions of metric spaces.\par
We are now ready for proving point (b) of \THM{iso}.

\begin{proof}[Proof of point \textup{(}b\textup{)} of \THM{iso}]
Let $G$ be a compact Polish group. Take a left invariant metric $\varrho \leqsl 1$ on $G$ and equip
the space $X = G \times [0,1]$ with the metric $d$ where $d((x,s),(y,t)) =
\max(\varrho(x,y),|t - s|)$. For $g \in G$ denote by $\psi_g$ the function $X \ni (x,t) \mapsto
(g^{-1}x,t) \in X$. Notice that $\psi_g \in \Iso(G,X)$ for any $g \in G$. Let $\Delta$ be the space
of all nonexpansive maps of $(X,d)$ into $[0,1]$ endowed with the supremum metric. Observe that
$\Delta$ is a convex set in the Banach space of all real-valued maps on $X$. What is more, $\Delta$
is infinite-dimensional, since $X$ is infinite, and $\Delta$ is compact, by the Ascoli type theorem.
So, we infer from \THM{Q} that $\Delta$ is homeomorphic to $Q$. Further, $\Phi_g(u) := u \circ
\psi_g \in \Delta$ for any $g \in G$ and $u \in \Delta$ (because $\psi_g$ is isometric). It is also
easily seen that the function $G \times \Delta \ni (g,u) \mapsto \Phi_g(u) \in \Delta$ is
a (proper---since both $G$ and $\Delta$ are compact) continuous action of $G$ on $\Delta$. Finally,
the function $\omega\dd X \ni (x,t) \mapsto d((x,t),(e,1)) \in \RRR$ belongs to $\Delta$ (since $d
\leqsl 1$) where $e$ is the neutral element of $G$. Observe that the set $K := \{\omega \circ
\psi_g\dd\ g \in G\}$ has empty interior in $\Delta$, since $\frac1n + (1 - \frac1n)\omega \circ
\psi_g \in \Delta \setminus K$ for any $n \geqsl 1$. Now an application of \THM{prop} finishes
the proof.
\end{proof}

Our last aim is to prove point (c) of \THM{iso}. To this end, we need more information on Hilbert
cube manifolds.\par
One of the deepest results in infinite-dimensional topology is Anderson's theorem on extending
homeomorphisms between $Z$-sets \cite{and}. Below we formulate it only in the Hilbert cube settings,
it holds however in a much more general context. (For the discussion on this topic consult
\cite[Chapter~V]{b-p}; see also \cite{a-m} and \cite{ch1}).

\begin{thm}{Z}
Every homeomorphism between two $Z$-sets in the Hilbert cube $Q$ is extendable to a homeomorphism
of $Q$ onto itself.
\end{thm}

The result stated below is a kind of folklore in Hilbert cube manifolds theory. We present its short
proof, because we could not find it in the literature.

\begin{thm}{Q*}
The spaces $Q \times [0,\infty)$ and $Q \setminus \{\textup{point}\}$ are homeomorphic.
\end{thm}
\begin{proof}
Since $Q \setminus \{\textup{point}\}$ is a Hilbert cube manifold, it follows from Schori's theorem
\cite{sch} (see also \cite{ch2}; compare with \cite[Theorem~IX.4.1]{b-p}) that $(Q \setminus
\{\textup{point}\}) \times Q$ is homeomorphic to $Q \setminus \{\textup{point}\}$. Now the assertion
follows from \THM{Z} since $(Q \times [0,1]) \setminus (Q \times [0,1)) = Q \times \{1\}$ is
a $Z$-set in $Q \times [0,1]$ homeomorphic to the $Z$-set (in $Q \times Q$) $(Q \times Q) \setminus
[(Q \setminus \{\textup{point}\}) \times Q]$.
\end{proof}

\begin{lem}{Q*}
Let $(X,d)$ be a nonempty separable metric space and let $E(X)$ be the set of all Kat\v{e}tov maps
on $(X,d)$ equipped with the pointwise convergence topology.
\begin{enumerate}[\upshape(i)]
\item For any $a \in X$ and $r > 0$ the set $\{f \in E(X)\dd\ f(a) \leqsl r\}$ is compact
   (in $E(X)$).
\item $E(X) \times Q$ is homeomorphic to $Q \setminus \{\textup{point}\}$.
\end{enumerate}
\end{lem}
\begin{proof}
Point (i) follows from the Ascoli type theorem, since $E(X)$ consists of nonexpansive maps and for
any $f \in E(X)$ and $x \in X$, $f(x) \in [0,d(x,a) + f(a)]$.\par
We turn to (ii). First of all, $E(X)$ is metrizable, because of the separability of $X$ and
the nonexpansivity of members of $E(X)$. Further, thanks to \THM{Q*}, it suffices to show that $E(X)
\times Q$ is homeomorphic to $Q \times [0,\infty)$. Fix $a \in X$ and let $\omega \in E(X)$ be given
by $\omega(x) = d(a,x)$. For each $n \geqsl 1$ let $K_n = \{f \in E(X)\dd\ f(a) \in [n-1,n]\}$ and
$Z_{n-1} = \{f \in E(X)\dd\ f(a) = n-1\}$. We infer from (i) that $K_n$ and $Z_{n-1}$ are compact.
It is also easily seen that both they are convex nonempty sets ($\omega + n-1 \in Z_{n-1} \subset
K_n$). Since $K_n \times Q$ and $Z_{n-1} \times Q$ are affinely homeomorphic to convex subsets
of Fr\'{e}chet spaces, \THM{Q} yields that both these sets are homeomorphic to $Q$. Let $h_{n-1}\dd
Z_{n-1} \times Q \to Q \times \{n-1\}$ be any homeomorphism. We claim that $Z_{n-1} \cup Z_n$ is
a $Z$-set in $K_n$. This easily follows from the fact that the maps $K_n \ni f \mapsto (1-\frac1k)f
+ \frac1k (\omega + n-\frac12) \in K_n$ send $K_n$ into $K_n \setminus (Z_{n-1} \cup Z_n)$ and
converge uniformly (as $k\to\infty$) to the identity map of $K_n$. Since $Q \times \{n-1,n\}$ is
a $Z$-set in $Q \times [n-1,n]$, \THM{Z} provides us the existence of a homeomorphism $H_n\dd K_n
\times Q \to Q \times [n-1,n]$ which extends both $h_{n-1}$ and $h_n$. We claim that the union $H\dd
E(X) \times Q \to Q \times [0,\infty)$ of all $H_n$'s ($n \geqsl 1$) is the homeomorphism we search.
It is clear that $H$ is a well defined bijection. Finally, notice that the interiors (in $E(X)$)
of the sets $\bigcup_{j=1}^n K_j\ (n \geqsl 1)$ cover $X$ and hence $H$ is indeed a homeomorphism.
\end{proof}

\begin{proof}[Proof of point \textup{(}c\textup{)} of \THM{iso}]
Let $G$ be a locally compact Polish group. By a theorem of Struble \cite{str} (see also \cite{amn}),
there exists a proper left invariant compatible metric $d$ on $G$. Let $E(G)$ be the space of all
Kat\v{e}tov maps on $(G,d)$ endowed with the pointwise convergence topology. By \LEM{Q*}, $L := E(G)
\times Q$ is homeomorphic to $Q \setminus \{\textup{point}\}$. So, it suffices to show that there is
a proper compatible metric $\varrho$ on $L$ such that $\Iso(L,\varrho)$ is isomorphic to $G$. For
any $g \in G$ and $(f,q) \in L$ let $g.(f,q) = (f_g,q) \in L$ where $f_g(x) = f(g^{-1}x)$ (since $d$
is left invariant, $f_g \in E(G)$ for each $f \in E(G)$). As in the proof of point (b)
of the theorem, we see that the function $G \times L \ni (g,x) \mapsto g.x \in L$ is a continuous
action of $G$ on $L$. It is also clear that each $G$-orbit (i.e. each of the sets $G.x$ with $x \in
L$) has empty interior. Similarly as in point (b) we show that there is $\omega \in L$ such that $G$
acts freely at $\omega$ (for example, $\omega = (u,q)$ with arbitrary $q \in Q$ and $u(x) = d(x,e)$
where $e$ is the neutral element of $G$). So, by virtue of \THM{prop}, it remains to check that
the action is proper. To this end, take any compact set $W$ in $L$. Then there is $r > 0$ such that
$W \subset \{f \in E(G)\dd\ f(e) \leqsl r\} \times Q$. Note that the set $\{g \in G\dd\ g.W \cap W
\neq \varempty\}$ is closed and contained in $D \times Q$ where
$$D = \{g \in G|\quad \exists f \in E(G)\dd\ f(e) \leqsl r\ \wedge\ f(g^{-1}) \leqsl r\}$$
and therefore it is enough to show that $D$ has compact closure in $G$. But if $g \in D$, and $f \in
E(G)$ is such that $f(e) \leqsl r$ and $f(g^{-1}) \leqsl r$, then $d(g,e) = d(e,g^{-1}) \leqsl f(e)
+ f(g^{-1}) \leqsl 2r$. This yields that $D \subset \bar{B}_G(e,2r)$ and the note that $d$ is proper
finishes the proof.
\end{proof}

\begin{rem}{Q*}
Van Dantzig and van der Waerden \cite{d-w} proved that the isometry group of a connected locally
compact metric space $(X,d)$ (possibly with non-proper or incomplete metric) is locally compact and
acts properly on $X$. It follows from our result \cite{pn0} that then there exists a proper
compatible metric $\varrho$ on $X$ such that $\Iso(X,d) = \Iso(X,\varrho)$. In particular,
\begin{multline*}
\{\Iso(Q \setminus \{\textup{point}\},d)\dd\ d \textup{ is a compatible metric}\} =\\=
\{\Iso(Q \setminus \{\textup{point}\},\varrho)\dd\ \varrho \textup{ is a proper compatible metric}\}
\end{multline*}
and hence if we omit the word \textit{proper} in point (c) of \THM{iso}, we will obtain
an equivalent statement.
\end{rem}

As we mentioned in the introductory part, each [locally] compact finite\hyp{}dimensional Polish
group is isomorphic to the isometry group of a [proper locally] compact finite-dimensional metric
space. Taking this, and \COR{0-dim}, into account, the following question may be interesting.

\begin{prb}{fd}
Is every finite-dimensional metrizable (resp. finite-dimensional Polish) group isomorphic
to the isometry group of a finite-dimensional (resp. finite-dimensional separable complete) metric
space?
\end{prb}

\end{document}